\documentclass[12pt,psamsfonts]{amsart}

\usepackage{amsmath,amssymb,amsthm,amsfonts,mathrsfs}
\usepackage{mathtools}
\usepackage{hyperref}
\usepackage{graphicx}
\usepackage{xcolor}
\usepackage{psfrag}
\usepackage{graphicx}
\usepackage{marvosym}
\usepackage{stmaryrd}
\usepackage{MnSymbol}
\usepackage{varioref}
\usepackage[english,capitalise]{cleveref} 
\usepackage{bbm}
\usepackage{bbold} 
\usepackage{tikz-cd}
\usepackage{tikz-3dplot}
\usetikzlibrary{arrows,decorations.markings,decorations.pathmorphing}
\usetikzlibrary{arrows.meta} 

\tdplotsetmaincoords{70}{110}

\numberwithin{equation}{section}

\newtheorem{lemma}[equation]{Lemma}
\newtheorem{corollary}[equation]{Corollary}
\newtheorem{theorem}[equation]{Theorem}
\newtheorem{remark}[equation]{Remark}
\newtheorem{proposition}[equation]{Proposition}

\theoremstyle{definition}
\newtheorem{definition}[equation]{Definition}
\newtheorem{notation}[equation]{Notation}
\newtheorem{example}[equation]{Example}

\newtheorem{question}[equation]{Question}

\def\CC{{\mathbb C}}
\def\PP{{\mathbb P}}
\def\RR{{\mathbb R}}
\def\TT{{\mathbb T}}
\def\ZZ{{\mathbb Z}}

\def\ft{{\mathfrak t}}

\def\cH{{\mathcal H}}
\def\cK{{\mathcal K}}

\def\cI{{\mathcal I}}
\def\cJ{\mathcal{J}}

\newcommand{\ftz}{\ft_{_\ZZ}}

\newcommand{\cO}{\mathcal{O}}

\def \one {\mathbbm{1}} 

\newcommand{\C}{\mathbb{C}}

\newcommand{\R}{\mathbb{R}}
\newcommand{\Z}{\mathbb{Z}}

\DeclarePairedDelimiter{\abs}{\lvert}{\rvert}

\DeclarePairedDelimiter{\pair}{\langle}{\rangle}
\DeclarePairedDelimiter{\paren}{(}{)}
\DeclarePairedDelimiterX{\set}[1]{\{}{\}}{\, #1 \,}  

\renewcommand{\epsilon}{\ensuremath\varepsilon}

\renewcommand{\phi}{\ensuremath{\varphi}}

\newcommand*{\conj}[1]{\overline{#1}}	
\newcommand*{\CP}[1]{\CC \PP^{#1}}		
\newcommand*{\RP}[1]{\RR \PP^{#1}}		
\newcommand*{\clos}[1]{\overline{#1}}	

\newcommand{\id}{\mathrm{id}}			
\newcommand{\T}{\mathbb{T}}				

\newcommand{\orm}{\mathrm{o}}

\DeclareMathOperator{\AGL}{AGL}
\DeclareMathOperator{\GL}{GL}			
\DeclareMathOperator{\Bl}{\mathcal{B}} 

\renewcommand{\arraystretch}{1.5}

\title{Toric real loci via moment polytopes}
\author{João Camarneiro, Ana Cannas da Silva}
\date{July 31, 2026}
\subjclass[2020]{Primary: 53D20, 57S12; Secondary: 32V40, 51M15, 14J45}
\keywords{Toric symplectic manifolds, real loci, moment polytopes, Fano manifolds.}

\begin{document}


\begin{abstract}
Toric real loci are distinguished lagrangian submanifolds of
toric symplectic manifolds.
We develop a polyhedral model for toric real loci, called a
\textit{kaleidoscope}, based on the restriction of the moment map.
This construction provides a direct geometric description
of the topology of toric real loci and leads to simple criteria
for orientability, together with transparent formulas
for the Euler characteristic.
We illustrate the theory with numerous examples and count,
in every complex dimension up to 9, the
smooth toric Fano varieties whose real loci are orientable.
The kaleidoscope construction offers a simple, visual, and effective
framework for understanding the topology
of toric real loci through the combinatorics of their moment polytopes.
\end{abstract}

\address{Maxwell Institute for Mathematical Sciences, School of Mathematics, University of Edinburgh, Edinburgh EH9 3FD, UK}
\email{joao.camarneiro@ed.ac.uk}

\address{Department of Mathematics, ETH Zurich,
8092 Zurich, Switzerland}
\email{ana.cannas@math.ethz.ch}

\maketitle

\thispagestyle{empty}
\tableofcontents
\clearpage


\section*{Introduction}

Toric symplectic manifolds form one of the most successful meeting points of symplectic
geometry, algebraic geometry, combinatorics, and topology.
Every toric symplectic manifold admits a canonical anti-symplectic involution
whose fixed-point set is a lagrangian submanifold called its \textit{real locus}.
Thanks to the one-to-one correspondence between compact toric symplectic manifolds
and unimodular polytopes due to Delzant~\cite{Delzant88},
it is possible to study questions of both toric symplectic manifolds
and their real loci through convex polytopes.

This paper introduces a new polyhedral model for toric real loci,
which we call a \textit{kaleidoscope}.
Rather than describing the real locus
through characteristic functions,
our approach builds it from copies of the moment polytope
identified along their facets.

Our main results show that the topology of a toric real locus
can be read directly from its associated kaleidoscope.
In particular, we obtain elementary geometric criteria for orientability,
derive transparent formulas for the Euler characteristic,
and describe the effect of toric blow-ups
within the same combinatorial model.
The kaleidoscope approach provides a convenient common path
to several previously known phenomena
that is particularly user-friendly and computationally simple.

The construction is closely related in spirit to the theory of
small covers introduced by Davis and Januszkiewicz~\cite{DavisJanuszkiewicz91},
but it is tailored specifically to toric real loci.
The emphasis throughout is on obtaining an
\textit{explicit geometric description} that makes
topological properties immediately accessible from the combinatorics
of the underlying polytope.

Basic examples are complex projective space $\CC\PP^n$
and its real locus $\RR\PP^n$, both governed by
the combinatorics of a standard $n$-dimensional simplex $\mathbb{\Delta}^n$.
The kaleidoscope construction presents $\RR\PP^n$ as a cluster of $2^n$ copies
of $\mathbb{\Delta}^n$, mirrored across their common facets,
with the outer facets further glued
in pairs in a center-symmetric fashion; cf.~\cref{fig:kaleidoscope_cp3}.
The Euler characteristics are
\[
\chi(\CC\PP^n)=n+1
\qquad \text{ and } \qquad
\chi(\RR\PP^n)= \begin{cases}
0 & \text{ when $n$ is odd,}\\
1 & \text{ when $n$ is even.}
\end{cases}
\]
From the toric point of view,
the first is simply the number of vertices of the simplex $\mathbb{\Delta}^n$.
In our perspective, the second is
the $(-2)^k$-weighted sum of the number of $k$-dimensional
faces ($k = 0, 1, \ldots, n$) of that simplex.
The binomial formula shows that this gives the same result (\cref{ex:euler_simplex}):
\[
\chi(\RR\PP^n)= \sum_{k=0}^n (-2)^k \textstyle{\binom{n+1}{k+1}} =
\tfrac{1}{2} \left( 1+(-1)^n \right) = \begin{cases}
0 & \text{ when $n$ is odd,}\\
1 & \text{ when $n$ is even.}
\end{cases}
\]
Moreover, $\RR\PP^n$ is orientable if
and only if $n$ is odd.
In our perspective, this occurs precisely when each primitive normal vector
to a facet of the simplex $\mathbb{\Delta}^n$ has an odd number of odd entries.

The principal contributions of this paper are:

\begin{itemize}
\item the introduction of polytope kaleidoscopes (\cref{def:kaleidoscope});

\item the topological realization of any toric real locus as a kaleidoscope (\cref{coroll:kaleidoscope_vs_real_locus});

\item a visual description of the behavior of kaleidoscopes under toric blow-ups
(\cref{prop:kaleidoscope_blowup} and \cref{prop:blow_up_along_edge});

\item simple formulas for the Euler characteristic (\cref{thm:euler_charact});

\item a simple geometric criterion for orientability (\cref{thm:orientability});

\item the application to counting smooth toric Fano varieties that have an orientable real locus (\cref{table:fano_orientable_real_loci}) and a discussion of these numbers;

\item the full list of 2-dimensional toric real loci (\cref{coroll:4dim});

\item numerous 3-dimensional real loci within the Miyake–Oda~\cite{Oda78} examples
(\cref{subsec:MiyakeOdaNagaya});

\item open questions (\cref{question:minimaltoricrealloci}, \cref{question:circlebundles} and \cref{question:3_geometries}).
\end{itemize}

The paper is organized as follows.
After reviewing the necessary background on toric symplectic manifolds
(\cref{sec:toric_symplectic_framework})
and their real loci (\cref{sec:toric_real_loci}),
we introduce kaleidoscopes and establish their basic properties (\cref{sec:kaleidoscopes}).
We then prove the correspondence between toric real loci and kaleidoscopes
and investigate the behavior under
toric blow-ups at a fixed point (\cref{sec:kaleidoscope_model}).
We go on to derive our Euler characteristic results (\cref{sec:euler_signature})
and the simple orientability criterion illustrated with the analysis
of toric Fano manifolds (\cref{sec:orientability}).
We conclude with a variety of concrete examples
demonstrating the effectiveness of the construction in low dimensions
(\cref{sec:case_n=2,sec:case_n=3}), complemented by an appendix
on the smooth projective toric varieties of complex dimension 3 with second Betti number at most 5.


\subsection*{Acknowledgments.}

The present work grew out of an EPFL MSc thesis written by J.C.
under the supervision of A.C. as an exchange project at ETH Zurich,
during the Fall semester of 2023/24.
We gratefully acknowledge stimulating and fruitful conversations with
Miguel Abreu, Paul Biran, Jo\'e Brendel, Carlos Florentino,
Leonor Godinho, Yael Karshon, Reto Kaufmann, Georgios Moschidis,
Ana Rita Pires, Silvia Sabatini, Johannes Schmitt,
Rosa Sena-Dias, and Nick Sheridan.
We are also grateful for the hospitality of the Department of Mathematics
at Instituto Superior T\'ecnico in the Spring of 2025 for the conclusion
of this project.
J.C.’s work was partially supported by an EPFL Excellence Fellowship
and by the UKRI Centre for Doctoral Training in Algebra,
Geometry and Quantum Fields (AGQ), Grant Number EP/Y035232/1.


This work includes results obtained with programs written using Sage \cite{sagemath} and Regina \cite{Regina}.
We are grateful to the developers and maintainers of these software projects.
We also thankfully acknowledge the use of datasets of Fano polytopes obtained by Mikkel {\O}bro \cite{ObroArchived} and Andreas Paffenholz \cite{PaffenholzSite}.
The complete code with accompanying commentary is available at \url{https://github.com/J-camarneiro/toric-real-loci}.




\section{Toric Symplectic Framework}
\label{sec:toric_symplectic_framework}

In this section, we review background notions
about toric symplectic manifolds,
mostly to fix notation and conventions.
Throughout, let $T$ denote an $n$-dimensional torus.\footnote{A
\textit{torus} $T$ is a compact connected abelian Lie group.
Its Lie algebra $\ft$ and the dual vector space $\ft^*$
come equipped with lattices:
the \textit{integral lattice} $\ftz \subset \ft$ is the kernel
of the exponential map, $\exp : \ft \to T$,
and the \textit{weight lattice} is
$\ftz^* := \{ \xi \in \ft^* \mid \xi ({\ftz}) \subseteq 2\pi\ZZ \}$.
A choice of an \emph{integral basis}, i.e., a $\ZZ$-basis
of $\ftz$, yields an
identification of $T$ with the \textit{standard torus}, $\TT^n$,
given by the product of $n$ unit circles $S^1 \subset \CC$.
Then $\ft$ and $\ft^*$ are also identified with $\RR^n$,
whereas $\ftz$ and $\ftz^*$ become $(2\pi\ZZ)^n$ and $\ZZ^n$, respectively.}

A \textit{toric symplectic $T$-manifold} is a triple $(M,\omega,\mu)$,
where $(M,\omega)$ is a compact connected symplectic $2n$-dimensional
manifold\footnote{Unless otherwise stated, all maps are assumed
to be smooth, i.e., infinitely differentiable, all manifolds are smooth
and have no boundary, and all submanifolds are smoothly embedded.}
equipped with a hamiltonian effective action of the torus $T$,
and with a choice of a moment map $\mu : M \to \ft^*$
for this action.\footnote{A \textit{moment map} for an action of a torus $T$
on a symplectic manifold $(M,\omega)$ is a $T$-invariant map
$\mu : M \to \ft^*$ such that for each $X \in \ft$ we have
$d\langle \mu , X \rangle = -\imath_{X^\#} \omega$, where $X^\#$ is
the vector field on $M$ generated by the one-parameter subgroup
$\{ \exp(tX) \mid t \in \RR \}$.}

Two toric symplectic $T$-manifolds, $(M_1,\omega_1, \mu_1)$
and $(M_2,\omega_2, \mu_2)$ are \emph{isomorphic},
if there exists a $T$-equivariant symplectomorphism
$\varphi \colon (M_1, \omega_1) \to (M_2, \omega_2)$ intertwining
the moment maps (called an \emph{isomorphism}), that is:
\[
   \varphi^* \omega_2 = \omega_1,
   \varphi (g \cdot p) = g \cdot \varphi (p)  \text{ and } 
   \mu_2 (\varphi (p)) = \mu_1 (p) 
   \text{ for each } g \in T, p \in M_1.
\]
They are \emph{weakly isomorphic}~\cite{KKP07,PPRS14,PelayoSantos23},
if there exists an automorphism
$h \colon T \to T$ and a symplectomorphism
$\varphi \colon (M_1, \omega_1) \to (M_2, \omega_2)$
such that, for any $g \in T$ and $p \in M_1$, we have
\[
   \varphi(g \cdot p) = h(g) \cdot \varphi(p).
\]

Given a toric symplectic $T$-manifold $(M,\omega,\mu)$,
an important role is played by the image of its moment map.
As a special case of
the Convexity Theorem~\cite{Atiyah82,GuilleminSternberg82},
this image,
\[
   \Delta \coloneqq \mu (M) \subset \ft^*,
\]
is the convex hull of the images of the fixed points of the action,
called the \textit{moment polytope} of $(M,\omega,\mu)$.
Moreover, $\Delta$ is
unimodular\footnote{\textit{Unimodularity} means that,
for each vertex of
$\Delta$, there is an integral basis of the weight lattice defining
the edges meeting at that vertex.} ~\cite[p.323]{Delzant88},
and any unimodular polytope in $\ft^*$ determines, up to isomorphism,
a toric symplectic $T$-manifold having that as moment polytope~\cite[Theorem 2.1 and Section 3]{Delzant88}.
That is, unimodular polytopes in $\ft^*$ classify
toric symplectic $T$-manifolds up to isomorphism,
where the correspondence is given by the moment map.
It follows that unimodular polytopes in $\ft^*$ up to 
$\AGL(n,\ZZ) := \RR^n \rtimes \GL(n,\ZZ)$
classify toric symplectic $T$-manifolds up to weak isomorphism.
Furthermore, the following facts hold~\cite[Lemma 2.2]{Delzant88}:
\begin{itemize}
\item
  The fibers of the moment map $\mu : M \to \Delta$ are the $T$-orbits.
\item
  For each $\xi \in \Delta$, the fiber $\mu^{-1}(\xi)$ is a torus
  of dimension equal to that of the open face of $\Delta$ containing $\xi$.
\item
  For each $p \in M$, the isotropy group
$T_p \coloneqq \{ g\in T \mid g \cdot p = p \}$
is the subtorus of $T$ whose Lie algebra is the annihilator in $\ft$
of the open face of $\Delta$ that contains $\mu(p)$.
\end{itemize}

Following Weinstein and Delzant (see, for
instance,~\cite[Lemma 2.5]{Delzant88} or~\cite[Proposition IV.4.21]{Audin}),
we have a simple semi-local normal form.\footnote{\textit{Semi-local}
refers to a neighborhood in $M$ that is the preimage under $\mu$
of an open subset in $\Delta$.}

\begin{proposition}[Semi-Local Normal Form for a Toric Symplectic Manifold]
\label{lem:delzant_normal_form}
Let \((M,\omega,\mu)\) be a toric symplectic $T$-manifold,
with moment polytope \(\Delta = \mu(M)\).
Let \(F\) be a $k$-dimensional face of \(\Delta\),
and \(V\) an open ball in \(F\) such that its closure
\(\overline{V}\) is compact and contained
in the relative interior of \(F\).
Identify \(\ft^* \cong \RR^n\) by choosing an integral basis
of the \(k\)-dimensional subspace of \(\ft^*\) parallel to \(F\),
and extending it to an integral basis\footnote{\textit{Integrality}
is with respect to the weight lattice, so that this isomorphism
identifies \(\ft^*_\Z \cong \Z^n\).} of \(\ft^*\), so that
\(V\) becomes an open subset of a standard
\(\RR^k \subseteq \RR^n\).
Let \(B(\epsilon)\) denote the open ball in \(\C\)
centered at the origin and with radius \(\epsilon\).
		
Then, there exist a neighborhood \(U\) of \(\mu^{-1}(V)\) in \(M\), a number \(\epsilon > 0\), and a $T \cong \TT^n$ equivariant symplectomorphism
\[
   \Phi \ \colon \ U \longrightarrow \TT^k \times V \times B(\epsilon)^{n-k}
   \ \subset \ \TT^k \times \RR^k \times \C^{n-k}
\]
intertwining the moment maps,
where $\TT^k \times \RR^k \times \C^{n-k}$
is equipped with
\begin{itemize}
\item
the symplectic form 
\[
   \omega_0 =
   \sum_{j=1}^k d\theta_j \wedge d\mu_j + \sum_{j=1}^{n-k} dx_j \wedge dy_j
\]
(here, $\theta_j$ is the $2\pi$-periodic coordinate on the $j$th factor
of $\TT^k$, $\mu_j$ the coordinate on the $j$th factor of $\RR^k$,
and $z_j = x_j + i y_j$ the coordinates on the $j$th factor of $\C^{n-k}$),
\item
the \(\TT^n\)-action by
\begin{align*}
& (g_1, \dots, g_n) \cdot (e^{i\theta_1}, \dots, e^{i\theta_k}, \mu_1, \dots, \mu_k, z_1, \dots, z_{n-k}) =\\ = \; & (g_1 e^{i\theta_1}, \dots, g_k e^{i\theta_k}, \mu_1, \dots, \mu_k, g_{k+1} z_1, \dots, g_n z_{n-k}),
\end{align*}
\item
and the moment map
\(\mu_0 \colon \TT^k \times V \times B(\epsilon)^{n-k} \to \RR^n\)
given by
\begin{align*}
& \mu_0(g_1, \dots, g_k, \mu_1, \dots, \mu_k, z_1, \dots, z_{n-k}) = \\
= \; & \xi + \paren*{\mu_1, \dots, \mu_k,
-\tfrac{1}{2}\abs{z_1}^2, \dots, -\tfrac{1}{2}\abs{z_{n-k}}^2}.
\end{align*}
\end{itemize}
\end{proposition}




\section{Toric Real Loci}
\label{sec:toric_real_loci}

Now we go over the definition and main facts
about \textit{real loci}.
We focus on a toric symplectic $T$-manifold $(M,\omega,\mu)$
and restrict ourselves to real structures compatible with the $T$-action.

\begin{definition}
\label{def:toric_real}
A \textbf{toric real structure} on $(M,\omega,\mu)$ is a map $\tau \colon M \to M$ such that:
\begin{itemize}
\item $\tau$ is an anti-symplectic involution, i.e.\ $\tau^2 = \id_M$ and $\tau^*\omega = -\omega$;
\item the moment map $\mu$ is invariant under $\tau$, i.e.\ $\mu \circ \tau = \mu$.
\end{itemize}
The \textbf{toric real locus} (or simply, \textit{real locus})
of $(M,\omega,\mu,\tau)$ is the fixed-point set
\[
M^\tau \coloneqq \set{p \in M \mid \tau(p) = p}.
\]
\end{definition}

Real loci are lagrangian submanifolds~\cite{Meyer80}.
The compatibility condition between the anti-symplectic involution
and the moment map (second bullet point in \cref{def:toric_real})
may be rephrased in terms of the $T$-action as
\begin{itemize}
\item
for all $g \in T$ and $p \in M$, we have
$\tau(g \cdot p) = g^{-1} \cdot \tau(p)$.
\end{itemize}
This fact was observed in \cite[p.418]{Duistermaat83},
and a detailed proof of it is in \cite[Prop. 3.3]{Camarneiro}.
	
\begin{example}\label{ex:locuscpn}
For $\CP{n}$ viewed as a toric symplectic manifold
$(\CP{n}, \omega_{_\text{FS}}, \mu)$,\footnote{The
\emph{Fubini-Study form} on $\CP{n}$ is
$\omega_{_{\text{FS}}} = \tfrac{i}{2} \partial \bar{\partial} \ln (1+|z|^2)$
with respect to standard charts.
In particular, the sphere $\CP{1}$
has $\omega_{_{\text{FS}}} = \frac 14 \omega_{\mathrm{eucl}}$
and total area $\pi$ with respect to $\omega_{_{\text{FS}}}$.}
the standard toric real structure $\tau \colon \CP{n} \to \CP{n}$
is given by complex conjugation:
\[
   \tau([z_0 : \cdots : z_n]) = [\conj{z_0} : \cdots : \conj{z_n}].
\]
The toric real locus is then $(\CP{n})^\tau = \RP{n}$.
\end{example}

By Delzant's construction~\cite[Section 3]{Delzant88},
any toric symplectic $T$-manifold is isomorphic
to one obtained by symplectic reduction from some complex
symplectic vector space. Hence, in particular,
any toric symplectic $T$-manifold admits an inherited
\textit{complex conjugation}.
Therefore, the previous example generalizes to all toric symplectic manifolds.

\begin{remark}[More General Real Lagrangians]
There are more general \emph{real lagrangians} in
a toric symplectic manifold,
arising as fixed-point sets of antisymplectic involutions
that are not necessarily anti-commuting with the torus action.
Examples of such antisymplectic involutions have
been constructed by Brendel, Kim, and Moon~\cite{BKM23}
from the symmetries of the moment polytope
(our involutions correspond to the trivial symmetry).
The counterparts of these involutions for (algebraic) toric varieties
were investigated earlier by Delaunay~\cite{DelaunayPhD,Delaunay2,Delaunay3}.
\end{remark}

We now establish some fundamental properties of real loci in
our set-up.
Let $\tau$ be a toric real structure on a toric symplectic $T$-manifold,
$(M,\omega,\mu)$.
Since the image by $\tau$ of any fixed point is also
a fixed point and has the same image by $\mu$, it must be the same point.
Therefore, the toric real locus $M^\tau$ contains all the fixed points of the
$T$-action. 
Since we now know that $M^\tau$ is not empty,
by~\cite[Theorem 2.5]{Duistermaat83}, the moment map image of
each of its connected components is the full moment polytope $\Delta$.
Because the preimage of each vertex of $\Delta$ is a singleton,
we conclude that there is only one connected component.
Therefore, the toric real locus $M^\tau$ must be connected and not empty,
and it follows from Duistermaat's convexity
theorem~\cite[p.423]{Duistermaat83} that:

\begin{proposition}[Duistermaat's Convexity in a Toric Setting]
\label{prop:duistermaat}
Let $(M,\omega,\mu,\tau)$ be a toric symplectic
$T$-manifold equipped with a toric real structure $\tau$.
Then the toric real locus $M^\tau$ is a compact connected lagrangian submanifold
with full moment map image, that is,
$\mu (M^\tau) = \mu(M)$.
\end{proposition}

Next, we see that the subgroup of $T$ of all elements of order $2$
acts on $M^\tau$.
Following \cite{CannasKarshon},\footnote{In \cite{CannasKarshon}, this
subgroup is denoted $\sqrt{\{ \one \}}$ to be consistent
with the general case $\sqrt{H}$ for a subtorus $H \subseteq T$.}
we denote this subgroup by
\[
   \sqrt{\one} \coloneqq \set{g \in T \mid g^2 = \one},
\]
where $\one$ denotes the identity element in $T$.
Any chosen isomorphism to the standard $n$-torus, $T \cong \TT^n$, exhibits
$\sqrt{\one} \cong \{-1,1\}^n$ as a finite subgroup of order $2^n$.

\begin{lemma}\label{lem:realtorusorbit}
Let $(M,\omega,\mu,\tau)$ be a toric symplectic $T$-manifold
equipped with a toric real structure $\tau$.
Then, for each $T$-orbit $\cO$ in $M$,
the intersection $M^\tau \cap \cO$ is exactly one (nonempty) $\sqrt{\one}$-orbit.
\end{lemma}

\begin{proof}
We begin by checking that, for each $g \in \sqrt{\one}$ and $p \in M^\tau$, we have $g \cdot p \in M^\tau$. This shows that each $M^\tau \cap \cO$ is a union of $\sqrt{\one}$-orbits. Indeed, we have that
		\[\tau(g \cdot p) = g^{-1} \cdot \tau(p) = g^{-1} \cdot p = g \cdot p.\]
		
Now, suppose that $p,q \in M^\tau$ and $g \in T$ are such that $q = g \cdot p$. We will show that there exists $g' \in \sqrt{\one}$ such that $q = g' \cdot p$. Indeed, we can see that
		\[q = \tau(q) = \tau(g \cdot p) = g^{-1} \cdot \tau(p) = g^{-1} \cdot p,
        \quad \text{ hence } \quad
        g^2 \cdot p = g \cdot q = p,
        \]
        i.e., $g^2$ is in the isotropy group of $p$, $T_p$,
        which we know is a subtorus of $T$.
Hence, we can find $h \in T_p$ such that $h^2 = g^2$.
Take $g' \coloneqq gh^{-1}$.
It follows that $q = g' \cdot p$.
Moreover, $(g')^2 = g^2 h^{-2} = \one$, and thus $g' \in \sqrt{\one}$.
		
Thus far, we can conclude that each intersection $M^\tau \cap \cO$
is either empty or exactly one $\sqrt{\one}$-orbit.
However, by Duistermaat, $\mu(M^\tau) = \mu(M)$, and by Delzant,
each fiber of $\mu$ is exactly one $T$-orbit; hence,
 $M^\tau$ intersects every $T$-orbit.
\end{proof}

Moreover, the toric real locus determines the toric real structure
and is unique up to diffeomorphism:

\begin{lemma}\label{lem:locus_determines_tau}
Let $\tau$ and $\tau'$ be two toric real structures on $(M,\omega,\mu)$.
If the corresponding toric real loci are equal, $M^\tau = M^{\tau'}$,
then $\tau = \tau'$.
\end{lemma}

\begin{proof}
By \cref{lem:realtorusorbit}, any $p \in M$ is of the form $p = g \cdot p_0$ for some
    $g \in T$ and $p_0 \in M^\tau$.
    Hence, it must be
    $\tau (p) = \tau (g \cdot p_0) = g^{-1} \cdot \tau (p_0) = g^{-1} \cdot p_0$.
    This determines $\tau$ and, similarly, $\tau'$.
\end{proof}

\begin{remark}[Uniqueness of Toric Real Locus]
\label{rmk:toric_lagrangian}
    It follows from \Cref{lem:realtorusorbit} that the $T$-action on $M$
    restricts to a $\sqrt{\one}$-action on $M^\tau$.
    Moreover, $M^\tau$ is a \emph{real toric lagrangian} in the sense of
    \cite{CannasKarshon}.
    Given any other toric real structure $\tau'$ on $(M,\omega,\mu)$,
    there is an isomorphism $\varphi : M \to M$
    with $\varphi(M^\tau)  = M^{\tau'}$ by~\cite{CannasKarshon}.
    Then, by \Cref{lem:locus_determines_tau}, we have $\varphi^* \tau' = \tau$.
    In other words, a toric real locus (as well as a toric real structure)
    is unique up to isomorphism of the ambient toric symplectic manifold.
\end{remark}

\begin{remark}[Toric Real Locus of Product]
    \label{rmk:product}
    As a consequence of \Cref{rmk:toric_lagrangian}, a toric real locus in a
product of toric symplectic manifolds is diffeomorphic to the
product of toric real loci in the factors.
Indeed, if $\tau_1,\tau_2$ are toric real structures on $M_1, M_2$
and $M = M_1 \times M_2$, then $\tau : M \to M$ defined by
$\tau (p_1,p_2) := (\tau_1 (p_1) , \tau_2 (p_2))$ is a toric
real structure on $M$ and has toric real locus
$M^\tau = M_1^{\tau_1} \times M_2^{\tau_2}$.
\end{remark}

\begin{remark}[Toric Real Locus of Toric Submanifold]
    \label{rmk:real_locus_of_submfld}
Let $(M,\omega,\mu,\tau)$ be a toric symplectic $T$-manifold
equipped with a toric real structure, and let $N \subset M$
be a toric symplectic submanifold.
Then $\tau$ restricts to a toric real structure on $N$,
$\tau|_{_N} : N \to N$, with
a toric real locus $N^\tau := M^\tau \cap N$.
\end{remark}

Next, we upgrade the local normal form in \cref{lem:delzant_normal_form}
to take into account a toric real structure.
The following result is due to
Delzant~\cite[Lemma 4.4 of page 333]{Delzant88}, building on
previous normal forms, including Duistermaat's~\cite[Prop. 2.2]{Duistermaat83}.

\begin{proposition}
\label{lem:delzant_normal_form_real}
In the conditions of \cref{lem:delzant_normal_form},
given a toric real structure $\tau$ on $(M,\omega,\mu)$,
the equivariant symplectomorphism
\[
   \Phi : U \longrightarrow \TT^k \times V \times B(\epsilon)^{n-k}
\]
intertwining the moment maps takes $M^\tau \cap U$ onto
\[
   \{ (\pm 1, \dots, \pm 1, \mu_1, \dots, \mu_k, x_1, \dots, x_{n-k}) \mid
   (\mu_1, \dots, \mu_k) \in V, x_j \in B(\epsilon) \cap \RR,
   j=1,\ldots,n-k \}.
\]
\end{proposition}


We are now in a position to interpret the
restriction $\mu\vert_{M^\tau}: M^\tau \to \Delta$
as a branched covering of the moment polytope,
where $\sqrt{\one}$ is its group of deck transformations
and there are $2^n$ sheets.

\begin{remark}[Historical Context]
Viewing this restriction of the moment map as a branched covering
goes back to at least Gelfand, Kapranov, and
Zelevinsky~\cite[Ch.11, Theorem 5.4]{GelfandKapranovZelevinsky}, and
Guillemin~\cite[p.286]{Guillemin94}.
Moreover, this fits into the concept of a \emph{small cover}
introduced by Davis and Januszkiewicz~\cite{DavisJanuszkiewicz91}.
Donaldson~\cite[\S 2.2.4]{Donaldson} used this branched covering
to obtain riemannian metrics on $M^\tau$.
When $n=2$, Abreu and Gadbled~\cite{AbreuGadbled17} coupled
the assembly of $M^\tau$ from four copies of $\Delta$ with
surgery techniques to provide constructions of new lagrangian submanifolds
in $\CC\PP^2$ and in $\CC\PP^1 \times \CC\PP^1$.
\end{remark}

\begin{proposition}[Toric Real Locus as Branched Covering of Moment Polytope]
\label{prop:branched_covering}
Let $(M, \omega, \mu)$ be a toric symplectic $\TT^n$-manifold
with moment polytope $\Delta$,
let $\tau$ be a toric real structure on it, and let $M^\tau$ be its toric real locus.
Then there exists a $\sqrt{\one}$-equivariant\footnote{The group
$\sqrt{\one}$ acts on $\Delta \times \sqrt{\one}$
by multiplication in the second factor and on $M^\tau$ as a consequence
of \cref{lem:realtorusorbit}.}
surjective continuous map
\[
   q : \Delta \times \sqrt{\one} \longrightarrow M^\tau,
\]
that makes the following diagram commute
\[
\begin{tikzcd}
   {\Delta \times \sqrt{\one}}
   \arrow["q"', from=1-1, to=2-1]
   \arrow["pr", from=1-1, to=2-2]\\
   M^\tau \arrow["\mu"', from=2-1, to=2-2] & \Delta
\end{tikzcd}
\]
where $pr : \Delta \times \sqrt{\one} \to \Delta$ denotes
the natural projection onto the first factor.
Over the interior $\Delta^\orm$, the quotient map $q$ restricts to a
$\sqrt{\one}$-equivariant diffeomorphism
\[
   \Delta^\orm \times \sqrt{\one} \stackrel{\sim}{\longrightarrow}
   M^\tau \cap \mu^{-1}(\Delta^\orm).
\]
\end{proposition}

\begin{proof}
Let $P$ be any connected component of $M^\tau \cap \mu^{-1}(\Delta^\orm)$.
We will show that $\mu\vert_{\clos{P}} : \clos{P} \to \Delta$
is a proper map and a local homeomorphism, to conclude that
it is a global homeomorphism by the Hadamard-Caccioppoli Theorem (using
the fact that $\Delta$ is simply connected; see, for instance,~\cite[Thm.\ 0.1]{DeMarco94}).
Properness follows immediately from the fact that $\clos{P}$ is compact
and $\Delta$ is Hausdorff.
To show that it is a local homeomorphism, we use
the local normal form for $\mu$ and $\tau$ from
\cref{lem:delzant_normal_form,lem:delzant_normal_form_real}.

First, we use an equivariant symplectomorphism $\Phi$ from
\cref{lem:delzant_normal_form,lem:delzant_normal_form_real}
to see what $M^\tau$ looks like (``looks like'' means ``becomes,
under conjugation by $\Phi$'').
Let $p \in M^\tau$ and let $F$ be the face of $\Delta$ whose relative interior
contains $\mu(p)$.  Then, in a neighborhood of $p$ in $M$,
the moment map looks like
$\mu : \TT^k \times V \times B(\epsilon)^{n-k} \to \RR^n$
given by
\[
\mu(g_1, \dots, g_k, \mu_1, \dots, \mu_k, z_1, \dots, z_{n-k}) =
\xi + \paren*{\mu_1, \dots, \mu_k,
-\tfrac{1}{2}\abs{z_1}^2, \dots, -\tfrac{1}{2}\abs{z_{n-k}}^2},
\]
where $V$ is an open ball in the interior of $F$,
and the involution
$\tau : \TT^k \times V \times B(\epsilon)^{n-k}
\to \TT^k \times V \times B(\epsilon)^{n-k}$ is given by
\[
\tau(g_1, \dots, g_k, \mu_1, \dots, \mu_k, z_1, \dots \, , z_{n-k}) =
(g_1^{-1}, \dots, g_k^{-1}, \mu_1, \dots, \mu_k, \conj{z_1}, \dots, \conj{z_{n-k}}).
\]
Hence, in these coordinates, the toric real locus locally looks like
\[
   \{-1,1\}^k \times V \times (-\epsilon,\epsilon)^{n-k}
   \; \text{ inside } \; \T^k \times V \times B(\epsilon)^{n-k},
\]
where $(-\epsilon,\epsilon)$ is the real interval in $B(\epsilon)$.
By recalling the torus action in these coordinates, we see that the points lying over the interior $\Delta^\orm$ are those whose
coordinates $z_1, \dots, z_{n-k}$ are all non-zero,
since these are the points with a trivial isotropy group.
Hence, the portion of the toric real locus lying over the interior $\Delta^\orm$
is given locally by
\[
   \{-1,1\}^k \times V \times ((-\epsilon,0) \cup (0, \epsilon))^{n-k},
\]
i.e., $2^n$ connected components, which are open neighborhoods
in ${M^\tau \cap \mu^{-1}(\Delta^\orm)}$.

From the above, taking the closure of one of these connected components, we conclude that $\clos{P}$ locally looks like
$V \times \left[0,\epsilon\right)^{n-k}$,
with the restricted moment map locally given by
\begin{eqnarray*}
\mu : V \times \left[0,\epsilon\right)^{n-k} & \longrightarrow & \RR^n,\\
(\mu_1, \dots, \mu_k, r_1, \dots, r_{n-k}) & \longmapsto &
\paren*{\mu_1, \dots, \mu_k, -\tfrac{1}{2} r_1^2 , \dots, -\tfrac{1}{2} r_{n-k}^2}.
\end{eqnarray*}
This is, in fact, a homeomorphism to its image,
and a diffeomorphism only when $k = n$.

By invariance of $\mu$, it follows that each map
$\mu_\sigma: \sigma \cdot \clos{P} \to \Delta \times \{ \sigma \}$
defined by $\mu_\sigma (p) := (\mu (p), \sigma)$ for $\sigma \in \sqrt{\one}$
is a homeomorphism.

Now we are able to define a $\sqrt{\one}$-equivariant surjective
continuous map $q: \Delta \times \sqrt{\one} \to M^\tau$ by
\[
   q\vert_{\Delta \times \{ \sigma \}} = \paren{\mu_{\sigma}}^{-1},
\]
for each $\sigma \in \sqrt{\one}$, which satisfies $\mu \circ q = pr$.

The fact that
\[
   \mu\vert_{P} : P \longrightarrow \Delta^\orm
\]
is a homeomorphism and a local diffeomorphism (we are in the case $k = n$)
implies that it is a diffeomorphism.
By symmetry, it follows that for each $\sigma \in \sqrt{\one}$, the restriction
\[
   q\vert_{\Delta^\orm \times \{ \sigma \}} : \Delta^\orm \times \{ \sigma \}
   \longrightarrow \sigma \cdot P
\]
is a diffeomorphism.
Concatenating these, we obtain the diffeomorphism
\[
\Delta^\orm \times \sqrt{\one} \stackrel{\sim}{\longrightarrow}
\bigcupdot_{\sigma \in \sqrt{\one}} \sigma \cdot P = M^\tau \cap \mu^{-1}(\Delta^\orm),
\]
which is $\sqrt{\one}$-equivariant by construction.
\qedhere
\end{proof}

\begin{remark}[Homeomorphism, not Diffeomorphism]
It follows from the previous proof that, for $n>0$,
$q : \Delta \times \sqrt{\one} \longrightarrow M^\tau$ is
not smooth\footnote{Our convention is that a map from
$\Delta \subset \RR^n$ to a manifold
be called \textit{smooth} when it can be locally extended to a smooth map
on an open subset of $\RR^n$.}
along the boundary of $\Delta$ because of the squaring terms
in the local normal form when $k < n$.
\end{remark}


\section{Polytope Kaleidoscopes}
\label{sec:kaleidoscopes}

We now work towards a description of toric real loci
as quotients of contractible polyhedral sets in $\RR^n$
under a suitable identification of pairs of facets.
By choosing these sets in a preferred standard position, it becomes
easier to assess their orientability, compute their Euler characteristic,
and determine them, especially in lower dimensions.
In this section, we focus on the relevant quotients of polyhedral sets,
whereas in the next section, we prove the homeomorphism to the
toric real loci.

Throughout this section, we work in $\RR^n$ (later interpreted
as the dual of the Lie algebra of the standard torus
$\TT^n \coloneqq S^1 \times \ldots \times S^1$),
where we consider a unimodular polytope $\Delta$.

\begin{definition}
\label{def:standard_vertex}
A unimodular polytope in $\RR^n$ has a \textbf{standard vertex at the origin}
if it has a vertex at the origin and the $n$ edges
incident to the origin lie along the positive coordinate axes.
A unimodular polytope in $\RR^n$ has a \textbf{standard vertex}
if, after some translation, it has
a standard vertex at the origin.
\end{definition}

\begin{notation}
\label{notation:cluster}
If \(\Delta\) has a standard vertex at the origin,
we denote by $\Delta_{(\sigma_1, \ldots, \sigma_n)}$ the image of
$\Delta$ by the coordinate reflection given by the diagonal matrix
with diagonal entries $\sigma_1, \ldots, \sigma_n$,
where each $\sigma_k$ is either $+1$ or $-1$.
The $n$-tuple $(\sigma_1, \ldots, \sigma_n) \in \{ \pm 1 \}^n$
may be viewed as an element of the finite group $\sqrt{\one} \cong \{ \pm 1 \}^n$,
or as a coordinate reflection, or as the \emph{orthant} in $\RR^n$
where $\Delta_{(\sigma_1, \ldots, \sigma_n)}$ lies.
We say that an orthant $(\sigma_1, \ldots, \sigma_n)$
is \textbf{positive} (resp.\ \textbf{negative}), when
$\sigma_1 \cdot \ldots \cdot \sigma_n$ is positive (resp.\ negative).
This sign is that of the determinant of the corresponding reflection.
For short, we often write only the signs $+$ or $-$.
In particular, $\Delta_{(+,+,\ldots,+)} = \Delta_{\one} = \Delta$,
and $\Delta_{(-,+,\ldots,+)}$, being the image of $\Delta$
under the reflection of $\RR^n$ flipping
just the first coordinate, lies in a negative orthant.
\end{notation}

\begin{definition}
Let $\Delta$ be a unimodular polytope with a standard vertex at the origin.
The \textbf{$\boldsymbol{\Delta}$-cluster} is the union of all
the $2^n$ polytopes $\Delta_{(\sigma_1, \ldots, \sigma_n)}$
defined in~\cref{notation:cluster}.
\end{definition}

By the assumption of a standard vertex at the origin,
each of the $\Delta_{(\sigma_1, \ldots, \sigma_n)}$ has $n$ facets
along the coordinate hyperplanes.
While merging these in the union, those facets
get pairwise superimposed and engulfed into the interior of the
$\Delta$-cluster.
The $\Delta$-cluster is a star-shaped polytope (as its
entire boundary is visible from the origin), yet it is not necessarily convex.
An example of a $\Delta$-cluster is the polyhedron in \cref{fig:kaleidoscope_orientability} with the arrows ignored.

\begin{notation}
\label{notation:kaleidoscope}
Given $\sigma := (\sigma_1, \ldots, \sigma_n) \in \{ \pm 1 \}^n \cong \sqrt{\one}$,
and $u = (w_1, \ldots, w_n) \in \ZZ^n$, then
  \begin{align*}
  (-1)^u & \; \mbox{ denotes the orthant } \;
  \left( (-1)^{w_1}, \ldots, (-1)^{w_n} \right)
  \text{ also viewed as an element of $\sqrt{\one}$, and} \\
  (-1)^u \sigma & \; \mbox{ denotes the product} \;
  \left( (-1)^{w_1}\sigma_1, \ldots, (-1)^{w_n}\sigma_n \right)
  \mbox{ in the group } \sqrt{\one}.
  \end{align*}
\end{notation}

For example, if $\sigma = (+,-,-)$ and $u=(1,2,3)$, then
$(-1)^u=(-,+,-)$ and $(-1)^u \sigma = (-,-,+)$.
An orthant $(-1)^{u}$ is positive if and only if
$u$ has an even number of odd entries; otherwise, it is negative.
Equivalently, the orthant $(-1)^{u}$ is positive if and only if
the componentwise mod 2 reduction of $u$, denoted $[u] \in (\ZZ/2\ZZ)^n$, has length 1.

\begin{definition}
\label{def:kaleidoscope}
Let $\Delta$ be a unimodular polytope with a standard vertex at the origin
and with facets $F_1, \dots, F_d$,
where the first $n$ are those lying in the coordinate hyperplanes.
Let $u_j \in \Z^n$ be the outward-pointing
primitive normal vector to the facet $F_j$, $j=1,\ldots,d$.
The \textbf{$\boldsymbol{\Delta}$-kaleidoscope} is the quotient, $K$, of the
$\Delta$-cluster by the following equivalence relation, $\sim$: for each $j = n+1, \dots, d$ and $\sigma \in \{\pm 1\}^n$, the facet $\sigma(F_j)$
of $\Delta_{\sigma}$ is identified with its copy
in $\Delta_{(-1)^{u_j} \sigma}$ through the reflection $(-1)^{u_j}$.
\end{definition}


\cref{fig:kaleidoscope_orientability,fig:kaleidoscope_cp3}
depict examples of $\Delta$-kaleidoscopes in $n=2$ and $n=3$, respectively.
The different arrowheads on the right-side
indicate the pairwise identification of
facets (edges in \cref{fig:kaleidoscope_orientability} and
faces in \cref{fig:kaleidoscope_cp3}),
as usual in low-dimensional topology.
If we ignore the arrows, then what we see in these figures
is simply the $\Delta$-clusters.

The $\Delta$-kaleidoscope of an $n$-dimensional $\Delta$
is an $n$-dimensional pseudomanifold $K$.\footnote{\textit{Pseudomanifold}
is defined, for instance, in cf.~\cite[\S IX.8]{Massey91}. In fact, \(K\) is actually a manifold, as will be seen in \cref{rmk:pl_str} and \cref{coroll:kaleidoscope_vs_real_locus}.}
Moreover, there is a natural map (pictured as folding)
\begin{equation}
    \label{eq:folding}
     F:K \longrightarrow \Delta,
\end{equation}
whose restriction to each $n$-cell $[\Delta_\sigma]$
is induced by the reflection $\sigma: \Delta_\sigma \to \Delta$.

\begin{remark}[Alternative Definition of Kaleidoscope]
\label{rmk:def_kaleidoscope}
Equivalently, we could have defined the $\Delta$-kaleidoscope as
\[
   K = \left. \Delta \times \sqrt{\one} \, \middle \slash \, \approx \right. ,
\]
where $(\xi,\sigma) \approx (\xi, (-1)^{u_j} \sigma)$
for each $(\xi,\sigma) \in F_j \times \sqrt{\one}$, $j=1,\ldots,d$,
without excluding the facets that lie in the coordinate
hyperplanes.\footnote{Points in lower-dimensional faces of $\Delta$ are subject to identifications
given by all facets containing them.
An equivalent way to define $\approx$ would be to consider,
for each $\xi \in \Delta$, the subgroup $G(\xi) \subseteq \sqrt{\one}$
generated by all the $(-1)^{u_j}$'s for the facets $F_j$ containing $\xi$,
and then set $(\xi,\sigma) \approx (\xi, g \sigma)$
for each $(\xi,\sigma) \in \Delta \times \sqrt{\one}$ and $g \in G(\xi)$,
as is done in the \emph{basic construction} of
\cite[\S 1.5]{DavisJanuszkiewicz91}.}
Indeed, the identification given by the facets lying in the coordinate hyperplanes
matches that given in the union of the $\Delta_{\sigma}$'s
to form the $\Delta$-cluster.
This latter perspective does not require a standard vertex,
and thus allows the extension of \Cref{def:kaleidoscope} to arbitrary
unimodular polytopes.

The reason we prefer our original definition (requiring a
standard vertex) is highlighted in the simple statement of
\Cref{prop:kaleidoscope_orientability}, in
\Cref{rmk:advertise_std_vertex},
and in \Cref{sec:orientability}.
On the other hand, \Cref{coroll:kaleidoscope_vs_real_locus} and
\Cref{prop:kaleidoscope_blowup} take advantage of the
equivalence of these definitions.
\end{remark}

\begin{remark}[PL-Structure on Kaleidoscope]
\label{rmk:pl_str}
The alternative perspective in \cref{rmk:def_kaleidoscope}
provides a framework
for exhibiting a structure of PL-manifold on a $\Delta$-kaleidoscope
as follows.
For each vertex $v$ of the unimodular polytope $\Delta$,
pick $A_v \in \AGL(n,\ZZ)$ such that $A_v (v) = 0$ and
$A_v(\Delta)$ has a standard vertex at the origin.
Write $A_v (x) = M_v x + c_v$ with $M_v \in \GL(n,\ZZ)$
and $c_v = -M_v v \in \RR^n$.
Then define a chart $(U_v,\varphi_v)$
for the $\Delta$-kaleidoscope $K$ by
\begin{align*}
U_v & := \{ [\xi,\sigma] \in K \mid \xi \text{ lies in the relative
interior of some face that has $v$ as one of its vertices} \}\\
\varphi_v & : U_v \longrightarrow \RR^n, \qquad [\xi, \sigma] \longmapsto (-1)^{(M_v^{-1})^T u} A_v (\xi),
\end{align*}
where $u \in \ZZ^n$ satisfies $\sigma = (-1)^u$.
Note that $A_v(\xi)$ lies in the first orthant,
whereas $(-1)^{(M_v^{-1})^T u}$ reflects into another orthant.
The collection of these charts forms a PL-atlas for $K$.
\end{remark}

The identifications defining the $\Delta$-kaleidoscope
depend only on the \emph{parity} of the components of the
primitive normal vectors to the facets.
This observation leads to the following
criterion for orientability.

\begin{proposition}[Orientability of Kaleidoscope]
\label{prop:kaleidoscope_orientability}
Let $\Delta$ be a unimodular polytope with a standard vertex at the origin.
Its $\Delta$-kaleidoscope is orientable if and only if each primitive
normal vector to a facet of $\Delta$ has an odd number of odd entries.
\end{proposition}

\begin{proof}
Since the $\Delta$-kaleidoscope $K$ is an $n$-dimensional
pseudomanifold,
it is orientable exactly when all the $n$-cells
$[\Delta_{\sigma}]$, $\sigma \in \sqrt{\one}$
can be simultaneously coherently oriented.\footnote{An $n$-pseudomanifold is,
in particular, a regular CW complex
where each $(n-1)$-cell is contained in exactly two
$n$-cells.
Then an $n$-pseudomanifold is \textit{oriented} if
each $n$-cell is oriented and for any $(n-1)$-cell
contained in two $n$-cells, the induced orientations
on it from the two $n$-cells are opposed;
cf.~\cite[\S IX.5 and \S IX.8]{Massey91}.}

Choose the standard orientation of $\Delta$ induced from $\RR^n$
and attempt to define an orientation of $K$ by declaring the
restriction of the map in \Cref{eq:folding},
\[
F|_{[\Delta_\one]} : [\Delta_\one] \longrightarrow \Delta,
\]
to be orientation-preserving.
By considering the facets lying on the coordinate hyperplanes,
for a potentially coherent orientation,
each other restriction
\[
F|_{[\Delta_\sigma]} : [\Delta_\sigma] \longrightarrow \Delta,
\]
would have to be orientation-preserving when the sign of $\sigma$ is $+$,
and, otherwise, orientation-reversing.
For this to be globally coherent, it is necessary and sufficient
that there is no facet incident to a pair $[\Delta_{\sigma}]$ and
$[\Delta_{\sigma'}]$, where $\sigma$ and $\sigma'$ have the same sign,
i.e., that there is no facet with normal vector $u$ having a positive $(-1)^u$.
\end{proof}

\begin{remark}[Orientability in Terms of Mod 2 Length]
\label{rmk:advertise_std_vertex}
It follows from \Cref{prop:kaleidoscope_orientability} that
the $\Delta$-kaleidoscope is orientable if and only if the mod 2 reduction
of each primitive normal vector of $\Delta$ has length 1.
The conciseness of this criterion is due to having a standard vertex.
\end{remark}

\begin{example}
\label{ex:nonorientable_kaleidoscope}
The $\Delta$-kaleidoscope illustrated in
\cref{fig:kaleidoscope_orientability}
is nonorientable because there is a normal vector
(for instance, $u_2$ or $u_4$) with an even number of odd entries.
After manipulating the identification of edges (following, for instance,
\cite[\S I.7]{Massey91}), we can check
that this $\Delta$-kaleidoscope represents the connected
sum of five copies of $\RR\PP^2$.
\end{example}

\begin{figure}[ht]
\centering
\begin{tikzpicture}[scale=0.6]
\fill (0,0) circle (4pt);
\fill (0,2) circle (4pt);
\fill (1,4) circle (4pt);
\fill (2.5,5.5) circle (4pt);
\fill (4.2,5.5) circle (4pt);
\fill (8.2,1.5) circle (4pt);
\fill (8.2,0) circle (4pt);
\fill (0,-2) circle (4pt);
\fill (1,-4) circle (4pt);
\fill (2.5,-5.5) circle (4pt);
\fill (4.2,-5.5) circle (4pt);
\fill (8.2,-1.5) circle (4pt);
\fill (-8.2,0) circle (4pt);
\fill (-1,4) circle (4pt);
\fill (-2.5,5.5) circle (4pt);
\fill (-4.2,5.5) circle (4pt);
\fill (-8.2,1.5) circle (4pt);
\fill (-1,-4) circle (4pt);
\fill (-2.5,-5.5) circle (4pt);
\fill (-4.2,-5.5) circle (4pt);
\fill (-8.2,-1.5) circle (4pt);
\begin{scope}[very thick,decoration={markings,
mark=at position 0.55 with {\arrow{>}}}]
\draw[postaction={decorate}] (8.2,0)--(8.2,1.5);
\draw[postaction={decorate}] (-8.2,0)--(-8.2,1.5);
\end{scope}
\begin{scope}[very thick,decoration={markings,
mark=at position 0.5 with {\arrow{>}},
mark=at position 0.6 with {\arrow{>}}}]
\draw[postaction={decorate}] (4.2,5.5)--(2.5,5.5);
\draw[postaction={decorate}] (4.2,-5.5)--(2.5,-5.5);
\end{scope}
\begin{scope}[very thick,decoration={markings,
mark=at position 0.45 with {\arrow{>}},
mark=at position 0.55 with {\arrow{>}},
mark=at position 0.65 with {\arrow{>}}}]
\draw[postaction={decorate}] (1,4)--(0,2);
\draw[postaction={decorate}] (1,-4)--(0,-2);
\end{scope}
\begin{scope}[very thick,decoration={markings,
mark=at position 0.45 with {\arrow{>}},
mark=at position 0.5 with {\arrow{>}},
mark=at position 0.55 with {\arrow{>}},
mark=at position 0.6 with {\arrow{>}}}]
\draw[postaction={decorate}] (-4.2,5.5) -- (-8.2,1.5);
\draw[postaction={decorate}] (4.2,-5.5) -- (8.2,-1.5);
\end{scope}
\begin{scope}[thick,decoration={markings,
mark=at position 0.6 with {\arrow{triangle 60}}}]
\draw[postaction={decorate}] (8.2,1.5) -- (4.2,5.5);
\draw[postaction={decorate}] (-8.2,-1.5) -- (-4.2,-5.5);
\end{scope}
\begin{scope}[thick,decoration={markings,
mark=at position 0.55 with {\arrow{triangle 60}},
mark=at position 0.7 with {\arrow{triangle 60}}}]
\draw[postaction={decorate}] (2.5,5.5) -- (1,4);
\draw[postaction={decorate}] (-2.5,-5.5) -- (-1,-4);
\end{scope}
\begin{scope}[very thick,decoration={markings,
mark=at position 0.55 with {\arrow{triangle 60}},
mark=at position 0.7 with {\arrow{triangle 60}},
mark=at position 0.85 with {\arrow{triangle 60}}}]
\draw[postaction={decorate}] (-1,4) -- (-2.5,5.5);
\draw[postaction={decorate}] (1,-4) -- (2.5,-5.5);
\end{scope}
\begin{scope}[thick,decoration={markings,
mark=at position 0.6 with {\arrow{open triangle 90}}}]
\draw[postaction={decorate}] (0,2)--(-1,4);
\draw[postaction={decorate}] (0,-2)--(-1,-4);
\end{scope}
\begin{scope}[thick,decoration={markings,
mark=at position 0.55 with {\arrow{open triangle 90}},
mark=at position 0.7 with {\arrow{open triangle 90}}}]
\draw[postaction={decorate}] (-2.5,5.5)--(-4.2,5.5);
\draw[postaction={decorate}] (-2.5,-5.5)--(-4.2,-5.5);
\end{scope}
\begin{scope}[very thick,decoration={markings,
mark=at position 0.45 with {\arrow{open triangle 90}},
mark=at position 0.6 with {\arrow{open triangle 90}},
mark=at position 0.75 with {\arrow{open triangle 90}}}]
\draw[postaction={decorate}] (-8.2,0)--(-8.2,-1.5);
\draw[postaction={decorate}] (8.2,0)--(8.2,-1.5);
\end{scope}
\node at (3.7,2.0) {$\Delta_{_{(+,+)}} = \Delta$};
\node at (4,-2.5) {$\Delta_{_{(+,-)}}$};
\node at (-3.5,2.0) {$\Delta_{_{(-,+)}}$};
\node at (-3.5,-2.5) {$\Delta_{_{(-,-)}}$};
\filldraw[fill=blue!50,fill opacity=0.25] (0,0) -- (0,2) -- (1,4) -- (2.5,5.5) -- (4.2,5.5) -- (8.2,1.5) -- (8.2,0) -- cycle;
\filldraw[fill=blue!20,fill opacity=0.25] (0,0) -- (0,2) -- (-1,4) -- (-2.5,5.5) -- (-4.2,5.5) -- (-8.2,1.5) -- (-8.2,0) -- cycle;
\filldraw[fill=blue!20,fill opacity=0.25] (0,0) -- (0,-2) -- (1,-4) -- (2.5,-5.5) -- (4.2,-5.5) -- (8.2,-1.5) -- (8.2,0) -- cycle;
\filldraw[fill=blue!20,fill opacity=0.25] (0,0) -- (0,-2) -- (-1,-4) -- (-2.5,-5.5) -- (-4.2,-5.5) -- (-8.2,-1.5) -- (-8.2,0) -- cycle;
\draw[thick, -stealth] (0.5,3) -- (-1.5,4) node[scale=0.5,above right] {$u_5=\begin{pmatrix}-2 \\1 \end{pmatrix}$};
\draw[thick, -stealth] (1.75,4.75) -- (0.75,5.75) node[scale=0.5,above] {$u_4=\begin{pmatrix}-1 \\1 \end{pmatrix}$};
\draw[thick, -stealth] (3.35,5.5) -- (3.35,6.5) node[scale=0.5,above right] {$u_3=\begin{pmatrix}0 \\1 \end{pmatrix}$};
\draw[thick, -stealth] (6.2,3.5) -- (7.2,4.5) node[scale=0.5,right] {$u_2=\begin{pmatrix}1 \\1 \end{pmatrix}$};
\draw[thick, -stealth] (8.2,0.75) -- (9.2,0.75) node[scale=0.5,right] {$u_1=\begin{pmatrix}1 \\0 \end{pmatrix}$};
\end{tikzpicture}
\caption{A nonorientable $\Delta$-kaleidoscope
for a polytope $\Delta$ with $7$ edges and $n=2$.}
\label{fig:kaleidoscope_orientability}
\end{figure}
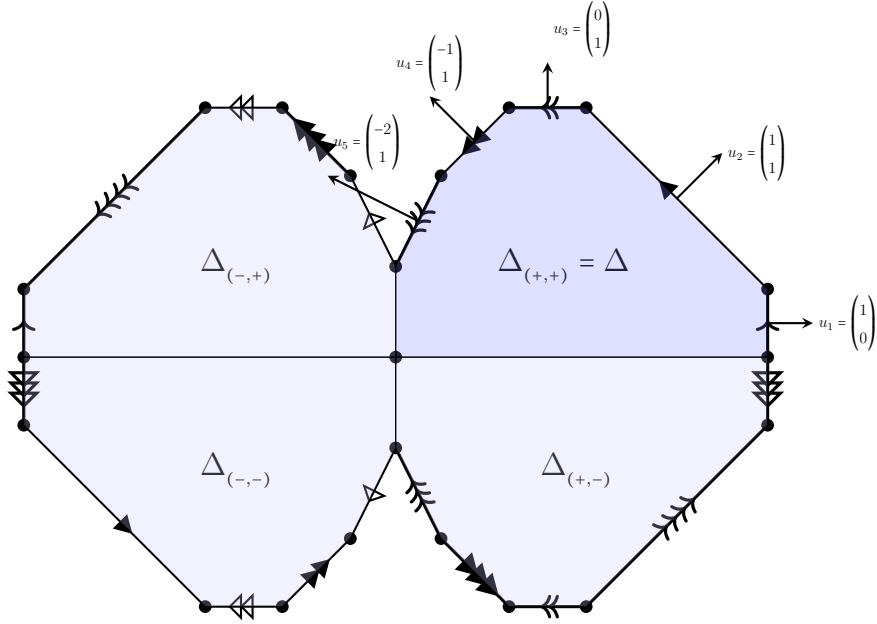

For $n=2$, if there is an odd number of edges, then the
$\Delta$-kaleidoscope is necessarily nonorientable; indeed,
it is enough to consider the primitive normal vectors of $\Delta$ mod 2,
denoted $[u_k]$,
to see that the parities required for orientability
force each $[u_k]$ to be one of the standard vectors,
which is not compatible with unimodularity at all vertices.
Actually, for $n=2$, the $\Delta$-kaleidoscope can only
be orientable when there are exactly four edges;
this follows from a
classification of unimodular polygons; see \cref{coroll:4dim}.

The criterion in \cref{prop:kaleidoscope_orientability} becomes
more interesting in higher dimensions, where a simple
classification of unimodular polytopes is not available;
see \cref{exs_fano_3,exs_fano_4,exs:orientable_miyake_oda_nagaya,exs:non-orientable_miyake_oda_nagaya}.

\begin{remark}[Orientable Double Cover]
    When the $\Delta$-kaleidoscope is not orientable, it might be
    helpful to use a model for its connected orientable double cover of the form
\[
   \widetilde K = \left. \Delta \times \sqrt{\one} \times \{ \pm 1 \}
   \, \middle \slash \, \approx \right. ,
\]
where $(\xi,\sigma,\varepsilon) \approx
(\xi, (-1)^{u_j} \sigma, -(-1)^{u_j}\varepsilon)$ for each
$(\xi,\sigma,\varepsilon) \in F_j \times \sqrt{\one} \times \{ \pm1 \}$,
$j=1,\ldots,d$.
\end{remark}


\begin{proposition}[Euler Characteristic of Kaleidoscope]
\label{prop:kaleidoscope_euler}
Let $\Delta$ be an $n$-dimensional unimodular polytope
with a standard vertex at the origin.
Let $a_k$ be the number of $k$-faces in $\Delta$.
Then the Euler characteristic of the $\Delta$-kaleidoscope is
\[
   \sum_{k=0}^n (-1)^k 2^k a_k \ .
\]
\end{proposition}

\begin{proof}
The polytope $\Delta$ is naturally a CW complex where the $k$-cells
are the $k$-faces for $k=0,1,\dots,n$.
Each of the $2^n$ reflected copies of $\Delta$ by reflections
$\sigma \in \{ \pm 1 \} ^n$ carries a similar decomposition.
The identification of facets in pairs induces a CW complex
on the $\Delta$-kaleidoscope, with $2^n$ $n$-cells
and $\frac 12 2^n a_{n-1}=2^{n-1} a_{n-1}$ $(n-1)$-cells.
The identification of facets induces identifications of lower
faces, since each of the $k$-faces of $\Delta$ lies at the
intersection of $n-k$ facets; hence, it originates
exactly $\frac{1}{2^{n-k}} 2^n a_k=2^k a_k$ distinct $k$-cells
in the $\Delta$-kaleidoscope.
In particular, each of the vertices of $\Delta$ (case $k=0$)
originates exactly one vertex in the $\Delta$-kaleidoscope.
\end{proof}

The condition of $\Delta$ having a standard vertex
at the origin is not required for the preceding proposition, in view of \cref{rmk:def_kaleidoscope}.

\begin{example}
The polygon $\Delta$ in \Cref{fig:kaleidoscope_orientability} has
$a_0 = 7, a_1=7, a_2=1$;
hence, its $\Delta$-kaleidoscope has Euler characteristic
$\chi = (-2)^0 \cdot 7 + (-2)^1 \cdot 7 + (-2)^2 \cdot 1 = -3$.
This fits with the fact that this $\Delta$-kaleidoscope
represents the surface $\#_5 \overline{\RP{2}}$
(see \Cref{ex:nonorientable_kaleidoscope}).
\end{example}

\begin{example}
\label{ex:euler_simplex}
    Let $\Delta$ be the $n$-dimensional simplex with
    vertices at the origin and the points with coordinate
    1 along each of the $n$ axes.
    Its number of $k$-faces is the binomial coefficient
    \[
    a_k = \binom{n+1}{k+1}.
    \]
By \cref{prop:kaleidoscope_euler},
the Euler characteristic of the $\Delta$-kaleidoscope is then
\begin{eqnarray*}
\sum_{k=0}^n (-1)^k 2^k \textstyle{\binom{n+1}{k+1}} & = & 
\sum_{j=1}^{n+1} (-1)^{j-1} \, 2^{j-1} \textstyle{\binom{n+1}{j}} \\
& = & 
-\tfrac{1}{2} \sum_{j=1}^{n+1} (-2)^j \textstyle{\binom{n+1}{j}} \\
& = & -\tfrac{1}{2} \Big( -1 + \underbrace{\sum_{j=0}^{n+1} 1^{n+1}
\, (-2)^j \textstyle{\binom{n+1}{j}}}_{=(1-2)^{n+1} = -(-1)^n} \Big) \\
& = & \tfrac{1}{2} \left( 1+(-1)^n \right) \\
& = & \begin{cases}
0 & \text{ when $n$ is odd,}\\
1 & \text{ when $n$ is even.}
\end{cases}
\end{eqnarray*}

\end{example}




\section{Kaleidoscope Model for a Toric Real Locus}
\label{sec:kaleidoscope_model}

It follows from \Cref{prop:branched_covering}
that the restriction of the moment map to the toric real locus,
$\mu\vert_{M^\tau}: M^\tau \to \Delta$,
is a $2^n$-fold branched covering with the branch set being the boundary
of $\Delta$.
In particular, since $\Delta^\orm$ is contractible,
$M^\tau \cap \mu^{-1}(\Delta^\orm)$ is the disjoint
union of $2^n$ open subsets of $M^\tau$, each of which is mapped
diffeomorphically onto $\Delta^\orm$ by $\mu$.
In this section, we exploit these facts translated to the $\Delta$-kaleidoscope.

Throughout this section, we choose an initial integral basis
of the Lie algebra $\ft$ of a $n$-dimensional torus $T$, thus
identifying $\ft$ with $\RR^n$, so that
the integral lattice becomes $(2\pi\ZZ)^n$ under this identification.
In this way, the torus $T$ is identified with the
standard torus $\TT^n \coloneqq S^1 \times \ldots \times S^1$,
and the dual of the Lie algebra $\ft^*$ is also identified with $\RR^n$.

Let $(M, \omega, \mu)$ be a toric symplectic $\TT^n$-manifold
equipped with a toric real structure $\tau$; cf. \cref{def:toric_real}.
Let $M^\tau$ be the corresponding toric real locus.
Use a translation to bring one vertex of its moment polytope
to the origin,
 followed by a transformation from $\GL(n,\Z)$ to make that a
\emph{standard vertex at the origin} (see \cref{def:standard_vertex}).
Attaining this condition amounts to adding a suitable
constant to $\mu$ and changing the integral basis of $\ft$,
i.e., choosing another splitting $\TT^n \cong S^1 \times \ldots \times S^1$.
Altogether, this amounts to acting by $\AGL(n,\Z)$
on the dual of the Lie algebra.
These changes yield weakly isomorphic toric symplectic
$\TT^n$-manifolds and diffeomorphic toric real loci.
Hence, for the purpose of studying the toric real locus, we may assume
that the moment polytope $\Delta$ has a standard vertex at the origin.

\begin{corollary}[Toric Real Locus is Homeomorphic to Kaleidoscope]
\label{coroll:kaleidoscope_vs_real_locus}
In the conditions of \cref{prop:branched_covering}
and assuming that $\Delta$ has a standard vertex at the origin,
the map $q$ (given in \cref{prop:branched_covering}) and
the equivalence relation $\approx$ (from \cref{rmk:def_kaleidoscope})
induce a homeomorphism $h$ that makes the following diagram commute:
\begin{eqnarray}
\label{eq:commutative_diagram}
\begin{tikzcd}
   & {\Delta \times \sqrt{\one}}
   \arrow["h"', from=2-1, to=2-2]
   \arrow["mod \, \approx"', from=1-2, to=2-1]\\
   K \arrow["q", from=1-2, to=2-2] & M^\tau 
\end{tikzcd}
\end{eqnarray}
\end{corollary}

\begin{proof}
We analyze the injectivity failure of $q$.
The restriction
\[
q \vert_{\Delta^\orm \times \sqrt{\one}} :
\Delta^\orm \times \sqrt{\one} \longrightarrow M^\tau \cap M^\orm,
\]
where $M^\orm := \mu^{-1} (\Delta^\orm)$, is bijective
since the $\sqrt{\one}$-action is free on $M^\tau \cap M^\orm$.
In general, $q(\xi,\sigma) = q(\xi',\sigma')$ implies that $\xi=\xi'$
because $\mu \circ q$ is the projection onto the first factor.

For starters, consider a point $\xi$ in the relative interior
of a facet of $\Delta$.
The corresponding $T$-orbit $\mu^{-1} (\xi)$ has an isotropy group,
the circle $G$, whose Lie algebra is spanned by a primitive
normal vector, $u$, to that facet; cf.~\cite[Lemma 2.2]{Delzant88}.
By the equivariance of $q$,
\[
   q(\xi,\sigma) = q(\xi,\sigma') \qquad \iff \qquad
   \sigma \cdot q(\xi,\one) = \sigma' \cdot q(\xi,\one)
\]
and this holds if and only if $\sigma ^{-1} \sigma' \in G$.
Therefore, $q |_{\{ \xi \} \times \sqrt{\one}}$ is the quotient map
by the action of the subgroup $G \cap \sqrt{\one}$, which is
generated by $(-1)^u$ in the sense of \cref{notation:kaleidoscope}.

In general, when $\xi$ is in the relative interior of a $k$-face of $\Delta$,
the corresponding $T$-orbit $\mu^{-1} (\xi)$ has isotropy group
the subtorus $G$, whose Lie algebra is spanned by the
outward-pointing primitive normal vectors $u_1, u_2,\ldots u_{n-k}$
to the $n-k$ facets containing that face; cf.~\cite[Lemma 2.2]{Delzant88}.
Again, by the equivariance of $q$, we have 
$q(\xi,\sigma) = q(\xi,\sigma')$ if and only if
$\sigma ^{-1} \sigma' \in G$.
Therefore, $q |_{\{ \xi \} \times \sqrt{\one}}$ is the quotient map
by the action of the subgroup $G \cap \sqrt{\one}$,
which is generated by the $n-k$ elements
\[
   (-1)^{u_1}, (-1)^{u_2}, \ldots , (-1)^{u_{n-k}} .
\]
This coincides with the equivalence relation $\approx$ defining
the kaleidoscope $K$ in \cref{rmk:def_kaleidoscope}.
Therefore, $q$ factors through $\approx$, giving a well-defined
bijection $h$ that makes the diagram \eqref{eq:commutative_diagram} commute.

Since $K$ has the quotient topology and $q$ is continuous,
by the commutativity of the diagram, $h$ is also continuous.
Since $h$ is bijective, its inverse is continuous if
and only if $h$ is a closed map, which holds because $K$ is compact
and $h$ is continuous.
Altogether, $h$ is a homeomorphism.
\end{proof}


We next focus on the effect of a blow-up at a fixed point
upon the toric real locus \emph{from the viewpoint of kaleidoscopes}.
The fact that blow-up can be performed in a symplectic
and equivariant fashion, in particular in a toric symplectic fashion,
goes back to Gromov~\cite[p.342]{GromovBook}, McDuff~\cite{McDuff}, and
Guillemin and Sternberg~\cite{GuilleminSternberg89}.
In the present paper,
\emph{blow-up} always means a \emph{toric blow-up},
preserving the category of toric symplectic manifolds.
This is described in~\cite[\S 3]{KarshonKessler07}
and \cite[Ch.1]{GuilleminBook} among other references.
Our starting point is the well-known effect of a blow-up at a fixed point
upon the moment polytope itself, which is called a
\emph{blow-up at a vertex} (or a \emph{corner chopping} \cite{KKP07}):

\begin{definition}
\label{def:blow_up_at_vertex}
    Let $\xi$ be a vertex of a unimodular polytope $\Delta$ in $\RR^n$,
let $v_1, \ldots, v_n \in \Z^n$ be the primitive vectors along the edges
pointing from $\xi$ towards each of the adjacent vertices,
and let $\epsilon$ be a positive number small enough
for all $\xi+\varepsilon v_j$, $j=1,\ldots,n$,
to lie in the interior of the edges incident to $\xi$.
    The \textbf{$\boldsymbol{\epsilon}$-blow-up of $\Delta$
at $\boldsymbol{\xi}$} is the transformation that produces the polytope
$\widetilde \Delta$
whose vertices are those of $\Delta$ except that $\xi$ is replaced by
the $n$ new vertices $\xi+\varepsilon v_j$, $j=1,\ldots,n$.
\end{definition}

The new polytope, $\widetilde \Delta$, obtained by
$\epsilon$-blow-up of $\Delta$ at $\xi$, is again unimodular:
at a new vertex $\xi+\varepsilon v_j$, the primitive edge vectors are
$v_j$ and $v_k-v_j$ with $k \neq j$.
The normal to the new facet (i.e., the facet shared by the new vertices) is the sum of
the normals to the old facets meeting at $\xi$.
Moreover, $\widetilde \Delta$ corresponds to the toric symplectic manifold
$(\widetilde M,\widetilde \omega,\widetilde \mu)$ that
(up to isomorphism) is the $\epsilon$-blow-up of the toric symplectic manifold
corresponding to $\Delta$ at the fixed point $p = \mu^{-1}(\xi)$.
Although this blow-up symplectically depends on the parameter $\epsilon$,
we disregard $\epsilon$ since it does not affect the diffeomorphism
type of $M$~\cite[\S 7.1]{McDuffSalamon}, nor that of a toric real locus.

We provide a proof of the following standard fact
\emph{via kaleidoscopes}.

\begin{proposition}[Toric Real Locus of Blow-Up is Blow-Up of Toric Real Locus, Instance 1]
\label{prop:kaleidoscope_blowup}
Let $(M, \omega, \mu, \tau)$ be a $2n$-dimensional toric symplectic manifold,
equipped with a toric real structure, let
$(\widetilde M,\widetilde \omega,\widetilde \mu)$ be a blow-up
of $(M, \omega, \mu)$ at a fixed point $p$,
and let $\widetilde{\tau}$ be a toric real structure on it.
Then the toric real locus $M^{\widetilde{\tau}}$ is homeomorphic to
$M^{\tau} \# \RR \PP^n$.\footnote{When $n$ is odd,
$\RR\PP^n$ is orientable, but the homeomorphism (or diffeomorphism) type
of the connected sum with it does not depend on the choice of orientation,
because $\RR\PP^n$ admits an orientation-reversing diffeomorphism.
When $n$ is even, $\RR\PP^n$ is not orientable anyway.}
\end{proposition}


\begin{figure}[ht]
\centering
\begin{tikzpicture}[scale = 0.6,baseline=-50]
\tikzset{snake it/.style={decorate, decoration={snake, amplitude=0.5mm}}}
\fill (0,0) circle (2pt);
\filldraw[fill=blue!30,fill opacity=0.25, draw=white] (0,2.71) -- (0,1) -- (1,0) -- (1.14 + 1*2.71,0)  decorate[snake it] {to[out=90, in=-10] (0,2.71)};
\draw (0,2.71) -- (0,1);
\draw (0,1) -- (1,0); 
\draw (1,0) -- (1.14 + 1*2.71,0);
\draw[snake it,dotted,thick] (1.14 + 1*2.71,0) to[out=90, in=-10] (0,2.71);
\fill (1,0) circle (2pt);
\fill (0,1) circle (2pt);
\filldraw[dashed,fill=blue!30,fill opacity=0] (1,0) -- (0,0) -- (0,1);
\node at (1.5,1.5) {$\widetilde \Delta$};
\end{tikzpicture}
\hspace{6em}
\begin{tikzpicture}[scale = 0.6]
	\tikzset{snake it/.style={decorate, decoration={snake, amplitude=0.5mm}}}
	\fill (0,0) circle (2pt);
	\filldraw[fill=blue!30,fill opacity=0.25, draw=white] (0,2.71) -- (0,1) -- (1,0) -- (1.14 + 1*2.71,0)  decorate[snake it] {to[out=90, in=-10] (0,2.71)};
	\draw (0,2.71) -- (0,1) -- (1,0) -- (1.14 + 1*2.71,0);
	\draw[snake it,dotted,thick] (1.14 + 1*2.71,0) to[out=90, in=-10] (0,2.71);
	\fill (1,0) circle (2pt);
	\fill (0,1) circle (2pt);
	\filldraw[dashed,fill=blue!30,fill opacity=0] (1,0) -- (0,0) -- (0,1);
	\draw[decoration={markings,
	mark=at position 0.5 with {\arrow{>}}}] (0,1) -- (1,0);
\begin{scope}[xscale=-1, yscale=1]
	\fill (0,0) circle (2pt);
	\filldraw[fill=blue!30,fill opacity=0.25, draw=white] (0,2.71) -- (0,1) -- (1,0) -- (1.14 + 1*2.71,0)  decorate[snake it] {to[out=90, in=-10] (0,2.71)};
	\draw (0,2.71) -- (0,1) -- (1,0) -- (1.14 + 1*2.71,0);
	\draw[snake it,dotted,thick] (1.14 + 1*2.71,0) to[out=90, in=-10] (0,2.71);
	\fill (1,0) circle (2pt);
	\fill (0,1) circle (2pt);
	\filldraw[dashed,fill=blue!30,fill opacity=0] (1,0) -- (0,0) -- (0,1);
\end{scope}
\begin{scope}[xscale=1, yscale=-1]
	\fill (0,0) circle (2pt);
	\filldraw[fill=blue!30,fill opacity=0.25, draw=white] (0,2.71) -- (0,1) -- (1,0) -- (1.14 + 1*2.71,0)  decorate[snake it] {to[out=90, in=-10] (0,2.71)};
	\draw (0,2.71) -- (0,1) -- (1,0) -- (1.14 + 1*2.71,0);
	\draw[snake it,dotted,thick] (1.14 + 1*2.71,0) to[out=90, in=-10] (0,2.71);
	\fill (1,0) circle (2pt);
	\fill (0,1) circle (2pt);
	\filldraw[dashed,fill=blue!30,fill opacity=0] (1,0) -- (0,0) -- (0,1);
\end{scope}
\begin{scope}[xscale=-1, yscale=-1]
	\fill (0,0) circle (2pt);
	\filldraw[fill=blue!30,fill opacity=0.25, draw=white] (0,2.71) -- (0,1) -- (1,0) -- (1.14 + 1*2.71,0)  decorate[snake it] {to[out=90, in=-10] (0,2.71)};
	\draw (0,2.71) -- (0,1) -- (1,0) -- (1.14 + 1*2.71,0);
	\draw[snake it,dotted,thick] (1.14 + 1*2.71,0) to[out=90, in=-10] (0,2.71);	
	\fill (1,0) circle (2pt);
	\fill (0,1) circle (2pt);
	\filldraw[dashed,fill=blue!30,fill opacity=0] (1,0) -- (0,0) -- (0,1);
\end{scope}
\begin{scope}[very thick,decoration={
		markings,
		mark=at position 0.5 with {\arrow{>}}}
	]
	\draw[postaction={decorate}] (1,0)--(0,1);
	\draw[postaction={decorate}] (-1,0)--(0,-1);
\end{scope}
\begin{scope}[very thick,decoration={
		markings,
		mark=at position 0.45 with {\arrow{>}},
		mark=at position 0.55 with {\arrow{>}}}
	]
	\draw[postaction={decorate}] (0,1)--(-1,0);
	\draw[postaction={decorate}] (0,-1)--(1,0);
\end{scope}
\node at (1.5,1.5) {$\widetilde K$};
\end{tikzpicture}
\caption{The effect of a blow-up at the origin for the toric real locus modelled
by a kaleidoscope, in the case $n=2$.
The image on the left depicts a portion of the moment polytope
of the blow-up near the origin, whereas the image on the right
depicts the corresponding portion of the kaleidoscope (in the
extended sense given by \Cref{rmk:def_kaleidoscope}).}
\label{fig:reallocusafterblowup}
\end{figure}
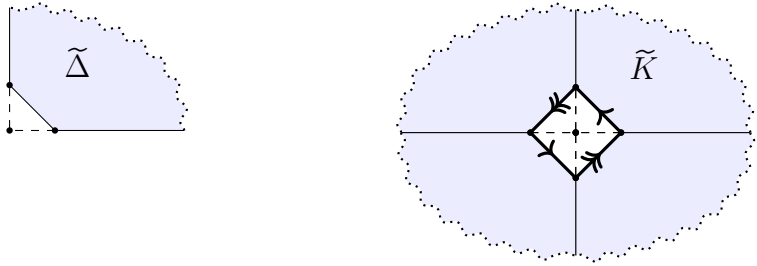

\begin{proof}
Without loss of generality, assume that the moment polytope
$\Delta := \mu(M)$ has a standard vertex at the origin,
and that the point at which blow-up is performed is $p=\mu^{-1} (0)$;
this may always be achieved by a shift of the moment map by a constant
plus an appropriate choice of integral basis. 
Let $F_1, \ldots, F_d$ be the facets of $\Delta$,
and let $K$ be the $\Delta$-kaleidoscope.

Let $\widetilde \Delta$ be the moment polytope
of  $(\widetilde M,\widetilde \omega,\widetilde \mu)$,
and let $\widetilde F_1, \ldots, \widetilde F_{d+1}$ be its facets,
where the first $d$ are subsets of those of $\Delta$,
and where $\widetilde F_{d+1}$ is the convex hull of the
$\epsilon$-multiples of the standard basis vectors
($\epsilon$ is the blow-up parameter and will become irrelevant).

The polytope $\widetilde \Delta$ no longer has a vertex
at the origin, yet the extended definition
of kaleidoscope described in \Cref{rmk:def_kaleidoscope}
still applies and yields
\[
   \widetilde K := \left. \widetilde \Delta \times \sqrt{\one}
   \, \middle \slash \, \approx \right. ,
\]
where $(\xi,\sigma) \approx (\xi, (-1)^{u_j} \sigma)$
for each $(\xi,\sigma) \in \widetilde F_j \times \sqrt{\one}$,
$j=1,\ldots,d+1$.
The primitive normal vectors are the old ones
$u_1,\ldots ,u_d$ for $\Delta$ and $u_{d+1} = (-1, \ldots ,-1)$.

Therefore, $\widetilde K$ is obtained from $K$ by removing
a cross-polytope (which is homeomorphic to a ball) with vertices given by the
$\pm\epsilon$-multiples of the standard basis vectors
and identifying the facets of that cross-polytope
pairwise in a fashion similar to that for the toric real locus of $\CC\PP^n$,
i.e., collapsing the ensuing boundary sphere by identifying antipodal points.
This is illustrated in \cref{fig:reallocusafterblowup} when $n=2$.
Equivalently, this transformation amounts to removing from $K$ a ball
and attaching the complement of a ball in $\RR\PP^n$
to the ensuing boundary sphere.
\cref{fig:connected_sum_rp2} illustrates the case $n=2$.
\end{proof}


\begin{figure}[ht]
\centering
\begin{tikzpicture}[scale = 0.6,baseline=-50]
  	\draw[thick,fill=blue!15,fill opacity=10] (0,2) -- (2,0) -- (0,-2) -- (-2,0) -- cycle;
\tikzset{snake it/.style={decorate, decoration={snake, amplitude=0.5mm}}}
	\fill (0,0) circle (2pt);
	\filldraw[fill=blue!30,fill opacity=0.25, draw=white] (0,2.71) -- (0,0) -- (1.14 + 1*2.71,0)  decorate[snake it] {to[out=90, in=-10] (0,2.71)};
	\draw (0,2.71) -- (0,0) -- (1.14 + 1*2.71,0);
	\draw[snake it,dotted,thick] (1.14 + 1*2.71,0) to[out=90, in=-10] (0,2.71);
\begin{scope}[xscale=-1, yscale=1]
	\filldraw[fill=blue!30,fill opacity=0.25, draw=white] (0,2.71) -- (0,0) -- (1.14 + 1*2.71,0)  decorate[snake it] {to[out=90, in=-10] (0,2.71)};
	\draw (0,2.71) -- (0,0) -- (1.14 + 1*2.71,0);
	\draw[snake it,dotted,thick] (1.14 + 1*2.71,0) to[out=90, in=-10] (0,2.71);
\end{scope}
\begin{scope}[xscale=1, yscale=-1]
	\filldraw[fill=blue!30,fill opacity=0.25, draw=white] (0,2.71) -- (0,0) -- (1.14 + 1*2.71,0)  decorate[snake it] {to[out=90, in=-10] (0,2.71)};
	\draw (0,2.71) -- (0,0) -- (1.14 + 1*2.71,0);
	\draw[snake it,dotted,thick] (1.14 + 1*2.71,0) to[out=90, in=-10] (0,2.71);
\end{scope}
\begin{scope}[xscale=-1, yscale=-1]
	\filldraw[fill=blue!30,fill opacity=0.25, draw=white] (0,2.71) -- (0,0) -- (1.14 + 1*2.71,0)  decorate[snake it] {to[out=90, in=-10] (0,2.71)};
	\draw (0,2.71) -- (0,0) -- (1.14 + 1*2.71,0);
	\draw[snake it,dotted,thick] (1.14 + 1*2.71,0) to[out=90, in=-10] (0,2.71);	
\end{scope}
\end{tikzpicture}
\hspace{6em}
\begin{tikzpicture}[scale = 0.6]
  	\draw[thick,fill=blue!15,fill opacity=10] (0,2) -- (2,0) -- (0,-2) -- (-2,0) -- cycle;
  	\draw[thick,fill=white,fill opacity=10] (0,1) -- (1,0) -- (0,-1) -- (-1,0) -- cycle;
\tikzset{snake it/.style={decorate, decoration={snake, amplitude=0.5mm}}}
	\fill (0,0) circle (2pt);
	\filldraw[fill=blue!30,fill opacity=0.25, draw=white] (0,2.71) -- (0,1) -- (1,0) -- (1.14 + 1*2.71,0)  decorate[snake it] {to[out=90, in=-10] (0,2.71)};
	\draw (0,2.71) -- (0,1) -- (1,0) -- (1.14 + 1*2.71,0);
	\draw[snake it,dotted,thick] (1.14 + 1*2.71,0) to[out=90, in=-10] (0,2.71);
	\fill (1,0) circle (2pt);
	\fill (0,1) circle (2pt);
	\filldraw[dashed,fill=blue!30,fill opacity=0] (1,0) -- (0,0) -- (0,1);
	\draw[decoration={markings,
	mark=at position 0.5 with {\arrow{>}}}] (0,1) -- (1,0);
\begin{scope}[xscale=-1, yscale=1]
	\fill (0,0) circle (2pt);
	\filldraw[fill=blue!30,fill opacity=0.25, draw=white] (0,2.71) -- (0,1) -- (1,0) -- (1.14 + 1*2.71,0)  decorate[snake it] {to[out=90, in=-10] (0,2.71)};
	\draw (0,2.71) -- (0,1) -- (1,0) -- (1.14 + 1*2.71,0);
	\draw[snake it,dotted,thick] (1.14 + 1*2.71,0) to[out=90, in=-10] (0,2.71);
	\fill (1,0) circle (2pt);
	\fill (0,1) circle (2pt);
	\filldraw[dashed,fill=blue!30,fill opacity=0] (1,0) -- (0,0) -- (0,1);
\end{scope}
\begin{scope}[xscale=1, yscale=-1]
	\fill (0,0) circle (2pt);
	\filldraw[fill=blue!30,fill opacity=0.25, draw=white] (0,2.71) -- (0,1) -- (1,0) -- (1.14 + 1*2.71,0)  decorate[snake it] {to[out=90, in=-10] (0,2.71)};
	\draw (0,2.71) -- (0,1) -- (1,0) -- (1.14 + 1*2.71,0);
	\draw[snake it,dotted,thick] (1.14 + 1*2.71,0) to[out=90, in=-10] (0,2.71);
	\fill (1,0) circle (2pt);
	\fill (0,1) circle (2pt);
	\filldraw[dashed,fill=blue!30,fill opacity=0] (1,0) -- (0,0) -- (0,1);
\end{scope}
\begin{scope}[xscale=-1, yscale=-1]
	\fill (0,0) circle (2pt);
	\filldraw[fill=blue!30,fill opacity=0.25, draw=white] (0,2.71) -- (0,1) -- (1,0) -- (1.14 + 1*2.71,0)  decorate[snake it] {to[out=90, in=-10] (0,2.71)};
	\draw (0,2.71) -- (0,1) -- (1,0) -- (1.14 + 1*2.71,0);
	\draw[snake it,dotted,thick] (1.14 + 1*2.71,0) to[out=90, in=-10] (0,2.71);	
	\fill (1,0) circle (2pt);
	\fill (0,1) circle (2pt);
	\filldraw[dashed,fill=blue!30,fill opacity=0] (1,0) -- (0,0) -- (0,1);
\end{scope}
\begin{scope}[very thick,decoration={
		markings,
		mark=at position 0.5 with {\arrow{>}}}
	]
	\draw[postaction={decorate}] (1,0)--(0,1);
	\draw[postaction={decorate}] (-1,0)--(0,-1);
\end{scope}
\begin{scope}[very thick,decoration={
		markings,
		mark=at position 0.45 with {\arrow{>}},
		mark=at position 0.55 with {\arrow{>}}}
	]
	\draw[postaction={decorate}] (0,1)--(-1,0);
	\draw[postaction={decorate}] (0,-1)--(1,0);
\end{scope}
\end{tikzpicture}
\caption{On the left, the shaded square represents a disk neighborhood
of the origin.
On the right, the shaded strip between the two squares represents a
M\"obius band.
Removing a disk and attaching a M\"obius band amounts to
taking a connected sum with $\RR\PP^2$.}
\label{fig:connected_sum_rp2}
\end{figure}
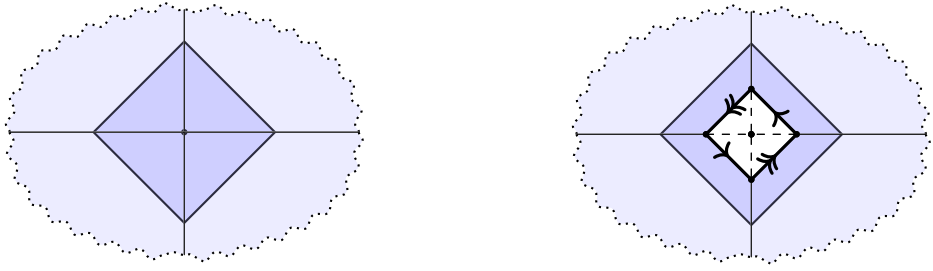

We close this section with another simple application
of the kaleidoscope perspective for toric real loci.

\begin{proposition}
\label{prop:no_spheres}
            Let \((M,\omega,\mu,\tau)\) be a \(2n\)-dimensional toric symplectic manifold equipped with a toric real structure. Then the toric real locus \(M^\tau\) satisfies \(H_{n-1}(M^\tau; \Z/2\Z) \ne 0\). In particular, \(M^\tau\) is not homeomorphic to a sphere when \(n \ge 2\).
        \end{proposition}
        
        \begin{proof}
            The kaleidoscope model induces a CW complex structure on \(M^\tau\) with a single \(n\)-cell and \(2^{n-1}(d-n)\) cells of dimension \(n-1\), where \(d\) is the number of facets of the moment polytope \(\Delta = \mu(M)\).
            Each \((n-1)\)-cell corresponds to a pair of facets in the \(\Delta\)-cluster that are glued together.
            Now consider the sum of all \((n-1)\)-cells.
            This cellular chain is actually a cycle with
            $\Z/2\Z$-coefficients because, in its boundary,
            each \((n-2)\)-cell occurs $2^{n-1}$ times, and this
            is even when $n \geq 2$.
            Moreover, this \((n-1)\)-chain is not a boundary since the cellular boundary of the unique \(n\)-cell vanishes in \(\Z/2\Z\)-coefficients (each \((n-1)\)-cell corresponds to two facets in the boundary of the \(\Delta\)-cluster).
        \end{proof}

        Such a result does not hold with coefficients in $\ZZ$;
        for example, \(H_2(\RP{3};\Z) = 0\).

\begin{remark}[Mimic Philosophy]
The fact that \(H_{n-1}(M^\tau; \Z/2\Z) \ne 0\)
aligns with the following \textit{mimic philosophy},
since $H_{2n-2}(M; \RR) \ne 0$.
The credo that toric real loci topologically
behave in a manner similar to their ambient toric symplectic manifolds,
with the $T$-action replaced by the $\sqrt{\one}$-action,
real coefficients replaced by $\ZZ/2\ZZ$, and degrees halved
goes back to Borel and Haefliger~\cite{BorelHaefliger}
in the context of complex analytic varieties and to Davis and
Januszkiewicz~\cite[Theorem 4.14]{DavisJanuszkiewicz91}
in the context of small covers.
Concretely, there is a degree-doubling ring isomorphism between the
homology with $\ZZ/2\ZZ$ coefficients of a toric manifold and that of its toric real locus.
This is proven in Holm's PhD
thesis~\cite[Thm. 5.5.4 and Cor. 5.5.8.]{Holm02}
in the general context of GKM (and mod 2 GKM) spaces 
and reproven in a different manner by Haug~\cite[Prop. 3.4]{Haug13}.
At the level of vector spaces, this isomorphism was given already by
Duistermaat~\cite[Thm 3.1]{Duistermaat83}.
Goldin and Holm~\cite{GoldinHolm04} further showed that the mimic philosophy
holds in the more general context of symplectic reduction.
Such isomorphisms have been arousing significant interest within
real algebraic geometry, more specifically in
the context of \textit{maximal varieties}; see, for instance,~\cite{Fu}.
\end{remark}


\section{Euler Characteristic and Signature}
\label{sec:euler_signature}

A formula for the Euler characteristic of a toric real locus
follows directly from \cref{prop:kaleidoscope_euler,coroll:kaleidoscope_vs_real_locus}.

\begin{theorem}[Euler Characteristic of Toric Real Locus]
\label{thm:euler_charact}
Let $(M, \omega, \mu)$ be a $2n$-dimensional toric symplectic manifold,
let $\tau$ be a toric real structure on it, and let $\Delta$ be the moment polytope.
Let $a_k$ be the number of $k$-faces in $\Delta$.
Then, the Euler characteristic of the toric real locus is
\[
   \sum_{k=0}^n (-1)^k 2^k a_k \ .
\]
\end{theorem}

As a reality check, from \cref{ex:euler_simplex},
we recover the Euler characteristic of a real
projective space, i.e., the real
locus of $\CC \PP^n$, whose moment polytope is a standard simplex.

\begin{remark}[Parity of Euler Characteristic]
\Cref{thm:euler_charact} confirms that
the parity of the Euler characteristic of a toric real locus
is the same as that of the number $a_0$ of vertices in $\Delta$, which is
equal to the Euler characteristic of the underlying toric symplectic manifold.
This coincidence of parities may already be deduced from Duistermaat's result~\cite{Duistermaat83} or from Smith theory and
has been observed multiple times;
see, for instance, \cite[Thm III.4.3]{Bredon1972} or \cite{Zagier90}.

Since the Euler characteristic of any odd-dimensional compact manifold vanishes,
we conclude that $\sum_{k=0}^n (-1)^k 2^k a_k = 0$ when $n$ is odd,
and, in particular, that $a_0$ must be even.
(The fact that $a_0$ must be even when $n$ is odd also follows from
Poincar\'e duality for the toric symplectic manifold.)
\end{remark}

The first few cases are easy to work out in detail.
\begin{itemize}
    \item
    When $n=2$, a polygon $\Delta$ has $a_1=a_0$; hence, the 
toric real locus is a surface with Euler characteristic
\[
   \chi (M^\tau) = a_0 - 2 a_0 + 2^2 = 4 - a_0 = 4 - \chi (M) \ .
\]
\item
When $n=3$, a unimodular polytope $\Delta$ has
$a_1 = \frac 32 a_0$ and $a_2 = 2 + \frac 12 a_0$; hence, the 
Euler characteristic of the toric real locus is
\[
   \chi (M^\tau) = a_0 - 3 a_0 + (8+2a_0) - 8 = 0\ ,
\]
as expected for any odd-dimensional compact manifold.
\item
When $n=4$, a unimodular polytope $\Delta$ has
$a_1 = 2 a_0$ and $a_2 = a_0 + a_3$; hence, the 
Euler characteristic of the toric real locus is
\[
   \chi (M^\tau) = a_0 - 4 a_0 + (4a_0+4a_3) - 8a_3 + 16 = a_0 - 4 a_3 + 16 \ ,
\]
and it already depends on the number of facets in addition to the number of vertices.
\end{itemize}

\begin{example}
\label{ex:two_fano}
A pair of unimodular polytopes with $n=4$, with the same number of vertices 
but different numbers of 3-faces may be found within Fano polytopes.
For instance, let $Z_1$ and $K_4$ be the toric Fano manifolds listed with
the numbers 122 and 104 by Batyrev~\cite[\S 4]{Batyrev1999}, respectively.
These manifolds have the following Betti numbers $b_2$ and $b_4$,
while their corresponding moment polytopes, $\Delta_{_{Z_1}}$ and $\Delta_{_{K_4}}$,
have the following numbers $a_0$ and $a_3$ of vertices and facets.
The concrete moment polytopes may be found with the help of~\cite{ObroSite}.
\[
\begin{array}{|c|c|c||c|c|c|}
\hline
  & b_2 & b_4 & & a_0 & a_3  \\
\hline
Z_1 & 4 & 8 & \Delta_{_{Z_1}} & 18 & 8 \\
\hline
K_4 & 5 & 6 & \Delta_{_{K_4}} & 18 & 9 \\
\hline
\end{array}
\]
It follows that the Euler characteristic of these toric symplectic
manifolds is the same $\chi (Z_1) = \chi (K_4) = 18$,
yet it is different for their toric real loci, namely
$\chi (Z_1^\tau) = 2$ and  $\chi (K_4^\tau) = -2$.
For instance, $K_4$ is the product of $\CC\PP^2$ with
$\CC\PP^2\#_3 \overline{\CC\PP^2}$; hence, by \cref{rmk:product,prop:kaleidoscope_blowup}, its toric real locus
$K_4^\tau$ is $\RR\PP^2 \times \#_4 \RR\PP^2$,
which confirms $\chi (K_4^\tau) = 1 \times (2-4) = -2$.
\end{example}


\Cref{thm:euler_charact} allows a simple alternative proof for the
congruence properties of the Euler characteristic with the signature 
in the case where the dimension is a multiple of four.
Whereas the signature $\sigma(M)$ of a compact oriented manifold of real dimension $4m$
always satisfies
\[
   \sigma (M) \equiv \chi (M) \mod 2
\]
(see, for instance, \cite[p.~164]{May}), this congruence is upgraded
to the statement in the following corollary, first pointed out
by Hirzebruch~\cite[p.777]{Hirzebruch_Coll_Works_I_II}
for the more general case where the manifold admits an almost complex structure (see also \cite[Theorem 1.3]{ADG}).

\begin{corollary}[Euler Characteristic vs.\ Signature of Toric Manifolds for Dimensions $4m$]
    \label{coroll:signature}
    Let $(M, \omega, \mu)$ be a $4m$-dimensional toric symplectic manifold.
    Then
\[
   \sigma (M) \equiv (-1)^m \chi (M) \mod 4.
\]
\end{corollary}

\begin{proof}
    For a simple polytope $\Delta$,
    the number $a_1$ of edges and the number $a_0$ of vertices
    are related by $a_1 = \frac{1}{2} a_0 \cdot \dim \Delta$.
    By \Cref{thm:euler_charact}, we obtain, in the case where $\dim \Delta = 2m$, that
    \[
    \chi (M^\tau) \equiv a_0-2a_1 = a_0 - 2m a_0 \mod 4.
    \]
    By the fact that $\chi(M) = a_0$,
    we conclude when $m$ is even that $\chi (M^\tau) \equiv \chi(M) \mod 4$,
    and when $m$ is odd that $\chi (M^\tau) \equiv -\chi(M) \mod 4$, hence
\[
   \chi (M^\tau) \equiv (-1)^m \chi(M) \mod 4.
\]
On the other hand, in the toric case, we have~\cite[Satz 6]{Ehlers75}
\[
\sigma(M) = \chi (M^\tau) .
\]
For completeness, we next point out the fundamental facts that lead to this last equality.

First, by the Hodge index theorem, the signature of a compact
K\"ahler manifold $M$ of even complex dimension $n$ is
\[
\sigma(M) = \sum_{p,q=0}^{n} (-1)^p\, h^{p,q}(M),
\]
where $h^{p,q}(M) = \dim H^{p,q}(M)$ are the Hodge numbers
(see, for instance, \cite[Coroll.\ 3.3.18]{Huybrechts}).

Second, for a toric $M$, all off-diagonal Hodge numbers,
$h^{p,q}(M)$ with $p \neq q$, vanish
(see, for instance, \cite[Thm.\ 9.4.11]{CoxLittleSchenck}).
Therefore, the Betti numbers are $b_{2p} (M) = h^{p,p}(M)$.

Third, because $M$ has no torsion
(see, for instance,~\cite[\S 5.2]{Fulton}) and
there is an isomorphism
$H_{2p} (M;\ZZ / 2\ZZ) \cong H_{p} (M^\tau ; \ZZ / 2\ZZ)$
(see, \cite[Coroll.\ 5.8]{BGH04}),
we have
\[
b_{2p} (M) = \dim H_{2p} (M;\ZZ / 2\ZZ) = \dim H_{p} (M^\tau ; \ZZ / 2\ZZ).
\]

Finally, the Euler characteristic
of a compact manifold
is independent of the choice of the coefficient field for homology
(see, for instance, \cite[Ex.~41 of \S 2.2]{HatcherBook}),
and thus\footnote{The observation that
$\sigma(M) = \sum_{p=0}^{n} (-1)^p \, b_{2p} (M)$ may be found
also in Metzler's work \cite[p.3519]{Metzler2000}.}
\[
\chi (M^\tau) = \sum_{p=0}^{n} (-1)^p \, \dim H_{p} (M^\tau ; \ZZ / 2\ZZ)
= \sum_{p=0}^{n} (-1)^p \, b_{2p} (M) = \sum_{p=0}^{n} (-1)^p h^{p,p}(M) = \sigma(M).\qedhere
\]
\end{proof}

For reference, we display in \Cref{table:124euler} the Euler characteristic $\chi$ and signature $\sigma$
for each of the 124 smooth toric Fano fourfolds, noting that $\sigma(M) = \chi (M^\tau)$.

\bgroup
\def\arraystretch{1.2}
\begin{table} 
		\begin{minipage}[t]{0.2\textwidth}
			\begin{tabular}{|c|c|c||c|c|}
            \hline
             & $b_2$ & $b_4$ & $\chi$ & $\sigma$ \\
            \hline
            \(\mathbb{P}^4\) & 1 & 1 & 5 & 1 \\
            \hline
            \(B_1\) & 2 & 2 & 8 & 0 \\
            \hline
            \(B_2\) & 2 & 2 & 8 & 0 \\
            \hline
            \(B_3\) & 2 & 2 & 8 & 0 \\
            \hline
            \(B_4\) & 2 & 2 & 8 & 0 \\
            \hline
            \(B_5\) & 2 & 2 & 8 & 0 \\
            \hline
            \(C_1\) & 2 & 3 & 9 & 1 \\
            \hline
            \(C_2\) & 2 & 3 & 9 & 1 \\
            \hline
            \(C_3\) & 2 & 3 & 9 & 1 \\
            \hline
            \(C_4\) & 2 & 3 & 9 & 1 \\
            \hline
            \(E_1\) & 3 & 3 & 11 & -1 \\
            \hline
            \(E_2\) & 3 & 3 & 11 & -1 \\
            \hline
            \(E_3\) & 3 & 3 & 11 & -1 \\
            \hline
            \(D_{1}\) & 3 & 4 & 12 & 0 \\
            \hline
            \(D_{2}\) & 3 & 4 & 12 & 0 \\
            \hline
            \(D_{3}\) & 3 & 4 & 12 & 0 \\
            \hline
            \(D_{4}\) & 3 & 4 & 12 & 0 \\
            \hline
            \(D_{5}\) & 3 & 4 & 12 & 0 \\
            \hline
            \(D_{6}\) & 3 & 4 & 12 & 0 \\
            \hline
            \(D_{7}\) & 3 & 4 & 12 & 0 \\
            \hline
            \(D_{8}\) & 3 & 4 & 12 & 0 \\
            \hline
            \(D_{9}\) & 3 & 4 & 12 & 0 \\
            \hline
            \(D_{10}\) & 3 & 4 & 12 & 0 \\
            \hline
            \(D_{11}\) & 3 & 4 & 12 & 0 \\
            \hline
            \(D_{12}\) & 3 & 4 & 12 & 0 \\
            \hline
            \(D_{13}\) & 3 & 4 & 12 & 0 \\
            \hline
            \(D_{14}\) & 3 & 4 & 12 & 0 \\
            \hline
            \(D_{15}\) & 3 & 4 & 12 & 0 \\
            \hline
            \(D_{16}\) & 3 & 4 & 12 & 0 \\
            \hline
            \(D_{17}\) & 3 & 4 & 12 & 0 \\
            \hline
            \(D_{18}\) & 3 & 4 & 12 & 0 \\
            \hline
            \end{tabular}
		\end{minipage}\hfill
		\begin{minipage}[t]{0.2\textwidth}
			\begin{tabular}{|c|c|c||c|c|}
            \hline
             & $b_2$ & $b_4$ & $\chi$ & $\sigma$ \\
            \hline
            \(D_{19}\) & 3 & 4 & 12 & 0 \\
            \hline
            \(G_1\) & 3 & 5 & 13 & 1 \\
            \hline
            \(G_2\) & 3 & 5 & 13 & 1 \\
            \hline
            \(G_3\) & 3 & 5 & 13 & 1 \\
            \hline
            \(G_4\) & 3 & 5 & 13 & 1 \\
            \hline
            \(G_5\) & 3 & 5 & 13 & 1 \\
            \hline
            \(G_6\) & 3 & 5 & 13 & 1 \\
            \hline
            \(H_{1}\) & 4 & 5 & 15 & -1 \\
            \hline
            \(H_{2}\) & 4 & 5 & 15 & -1 \\
            \hline
            \(H_{3}\) & 4 & 5 & 15 & -1 \\
            \hline
            \(H_{4}\) & 4 & 5 & 15 & -1 \\
            \hline
            \(H_{5}\) & 4 & 5 & 15 & -1 \\
            \hline
            \(H_{6}\) & 4 & 5 & 15 & -1 \\
            \hline
            \(H_{7}\) & 4 & 5 & 15 & -1 \\
            \hline
            \(H_{8}\) & 4 & 5 & 15 & -1 \\
            \hline
            \(H_{9}\) & 4 & 5 & 15 & -1 \\
            \hline
            \(H_{10}\) & 4 & 5 & 15 & -1 \\
            \hline
            \(L_{1}\) & 4 & 6 & 16 & 0 \\
            \hline
            \(L_{2}\) & 4 & 6 & 16 & 0 \\
            \hline
            \(L_{3}\) & 4 & 6 & 16 & 0 \\
            \hline
            \(L_{4}\) & 4 & 6 & 16 & 0 \\
            \hline
            \(L_{5}\) & 4 & 6 & 16 & 0 \\
            \hline
            \(L_{6}\) & 4 & 6 & 16 & 0 \\
            \hline
            \(L_{7}\) & 4 & 6 & 16 & 0 \\
            \hline
            \(L_{8}\) & 4 & 6 & 16 & 0 \\
            \hline
            \(L_{9}\) & 4 & 6 & 16 & 0 \\
            \hline
            \(L_{10}\) & 4 & 6 & 16 & 0 \\
            \hline
            \(L_{11}\) & 4 & 6 & 16 & 0 \\
            \hline
            \(L_{12}\) & 4 & 6 & 16 & 0 \\
            \hline
            \(L_{13}\) & 4 & 6 & 16 & 0 \\
            \hline
            \(I_{1}\) & 4 & 6 & 16 & 0 \\
            \hline
            \end{tabular}
		\end{minipage}\hfill
		\begin{minipage}[t]{0.2\textwidth}
			\begin{tabular}{|c|c|c||c|c|}
            \hline
             & $b_2$ & $b_4$ & $\chi$ & $\sigma$ \\
            \hline
            \(I_{2}\) & 4 & 6 & 16 & 0 \\
            \hline
            \(I_{3}\) & 4 & 6 & 16 & 0 \\
            \hline
            \(I_{4}\) & 4 & 6 & 16 & 0 \\
            \hline
            \(I_{5}\) & 4 & 6 & 16 & 0 \\
            \hline
            \(I_{6}\) & 4 & 6 & 16 & 0 \\
            \hline
            \(I_{7}\) & 4 & 6 & 16 & 0 \\
            \hline
            \(I_{8}\) & 4 & 6 & 16 & 0 \\
            \hline
            \(I_{9}\) & 4 & 6 & 16 & 0 \\
            \hline
            \(I_{10}\) & 4 & 6 & 16 & 0 \\
            \hline
            \(I_{11}\) & 4 & 6 & 16 & 0 \\
            \hline
            \(I_{12}\) & 4 & 6 & 16 & 0 \\
            \hline
            \(I_{13}\) & 4 & 6 & 16 & 0 \\
            \hline
            \(I_{14}\) & 4 & 6 & 16 & 0 \\
            \hline
            \(I_{15}\) & 4 & 6 & 16 & 0 \\
            \hline
            \(M_1\) & 4 & 7 & 17 & 1 \\
            \hline
            \(M_2\) & 4 & 7 & 17 & 1 \\
            \hline
            \(M_3\) & 4 & 7 & 17 & 1 \\
            \hline
            \(M_4\) & 4 & 7 & 17 & 1 \\
            \hline
            \(M_5\) & 4 & 7 & 17 & 1 \\
            \hline
            \(J_1\) & 4 & 7 & 17 & 1 \\
            \hline
            \(J_2\) & 4 & 7 & 17 & 1 \\
            \hline
            \(Q_{1}\) & 5 & 8 & 20 & 0 \\
            \hline
            \(Q_{2}\) & 5 & 8 & 20 & 0 \\
            \hline
            \(Q_{3}\) & 5 & 8 & 20 & 0 \\
            \hline
            \(Q_{4}\) & 5 & 8 & 20 & 0 \\
            \hline
            \(Q_{5}\) & 5 & 8 & 20 & 0 \\
            \hline
            \(Q_{6}\) & 5 & 8 & 20 & 0 \\
            \hline
            \(Q_{7}\) & 5 & 8 & 20 & 0 \\
            \hline
            \(Q_{8}\) & 5 & 8 & 20 & 0 \\
            \hline
            \(Q_{9}\) & 5 & 8 & 20 & 0 \\
            \hline
            \(Q_{10}\) & 5 & 8 & 20 & 0 \\
            \hline
            \end{tabular}
		\end{minipage}\hfill
		\begin{minipage}[t]{0.2\textwidth}
			\begin{tabular}{|c|c|c||c|c|}
            \hline
             & $b_2$ & $b_4$ & $\chi$ & $\sigma$ \\
            \hline
            \(Q_{11}\) & 5 & 8 & 20 & 0 \\
            \hline
            \(Q_{12}\) & 5 & 8 & 20 & 0 \\
            \hline
            \(Q_{13}\) & 5 & 8 & 20 & 0 \\
            \hline
            \(Q_{14}\) & 5 & 8 & 20 & 0 \\
            \hline
            \(Q_{15}\) & 5 & 8 & 20 & 0 \\
            \hline
            \(Q_{16}\) & 5 & 8 & 20 & 0 \\
            \hline
            \(Q_{17}\) & 5 & 8 & 20 & 0 \\
            \hline
            \(K_1\) & 5 & 6 & 18 & -2 \\
            \hline
            \(K_2\) & 5 & 6 & 18 & -2 \\
            \hline
            \(K_3\) & 5 & 6 & 18 & -2 \\
            \hline
            \(K_4\) & 5 & 6 & 18 & -2 \\
            \hline
            \(R_1\) & 5 & 9 & 21 & 1 \\
            \hline
            \(R_2\) & 5 & 9 & 21 & 1 \\
            \hline
            \(R_3\) & 5 & 9 & 21 & 1 \\
            \hline
            \(\) & 5 & 9 & 21 & 1 \\
            \hline
            \(U_1\) & 6 & 10 & 24 & 0 \\
            \hline
            \(U_2\) & 6 & 10 & 24 & 0 \\
            \hline
            \(U_3\) & 6 & 10 & 24 & 0 \\
            \hline
            \(U_4\) & 6 & 10 & 24 & 0 \\
            \hline
            \(U_5\) & 6 & 10 & 24 & 0 \\
            \hline
            \(U_6\) & 6 & 10 & 24 & 0 \\
            \hline
            \(U_7\) & 6 & 10 & 24 & 0 \\
            \hline
            \(U_8\) & 6 & 10 & 24 & 0 \\
            \hline
            \(\tilde{V}^4\) & 5 & 11 & 23 & 3 \\
            \hline
            \(V^4\) & 6 & 16 & 30 & 6 \\
            \hline
            \(S_2 \times S_2\) & 6 & 11 & 25 & 1 \\
            \hline
            \(S_2 \times S_3\) & 7 & 14 & 30 & 2 \\
            \hline
            \(S_3 \times S_3\) & 8 & 18 & 36 & 4 \\
            \hline
            \(Z_1\) & 4 & 8 & 18 & 2 \\
            \hline
            \(Z_2\) & 4 & 8 & 18 & 2 \\
            \hline
            \(W\) & 5 & 12 & 24 & 4 \\
            \hline
            \end{tabular}
		\end{minipage}
        \vspace{2ex}
        \caption{The Euler characteristic $\chi$ and signature $\sigma$ of each of the 124 smooth toric Fano fourfolds, following the notation from \cite{Batyrev1999,Sato}.}
        \label{table:124euler}
    	\end{table}
\egroup

\begin{remark}
The signature of an orientable toric real locus always vanishes:
\[
\sigma (M^\tau) = 0.
\]
This is because these manifolds
always admit an orientation-reversing diffeomorphism,
for instance, that given by the element $(-1,1,1,\ldots,1)$
of $\sqrt{\one}$.
(If the dimension is not a multiple of $4$, the signature vanishes
by definition; hence, this observation is only relevant for those
whose dimension is a multiple of 4.) 
\end{remark}


\section{Orientability of Toric Real Loci}
\label{sec:orientability}

A simple criterion for the orientability of a toric real locus
in terms of the moment polytope follows directly from
the orientability of a kaleidoscope (\cref{prop:kaleidoscope_orientability})
and the fact that the toric real locus and the kaleidoscope are
homeomorphic (\cref{coroll:kaleidoscope_vs_real_locus}):

\begin{theorem}[Orientability of Toric Real Locus]
\label{thm:orientability}
Let $(M, \omega, \mu, \tau)$ be a toric symplectic manifold
equipped with a toric real structure.
Choose an integral lattice basis so that the moment polytope,
$\Delta$, has a standard vertex.
The toric real locus is orientable if and only if
each primitive normal vector to a facet of $\Delta$
has an odd number of odd entries.
\end{theorem}

Thanks to the assumption of a standard vertex, the orientability
check becomes as simple as counting odd entries in the primitive normal
vectors that are not parallel to the coordinate axes.

As a reality check, we rediscover that $\RR\PP^n$, i.e., the real
locus of $\CC \PP^n$ whose moment polytope is a standard simplex,
is orientable exactly when $n$ is odd, since that is when
the $n$-dimensional vector $(1,1,\ldots,1)$ has an odd number of odd entries.
The case $n=3$ is illustrated on the left-side of \cref{fig:kaleidoscope_cp3}.


The orientability criterion in the next corollary
follows directly from \Cref{prop:kaleidoscope_blowup}
and the orientability of $\RR\PP^n$.
Alternatively, it can be understood from the description of
the kaleidoscope after blow-up.

\begin{corollary}
\label{coroll:orientability_blowup}
When $n$ is even, the toric real locus of a blow-up at a fixed point
of any $2n$-dimensional toric symplectic manifold is never orientable.
When $n$ is odd, a blow-up at a fixed point
does not change the orientability of the toric real locus.
\end{corollary}


\begin{remark}[Relevance of Orientability]
The orientability of the toric real locus is relevant, for instance, in
K\"ahler geometry, where the restriction of the riemannian metric
to the toric real locus may yield interesting global invariants of the
original toric manifold.
When the \emph{Leitmotif} ``how much does the spectrum of (the Laplacian of) a metric
determine the manifold'' is stretched to the spectrum of the metric restricted
to a toric real locus, the lack of orientability may
require an orientable double cover to be considered in order to recover
certain features of Hodge theory.
Moreover, the conformal class of the induced metric on the toric real locus
is an interesting invariant of the original K\"ahler metric on the toric symplectic manifold.

When $n=2$, the conformal structure produces a Riemann surface
structure on the oriented cover of the toric real locus.
In this context, there was some confusion in the literature regarding when
the toric real locus is actually orientable: for $n=2$,
the torus occurs as the toric real locus not only when the moment polytope is
a rectangle, as claimed in \cite[\S 2.2.4]{Donaldson},
but also when the moment polytope is another even Hirzebruch surface,
as clarified in \Cref{coroll:4dim}.
\end{remark}

\begin{example}
Within smooth Fano polytopes with $n=2$ (i.e., polytopes corresponding
to smooth toric Fano varieties of complex dimension 2, a.k.a.\ smooth toric del Pezzo surfaces), we find only one
orientable kaleidoscope, namely that of the polytope
corresponding to the manifold $\CC\PP^1 \times \CC\PP^1$,
whose toric real locus is a torus.
This is illustrated in \cref{fig:fano2}, where the orientable case
is the second from left (square polytope).
The general case of 4-dimensional toric symplectic manifolds is
tackled in \cref{sec:case_n=2}. 
\end{example}

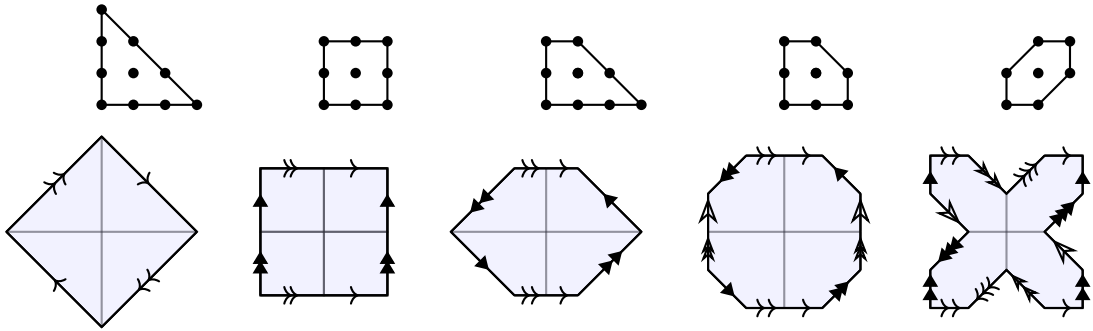
\begin{figure}[ht]
\begin{tikzpicture}[scale=0.42]
  \tikzstyle{every path}=[thick]
  \tikzstyle{pt}=[circle, fill=black, inner sep=2pt]

  \newcommand{\drawdots}[1]{
    \foreach \p in {#1} \node[pt,scale=0.7] at \p {};
  }

  \begin{scope}[xshift=0cm]
    \draw (0,0) -- (3,0) -- (0,3) -- cycle;
    \drawdots{(0,0),(1,0),(2,0),(3,0),(0,1),(0,2),(0,3),(1,2),(2,1)}
    \drawdots{(1,1)}
  \end{scope}

  \begin{scope}[xshift=0cm,yshift=-4cm]
\begin{scope}[very thin]
\draw[opacity=.5] (-3,0) -- (3,0);
\draw[opacity=.5] (0,-3) -- (0,3);
\end{scope}
\filldraw[fill=blue!20,fill opacity=0.25] (3,0) -- (0,3) -- (-3,0) -- (0,-3) -- cycle;
\begin{scope}[very thick,decoration={markings,
mark=at position 0.5 with {\arrow{>}}}]
\draw[postaction={decorate}] (0,3)--(3,0);
\draw[postaction={decorate}] (0,-3)--(-3,0);
\end{scope}
\begin{scope}[very thick,decoration={markings,
mark=at position 0.45 with {\arrow{<}},
mark=at position 0.55 with {\arrow{<}}}]
\draw[postaction={decorate}] (0,3)--(-3,0);
\draw[postaction={decorate}] (0,-3)--(3,0);
\end{scope}
  \end{scope}

  \begin{scope}[xshift=7cm]
    \draw (0,0) -- (2,0) -- (2,2) -- (0,2) -- cycle;
    \drawdots{(0,0),(1,0),(2,0),(2,1),(2,2),(1,2),(0,2),(0,1)}
    \drawdots{(1,1)}
  \end{scope}

  \begin{scope}[xshift=7cm,yshift=-4cm]
\begin{scope}[very thin]
    \draw[opacity=.5] (0,0) -- (2,0) -- (2,2) -- (0,2) -- cycle;
    \draw[opacity=.5] (0,0) -- (-2,0) -- (-2,2) -- (0,2) -- cycle;
    \draw[opacity=.5] (0,0) -- (2,0) -- (2,-2) -- (0,-2) -- cycle;
    \draw[opacity=.5] (0,0) -- (-2,0) -- (-2,-2) -- (0,-2) -- cycle;
\end{scope}
\filldraw[fill=blue!20,fill opacity=0.25] (2,2) -- (2,-2) -- (-2,-2) -- (-2,2) -- cycle;
\begin{scope}[very thick,decoration={markings,
mark=at position 0.55 with {\arrow{>}}}]
\draw[postaction={decorate}] (0,2)--(2,2);
\draw[postaction={decorate}] (0,-2)--(2,-2);
\end{scope}
\begin{scope}[very thick,decoration={markings,
mark=at position 0.5 with {\arrow{>}},
mark=at position 0.6 with {\arrow{>}}}]
\draw[postaction={decorate}] (-2,2)--(0,2);
\draw[postaction={decorate}] (-2,-2)--(0,-2);
\end{scope}
\begin{scope}[thick,decoration={markings,
mark=at position 0.6 with {\arrow{Triangle[angle=60:2mm]}}}]
\draw[postaction={decorate}] (2,0) -- (2,2);
\draw[postaction={decorate}] (-2,0) -- (-2,2);
\end{scope}
\begin{scope}[thick,decoration={markings,
mark=at position 0.55 with {\arrow{Triangle[angle=60:2mm]}},
mark=at position 0.7 with {\arrow{Triangle[angle=60:2mm]}}}]
\draw[postaction={decorate}] (2,-2) -- (2,0);
\draw[postaction={decorate}] (-2,-2) -- (-2,0);
\end{scope}
  \end{scope}

  \begin{scope}[xshift=14cm]
    \draw (0,0) -- (3,0) -- (1,2) -- (0,2) -- cycle;
    \drawdots{(0,0),(1,0),(2,0),(3,0),(0,1),(0,2),(1,2),(2,1),(1,1)}
    \drawdots{(1,1)}
  \end{scope}

  \begin{scope}[xshift=14cm,yshift=-4cm]
\begin{scope}[very thin]
\draw[opacity=.5] (-3,0) -- (3,0);
\draw[opacity=.5] (0,-2) -- (0,2);
\end{scope}
\filldraw[fill=blue!20,fill opacity=0.25] (-3,0) -- (-1,2) -- (1,2) -- (3,0) -- (1,-2) -- (-1,-2) -- cycle;
\begin{scope}[very thick,decoration={markings,
mark=at position 0.7 with {\arrow{>}}}]
\draw[postaction={decorate}] (0,2)--(1,2);
\draw[postaction={decorate}] (0,-2)--(1,-2);
\end{scope}
\begin{scope}[very thick,decoration={markings,
mark=at position 0.5 with {\arrow{>}},
mark=at position 0.8 with {\arrow{>}}}]
\draw[postaction={decorate}] (-1,2)--(0,2);
\draw[postaction={decorate}] (-1,-2)--(0,-2);
\end{scope}
\begin{scope}[thick,decoration={markings,
mark=at position 0.6 with {\arrow{Triangle[angle=60:2mm]}}}]
\draw[postaction={decorate}] (3,0) -- (1,2);
\draw[postaction={decorate}] (-3,0) -- (-1,-2);
\end{scope}
\begin{scope}[thick,decoration={markings,
mark=at position 0.55 with {\arrow{Triangle[angle=60:2mm]}},
mark=at position 0.7 with {\arrow{Triangle[angle=60:2mm]}}}]
\draw[postaction={decorate}] (-1,2) -- (-3,0);
\draw[postaction={decorate}] (1,-2) -- (3,0);
\end{scope}
  \end{scope}

  \begin{scope}[xshift=21.5cm]
    \draw (0,0) -- (2,0) -- (2,1) -- (1,2) -- (0,2) -- cycle;
    \drawdots{(0,0),(1,0),(2,0),(2,1),(1,2),(0,2),(0,1),(1,1)}
    \drawdots{(1,1)}
  \end{scope}

  \begin{scope}[xshift=21.5cm,yshift=-4cm,scale=1.2]
\begin{scope}[very thin]
\draw[opacity=.5] (-2,0) -- (2,0);
\draw[opacity=.5] (0,-2) -- (0,2);
\end{scope}
\filldraw[fill=blue!20,fill opacity=0.25] (-2,0) -- (-2,1) -- (-1,2) -- (1,2) -- (2,1) -- (2,-1) -- (1,-2) -- (-1,-2) -- (-2,-1) -- cycle;
\begin{scope}[very thick,decoration={markings,
mark=at position 0.7 with {\arrow{>}}}]
\draw[postaction={decorate}] (0,2)--(1,2);
\draw[postaction={decorate}] (0,-2)--(1,-2);
\end{scope}
\begin{scope}[very thick,decoration={markings,
mark=at position 0.5 with {\arrow{>}},
mark=at position 0.8 with {\arrow{>}}}]
\draw[postaction={decorate}] (-1,2)--(0,2);
\draw[postaction={decorate}] (-1,-2)--(0,-2);
\end{scope}
\begin{scope}[very thick,decoration={markings,
mark=at position 0.9 with {\arrow{Stealth[length=3.5mm, open]}}}]
\draw[postaction={decorate}] (2,0)--(2,1);
\draw[postaction={decorate}] (-2,0)--(-2,1);
\end{scope}
\begin{scope}[very thick,decoration={markings,
mark=at position 0.7 with {\arrow{Stealth[length=2.5mm, open]}},
mark=at position 0.9 with {\arrow{Stealth[length=2.5mm, open]}}}]
\draw[postaction={decorate}] (2,-1)--(2,0);
\draw[postaction={decorate}] (-2,-1)--(-2,0);
\end{scope}
\begin{scope}[thick,decoration={markings,
mark=at position 0.7 with {\arrow{Triangle[angle=60:2mm]}}}]
\draw[postaction={decorate}] (2,1)--(1,2);
\draw[postaction={decorate}] (-2,-1)--(-1,-2);
\end{scope}
\begin{scope}[thick,decoration={markings,
mark=at position 0.55 with {\arrow{Triangle[angle=60:2mm]}},
mark=at position 0.7 with {\arrow{Triangle[angle=60:2mm]}}}]
\draw[postaction={decorate}] (1,-2)--(2,-1);
\draw[postaction={decorate}] (-1,2)--(-2,1);
\end{scope}
  \end{scope}

  \begin{scope}[xshift=28.5cm]
    \draw (0,0) -- (0,1) -- (1,2) -- (2,2) -- (2,1) -- (1,0) -- cycle;
    \drawdots{(0,0),(0,1),(1,2),(2,2),(2,1),(1,0)}
    \drawdots{(1,1)}
  \end{scope}

  \begin{scope}[xshift=28.5cm,yshift=-4cm,scale=1.2]
\begin{scope}[very thin]
\draw[opacity=.5] (-1,0) -- (1,0);
\draw[opacity=.5] (0,-1) -- (0,1);
\end{scope}
\filldraw[fill=blue!20,fill opacity=0.25] (-1,0) -- (-2,1) -- (-2,2) -- (-1,2) -- (0,1) -- (1,2) -- (2,2) -- (2,1) -- (1,0) -- (2,-1) -- (2,-2) -- (1,-2) -- (0,-1) -- (-1,-2) -- (-2,-2) -- (-2,-1) -- cycle;
\begin{scope}[very thick,decoration={markings,
mark=at position 0.7 with {\arrow{>}}}]
\draw[postaction={decorate}] (1,2)--(2,2);
\draw[postaction={decorate}] (1,-2)--(2,-2);
\end{scope}
\begin{scope}[very thick,decoration={markings,
mark=at position 0.5 with {\arrow{>}},
mark=at position 0.8 with {\arrow{>}}}]
\draw[postaction={decorate}] (-2,2)--(-1,2);
\draw[postaction={decorate}] (-2,-2)--(-1,-2);
\end{scope}
\begin{scope}[very thick,decoration={markings,
mark=at position 0.5 with {\arrow{>}},
mark=at position 0.65 with {\arrow{>}},
mark=at position 0.8 with {\arrow{>}}}]
\draw[postaction={decorate}] (0,1)--(1,2);
\draw[postaction={decorate}] (0,-1)--(-1,-2);
\end{scope}
\begin{scope}[thick,decoration={markings,
mark=at position 0.6 with {\arrow{Triangle[angle=60:2mm]}}}]
\draw[postaction={decorate}] (2,1) -- (2,2);
\draw[postaction={decorate}] (-2,1) -- (-2,2);
\end{scope}
\begin{scope}[thick,decoration={markings,
mark=at position 0.55 with {\arrow{Triangle[angle=60:2mm]}},
mark=at position 0.9 with {\arrow{Triangle[angle=60:2mm]}}}]
\draw[postaction={decorate}] (2,-2) -- (2,-1);
\draw[postaction={decorate}] (-2,-2) -- (-2,-1);
\end{scope}
\begin{scope}[thick,decoration={markings,
mark=at position 0.5 with {\arrow{Triangle[angle=60:2mm]}},
mark=at position 0.65 with {\arrow{Triangle[angle=60:2mm]}},
mark=at position 0.8 with {\arrow{Triangle[angle=60:2mm]}}}]
\draw[postaction={decorate}] (1,0) -- (2,1);
\draw[postaction={decorate}] (-1,0) -- (-2,-1);
\end{scope}
\begin{scope}[very thick,decoration={markings,
mark=at position 0.8 with {\arrow{Stealth[length=3.5mm, open]}}}]
\draw[postaction={decorate}] (2,-1)--(1,0);
\draw[postaction={decorate}] (-2,1)--(-1,0);
\end{scope}
\begin{scope}[very thick,decoration={markings,
mark=at position 0.6 with {\arrow{Stealth[length=2.5mm, open]}},
mark=at position 0.9 with {\arrow{Stealth[length=2.5mm, open]}}}]
\draw[postaction={decorate}] (-1,2)--(0,1);
\draw[postaction={decorate}] (1,-2)--(0,-1);
\end{scope}
  \end{scope}

\end{tikzpicture}
\caption{The moment polytopes for the smooth toric Fano varieties with $n=2$,
and the corresponding kaleidoscopes.  The columns from left to right represent:
$\CC\PP^2$ with toric real locus $\RR \PP^2$,
$\CC\PP^1 \times \CC\PP^1$ with toric real locus the torus $S^1 \times S^1$,
$\CC\PP^2 \# \overline{\CC\PP^2}$
with toric real locus the Klein bottle $\RR \PP^2 \# \RR \PP^2$,
$\CC\PP^2 \#_2 \overline{\CC\PP^2}$
with toric real locus $\#_3 \RR \PP^2$, and
$\CC\PP^2 \#_3 \overline{\CC\PP^2}$ with toric real locus $\#_4 \RR \PP^2$.}
\label{fig:fano2}
\end{figure}

We easily find examples among higher-dimensional smooth Fano polytopes 
thanks to the data obtained by \O bro \cite{Obro2007} and hosted by the Graded Ring Database \cite{GradedRingDatabase,ObroSite}. 
This database provides a list of isomorphism classes of smooth Fano
$n$-polytopes for each dimension $n \leq 6$.
However, the given polytopes on this list are the duals of our
moment polytopes \cite[Def. 2.1]{Obro2007}, so the numbers of
vertices and facets are swapped: the listed vertices give
the primitive inward-pointing normal vectors to our facets,
and the points listed under ``Dual'' are the vertices of our moment polytopes.
The representatives of the isomorphism classes in \O bro's algorithm
correspond to what he calls \emph{special embeddings}, which
imply our \emph{standard vertex} in the dual polytope
\cite[Def. 3.4 and \S 5.3]{Obro2007}.
All in all, we are left with the task of going through Øbro's list
looking for the polytopes all of whose listed vertices
(i.e.\ primitive normal vectors for us)
have an odd number of odd entries.
Already other databases handle some higher dimensions.

\begin{remark}[Proportion of Orientable Toric Real Loci in Smooth Fano Case]
The Fano examples illustrate the expectation that,
with growing dimension, it should become rare to find orientable
toric real loci, since it is enough to have one primitive normal
vector with an even number of odd entries for orientability to fail.
With the help of \emph{Sage}, we analyzed data files provided by
\O bro~\cite{ObroArchived} and Paffenholz~\cite{PaffenholzSite} in order to count, for each \(n \le 9\), 
the number of smooth toric Fano polytopes that give rise to orientable kaleidoscopes.
\cref{table:fano_orientable_real_loci} lists the percentage of smooth toric Fano
$n$-folds having orientable toric real loci for $n$ up to $9$.
\end{remark}

\begin{table}
\centering
\begin{tabular}{|c|c|c|c|}
\hline
\begin{minipage}{1cm}
\centering $n$
\end{minipage} &
\begin{minipage}{2.5cm}\begin{center}
$\phantom{x}$\\ Smooth toric\\ Fano\\ $n$-polytopes\\ $\phantom{x}$
\end{center}\end{minipage} &
\begin{minipage}{2.5cm}\begin{center}
Those giving\\ orientable\\ toric real loci
\end{center}\end{minipage} &
\begin{minipage}{2.5cm}\begin{center}
Proportion\\ of orientable\\ toric real loci
\end{center}\end{minipage}
\\
\hline
2 & 5 & 1 & 20\% \\
\hline
3 & 18 & 3 & 16.67\% \\
\hline
4 & 124 & 4 & 3.23\% \\
\hline
5 & 866 & 12 & 1.39\% \\
\hline
6 & 7622 & 28 & 0.37\% \\
\hline
7 & 72256 & 85 & 0.12\% \\
\hline
8 & 749892 & 258 & 0.034\% \\
\hline
9 & 8229721 & 896 & 0.011\% \\
\hline
\end{tabular}

\vspace{1ex}

\caption{Proportion of orientable toric real loci in smooth toric
Fano $n$-folds for $n$ up to 9.}
\label{table:fano_orientable_real_loci}
\end{table}

\begin{remark}[Proportion vs.\ Statistical Estimate of Orientable Toric Real Loci]
We compare the found proportion from \Cref{table:fano_orientable_real_loci}
with a statistical expectation, based on the following reasoning.

Let $\Delta$ be the moment polytope for a smooth toric Fano variety
of complex dimension $n$ and let $d$ be its number of facets.
Without loss of generality, we assume that $\Delta$ has a standard vertex. 
Only the primitive normal vectors to the
$d-n$ facets away from the coordinate hyperplanes need to be checked
for the orientability criterion of the toric real locus (cf.\ \Cref{thm:orientability}).
\textit{A priori}, each of those $d-n$ relevant vectors has a 50\% chance of
boycotting orientability, i.e., of having an even number of odd entries.
Hence, we estimate at $(\frac 12)^{d-n}$ the chance that the
toric real locus is orientable, simply in terms of the
Picard number\footnote{The \emph{Picard number} of the toric variety
is the rank of its \emph{Picard group}, coincides with the second Betti number
$b_2$ of the toric symplectic manifold, and is equal to $d-n$,
where $d$ is the number of facets and $n$ the dimension of the
torus~\cite[p.331]{Delzant88}.}
of the variety, namely $d-n$, without looking at specific normal vectors.

For instance, when $n=2$, knowing that there is
one polygon with $d=3$, two with $d=4$, one with $d=5$, and one with $d=6$
(cf.\ \cref{fig:fano2}), we find the expected proportion
of orientable toric real loci
for smooth toric Fano varieties in this dimension as
\[
   \frac{\left( \tfrac 12 \right)^{1}
   + 2 \cdot \left( \tfrac 12 \right)^{2}
   + \left( \tfrac 12 \right)^{3}
   + \left( \tfrac 12 \right)^{4}}{5}
   = \frac{19}{80} = 0.2375.
\]

Similarly, when $n=3$, the expected proportion is
\[
   \frac{\left( \tfrac 12 \right)^{1}
   + 4 \cdot \left( \tfrac 12 \right)^{2}
   + 7 \cdot \left( \tfrac 12 \right)^{3}
   + 4 \cdot \left( \tfrac 12 \right)^{4}
   + 2 \cdot \left( \tfrac 12 \right)^{5}}{18}
   = \frac{43}{288} = 0.14930556.
\]

\Cref{table:fano_orientable_statistics} compares, for $n$ up to $9$,
the actual percentage $\mathcal{P}_n$ 
of smooth toric Fano $n$-folds having orientable toric real loci 
with the expected estimate $\mathcal{E}_n$ by the above procedure.
It is notable how the actual proportion deviates from this estimate,
possibly because the unimodularity condition on the normal vectors
makes them more prone to violate the condition for orientability of the toric real locus.
\end{remark}

\begin{table}
\centering
\begin{tabular}{|c|c|c|c|}
\hline
\begin{minipage}{1cm}
\centering $n$
\end{minipage} &
\begin{minipage}{2.7cm}\begin{center}
$\phantom{x}$\\ Actual\\ proportion, $\mathcal{P}_n$,\\ of orientable\\ toric real loci\\ $\phantom{x}$ 
\end{center}\end{minipage} &
\begin{minipage}{2.7cm}\begin{center}
Estimated\\ proportion, $\mathcal{E}_n$,\\ based on\\ Picard nos.
\end{center}\end{minipage} &
\begin{minipage}{2.7cm}\begin{center}
$\phantom{x}$\\ Relative\\ error in this\\ estimation,\\
$|\mathcal{E}_n-\mathcal{P}_n|/\mathcal{P}_n$\\ $\phantom{x}$ 
\end{center}\end{minipage}
\\
\hline
2 & 20.0000\% & 23.7500\% & 18.7500\% \\
\hline
3 & 16.6667\% & 14.9306\% & 10.4167\% \\
\hline
4 & 3.22581\% & 8.22518\% & 154.980\% \\
\hline
5 & 1.38568\% & 5.14579\% & 271.354\% \\
\hline
6 & 0.367358\% & 3.28646\% & 794.621\%\\
\hline
7 & 0.117637\% & 2.18751\% & 1759.54\%\\
\hline
8 & 0.0344050\% & 1.49221\% & 4237.19\%\\
\hline
9 & 0.0108874\% & 1.03563\% & 9412.21\%\\
\hline
\end{tabular}

\vspace{1ex}

\caption{Proportion of orientable toric real loci in smooth toric
Fano $n$-folds for $n$ up to 9 vs.\ estimate of proportion
by considering only the Picard numbers.}
\label{table:fano_orientable_statistics}
\end{table}

Using earlier classifications of low-dimensional smooth toric Fano varieties, which explicitly describe the variety associated to each polytope, we can be more concrete for $n=3$ (\Cref{exs_fano_3}) and $n=4$ (\Cref{exs_fano_4}).


\begin{example}
\label{exs_fano_3}
Smooth toric Fano threefolds were originally classified independently by
Batyrev \cite{Batyrev81} and by Watanabe and Watanabe \cite{WatanabeWatanabe}.
Within a total of eighteen examples, we find only three
orientable kaleidoscopes, namely those of the polytopes
corresponding to the manifolds
\begin{itemize}
    \item $\CC\PP^3$,
    \item $\CC\PP^3 \# \overline{\CC\PP^3}$ and
    \item $\CC\PP^1 \times \CC\PP^1 \times \CC\PP^1$.
\end{itemize}
These are the manifolds with ID numbers
23, 20, and 21 in Øbro's database, 
and numbers 1, 3, and 9 in Batyrev's list \cite{Batyrev81}.
Their respective (orientable) toric real loci are
\begin{itemize}
    \item $\RR\PP^3$,
    \item $\RR\PP^3 \# \RR\PP^3$ and
    \item $S^1 \times S^1 \times S^1$.
\end{itemize}    
The corresponding moment polytopes and kaleidoscopes appear in \cref{fig:kaleidoscope_cp3,fig:kaleidoscope_cp3_blownup_at_point,fig:kaleidoscope_cp1_cp1_cp1}.

The real loci of smooth toric Fano threefolds were also studied by Delaunay~\cite{DelaunayPhD}, although she worked in a more general setting than ours.
\end{example}


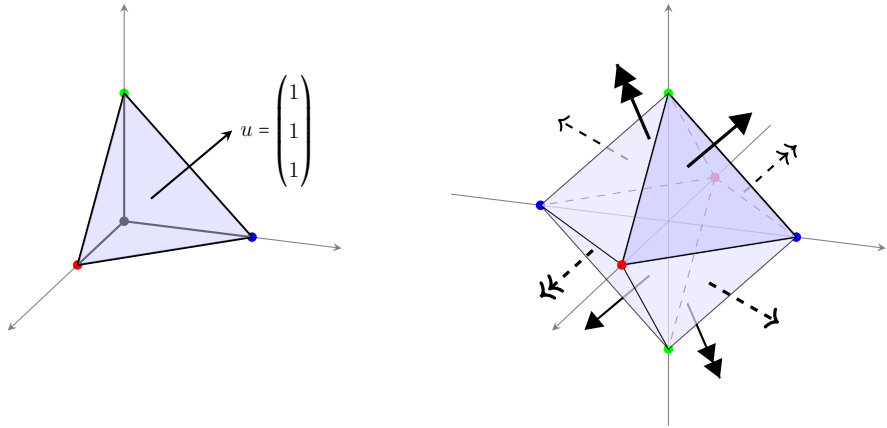
\begin{figure}
\begin{tikzpicture}[tdplot_main_coords, scale=1.8]

\begin{scope}[xshift=0cm]

  \coordinate (O) at (0,0,0);
  \coordinate (A) at (1,0,0); 
  \coordinate (B) at (0,1,0); 
  \coordinate (C) at (0,0,1); 

\fill (0,0,0) circle (1pt);
\fill[color=red] (1,0,0) circle (1pt);
\fill[color=blue] (0,1,0) circle (1pt);
\fill[color=green] (0,0,1) circle (1pt);

\draw[-stealth,opacity=.5] (O) -- (2.5,0,0);
\draw[-stealth,opacity=.5] (O) -- (0,1.7,0);
\draw[-stealth,opacity=.5] (O) -- (0,0,1.7);

  \draw[thick] (O) -- (A);
  \draw[thick] (O) -- (B);
  \draw[thick] (O) -- (C);
  \draw[thick] (A) -- (B);
  \draw[thick] (B) -- (C);
  \draw[thick] (C) -- (A);

  \filldraw[fill=blue!20, opacity=0.5] (A) -- (B) -- (C) -- cycle;

\draw[thick, -stealth] (0.33,0.33,0.33) -- (1.33,1.33,1.33) node[scale=0.7,right] {$u=\begin{pmatrix}1\\1\\1 \end{pmatrix}$};

\end{scope}

\begin{scope}[xshift=4cm]

  \coordinate (O) at (0,0,0);
  \coordinate (A) at (1,0,0); 
  \coordinate (B) at (0,1,0); 
  \coordinate (C) at (0,0,1); 
  \coordinate (-A) at (-1,0,0); 
  \coordinate (-B) at (0,-1,0); 
  \coordinate (-C) at (0,0,-1); 

\draw[-stealth,opacity=.5] (-2.2,0,0) -- (1,0,0);
\draw[-stealth,opacity=.5] (0,-1.7,0) -- (0,1.7,0);
\draw[-stealth,opacity=.5] (0,0,-1.6) -- (0,0,1.7);

  \draw[thick] (A) -- (B);
  \draw[thick] (B) -- (C);
  \draw[thick] (C) -- (A);

  \draw[dashed] (-A) -- (-B);
  \draw[dashed] (-A) -- (-C);
  \draw[dashed] (-A) -- (B);
  \draw[dashed] (-A) -- (C);

\fill[color=red] (-1,0,0) circle (1pt);
\fill[color=blue] (0,1,0) circle (1pt);
\fill[color=blue] (0,-1,0) circle (1pt);
\fill[color=green] (0,0,1) circle (1pt);
\fill[color=green] (0,0,-1) circle (1pt);

\draw[thick, -{triangle 60}] (-0.233,-0.233,-0.533) -- (-1.033,-1.033,-1.333);
\begin{scope}[thick, decoration={markings,
mark=at position 0.8 with {\arrow{triangle 60}},
mark=at position 1 with {\arrow{triangle 60}}}]
\draw[postaction={decorate}] (0.233,0.233,-0.533) -- (0.633,0.633,-0.933);
\end{scope}
\draw[thick, dashed, ->] (0.233,-0.233,0.533) -- (0.633,-0.633,0.933);
\begin{scope}[thick, dashed, decoration={markings,
mark=at position 0.85 with {\arrow{>}},
mark=at position 1 with {\arrow{>}}}]
\draw[postaction={decorate}] (-0.433,0.433,0.133) -- (-0.733,0.733,0.433);
\end{scope}

  \filldraw[fill=blue!10, opacity=0.7] (A) -- (-B) -- (C) -- cycle;
  \filldraw[fill=blue!10, opacity=0.7] (A) -- (B) -- (-C) -- cycle;
  \filldraw[fill=blue!10, opacity=0.7] (A) -- (-B) -- (-C) -- cycle;

  \filldraw[fill=blue!20, opacity=0.8] (A) -- (B) -- (C) -- cycle;

\fill[color=red] (1,0,0) circle (1pt);
\draw[-stealth,opacity=.5] (1,0,0) -- (2.5,0,0);

\draw[very thick, -{triangle 60}] (0.233,0.233,0.533) -- (1.033,1.033,1.333);
\begin{scope}[very thick, decoration={markings,
mark=at position 0.8 with {\arrow{triangle 60}},
mark=at position 1 with {\arrow{triangle 60}}}]
\draw[postaction={decorate}] (-0.233,-0.233,0.533) -- (-0.633,-0.633,0.933);
\end{scope}
\draw[very thick, dashed, ->] (-0.233,0.233,-0.533) -- (-0.633,0.633,-0.933);
\begin{scope}[very thick, dashed, decoration={markings,
mark=at position 0.85 with {\arrow{>}},
mark=at position 1 with {\arrow{>}}}]
\draw[postaction={decorate}] (0.433,-0.433,-0.133) -- (0.733,-0.733,-0.433);
\end{scope}

\end{scope}

\end{tikzpicture}
\caption{The moment polytope (a tetrahedron) and its kaleidoscope
for a standard toric $\CC\PP^3$, where the facets
are pairwise identified according to the depicted four arrow types.}
\label{fig:kaleidoscope_cp3}
\end{figure}


\begin{figure}[ht]
\begin{tikzpicture}[tdplot_main_coords, scale=0.8]

\begin{scope}[xshift=0cm]

  \coordinate (O) at (0,0,0);
  \coordinate (A) at (3,0,0); 
  \coordinate (B) at (0,3,0); 
  \coordinate (C) at (0,0,2); 
  \coordinate (D) at (1,0,2); 
  \coordinate (E) at (0,1,2); 

\draw[-stealth,opacity=.5] (O) -- (5,0,0);
\draw[-stealth,opacity=.5] (O) -- (0,4,0);
\draw[-stealth,opacity=.5] (O) -- (0,0,3);

  \draw[thick] (O) -- (A);
  \draw[thick] (O) -- (B);
  \draw[thick] (O) -- (C);
  \draw[thick] (A) -- (B);
  \draw[thick] (B) -- (E);
  \draw[thick] (D) -- (A);
  \draw[thick] (C) -- (E);
  \draw[thick] (C) -- (D);
  \draw[thick] (D) -- (E);

  \filldraw[fill=blue!20, opacity=0.5] (A) -- (B) -- (E) -- (D) -- cycle;
  \filldraw[fill=blue!20, opacity=0.5] (C) -- (D) -- (E) -- cycle;

\fill (O) circle (2pt);
\fill[color=red] (A) circle (2pt);
\fill[color=blue] (B) circle (2pt);
\fill (C) circle (2pt);
\fill[color=green] (D) circle (2pt);
\fill (E) circle (2pt);

\end{scope}

\begin{scope}[xshift=10cm]

  \coordinate (O) at (0,0,0);
  \coordinate (A) at (3,0,0); 
  \coordinate (B) at (0,3,0); 
  \coordinate (C) at (0,0,2); 
  \coordinate (D) at (1,0,2); 
  \coordinate (E) at (0,1,2); 

  \coordinate (-A) at (-3,0,0); 
  \coordinate (-B) at (0,-3,0); 
  \coordinate (-C) at (0,0,-2); 
  \coordinate (-D) at (-1,0,-2); %
  \coordinate (-E) at (0,-1,-2); %

  \coordinate (F) at (1,0,-2); %
  \coordinate (G) at (0,1,-2); %
  \coordinate (-F) at (-1,0,2); %
  \coordinate (-G) at (0,-1,2); %
  

  \draw[thick] (A) -- (D) -- (E) -- (B) -- cycle;
  \draw[thick] (A) -- (F) -- (G) -- (B) -- cycle;
  \draw[thick] (C) -- (D) -- (E) -- cycle;
  \draw[thick] (C) -- (D) -- (-G) -- cycle;
  \draw[thick] (C) -- (-F) -- (-G) -- cycle;
  \draw[thick] (C) -- (-F) -- (E) -- cycle;

  \draw[dashed] (B) -- (-A) -- (-B);
  \draw[dashed] (-E) -- (-D) -- (G);
  \draw[dashed] (-F) -- (-A) -- (-D);
  \draw[dashed] (-A) -- (A);
  \draw[dashed] (-B) -- (B);
  \draw[dashed] (-C) -- (C);
  \draw[dashed] (-D) -- (F);
  \draw[dashed] (-E) -- (G);

\draw[thick, -{triangle 60}] (-1,-1,-1) -- (-3,-3,-3);
\begin{scope}[thick, decoration={markings,
mark=at position 0.8 with {\arrow{triangle 60}},
mark=at position 1 with {\arrow{triangle 60}}}]
\draw[postaction={decorate}] (-1,1,-1) -- (-2,2,-2);
\end{scope}
\draw[thick, dashed, ->] (-1,-1,1) -- (-2,-2,2);
\begin{scope}[thick, dashed, decoration={markings,
mark=at position 0.85 with {\arrow{>}},
mark=at position 1 with {\arrow{>}}}]
\draw[postaction={decorate}] (-1.25,1.25,0.5) -- (-2.25,2.25,1.5);
\end{scope}

\fill (O) circle (2pt);
\fill[color=red] (-A) circle (2pt);
\fill (-C) circle (2pt);
\fill (-E) circle (2pt);
\fill (-F) circle (2pt);
\fill (-G) circle (2pt);

  \filldraw[fill=blue!10, opacity=0.7] (A) -- (-B) -- (-E) -- (F) -- cycle;
  \filldraw[fill=blue!10, opacity=0.7] (A) -- (B) -- (G) -- (F) -- cycle;
  \filldraw[fill=blue!10, opacity=0.7] (A) -- (-B) -- (-G) -- (D) -- cycle;
  \filldraw[fill=blue!10, opacity=0.7] (C) -- (-G) -- (D) -- cycle;
  \filldraw[fill=blue!10, opacity=0.7] (C) -- (-G) -- (-F) -- cycle;
  \filldraw[fill=blue!10, opacity=0.7] (C) -- (-F) -- (E) -- cycle;

  \filldraw[fill=blue!20, opacity=0.8] (A) -- (B) -- (E) -- (D) -- cycle;
  \filldraw[fill=blue!20, opacity=0.8] (C) -- (E) -- (D) -- cycle;

\fill[color=red] (A) circle (2pt);
\fill[color=blue] (B) circle (2pt);
\fill (C) circle (2pt);
\fill[color=green] (D) circle (2pt);
\fill (E) circle (2pt);
\fill (F) circle (2pt);
\fill (G) circle (2pt);
\fill[color=blue] (-B) circle (2pt);
\fill (-E) circle (2pt);
\fill (-F) circle (2pt);
\fill (-G) circle (2pt);

\draw[very thick, -{triangle 60}] (1,1,1) -- (3,3,3);
\begin{scope}[very thick, decoration={markings,
mark=at position 0.8 with {\arrow{triangle 60}},
mark=at position 1 with {\arrow{triangle 60}}}]
\draw[postaction={decorate}] (1,-1,1) -- (2,-2,2);
\end{scope}
\draw[very thick, dashed, ->] (1,1,-1) -- (2,2,-2);
\begin{scope}[very thick, dashed, decoration={markings,
mark=at position 0.85 with {\arrow{>}},
mark=at position 1 with {\arrow{>}}}]
\draw[postaction={decorate}] (1.25,-1.25,-0.5) -- (2.25,-2.25,-1.5);
\end{scope}

\end{scope}

\end{tikzpicture}
\caption{The moment polytope (a truncated tetrahedron) and its kaleidoscope
for a $\CC\PP^3 \# \overline{\CC\PP^3}$.
The horizontal facets of the kaleidoscope on the right are
translationally identified and the other facets are pairwise identified
according to the depicted four arrow types.}
\label{fig:kaleidoscope_cp3_blownup_at_point}
\end{figure}
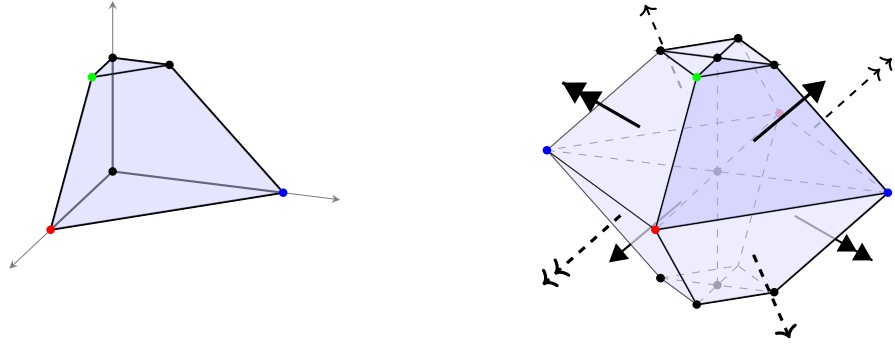


\begin{figure}[ht]
\begin{tikzpicture}[tdplot_main_coords, scale=1.8]

\begin{scope}[xshift=0cm]

  \coordinate (O) at (0,0,0);
  \coordinate (A) at (1,0,0); 
  \coordinate (B) at (0,1,0); 
  \coordinate (C) at (0,0,1); 
  \coordinate (AB) at (1,1,0);
  \coordinate (BC) at (0,1,1);
  \coordinate (AC) at (1,0,1);
  \coordinate (ABC) at (1,1,1);

\fill (0,0,0) circle (1pt);
\fill (1,0,0) circle (1pt);
\fill (0,1,0) circle (1pt);
\fill (0,0,1) circle (1pt);
\fill (1,1,0) circle (1pt);
\fill (0,1,1) circle (1pt);
\fill (1,0,1) circle (1pt);
\fill (1,1,1) circle (1pt);

\draw[-stealth,opacity=.5] (O) -- (2.5,0,0);
\draw[-stealth,opacity=.5] (O) -- (0,1.7,0);
\draw[-stealth,opacity=.5] (O) -- (0,0,1.7);

  \draw[thick] (O) -- (A);
  \draw[thick] (O) -- (B);
  \draw[thick] (O) -- (C);
  \draw[thick] (A) -- (AC);
  \draw[thick] (C) -- (AC);
  \draw[thick] (A) -- (AB);
  \draw[thick] (B) -- (AB);
  \draw[thick] (B) -- (BC);
  \draw[thick] (C) -- (BC);
  \draw[thick] (ABC) -- (AB);
  \draw[thick] (ABC) -- (AC);
  \draw[thick] (ABC) -- (BC);

  \filldraw[fill=blue!20, opacity=0.5] (A) -- (AC) -- (ABC)  -- (AB) -- cycle;
  \filldraw[fill=blue!20, opacity=0.5] (B) -- (BC) -- (ABC)  -- (AB) -- cycle;
  \filldraw[fill=blue!20, opacity=0.5] (C) -- (AC) -- (ABC)  -- (BC) -- cycle;

\end{scope}

\begin{scope}[xshift=4cm]

  \coordinate (O) at (0,0,0);
  \coordinate (A) at (1,0,0); 
  \coordinate (B) at (0,1,0); 
  \coordinate (C) at (0,0,1); 
  \coordinate (AB) at (1,1,0);
  \coordinate (BC) at (0,1,1);
  \coordinate (AC) at (1,0,1);
  \coordinate (ABC) at (1,1,1);
  \coordinate (-A) at (-1,0,0); 
  \coordinate (-B) at (0,-1,0); 
  \coordinate (-C) at (0,0,-1); 
  \coordinate (-A-B) at (-1,-1,0);
  \coordinate (-B-C) at (0,-1,-1);
  \coordinate (-A-C) at (-1,0,-1);
  \coordinate (-A-B) at (-1,-1,0);
  \coordinate (-B-C) at (0,-1,-1);
  \coordinate (-A-C) at (-1,0,-1);
  \coordinate (A-B) at (1,-1,0);
  \coordinate (B-C) at (0,1,-1);
  \coordinate (A-C) at (1,0,-1);
  \coordinate (-AB) at (-1,1,0);
  \coordinate (-BC) at (0,-1,1);
  \coordinate (-AC) at (-1,0,1);
  \coordinate (-ABC) at (-1,1,1);
  \coordinate (A-BC) at (1,-1,1);
  \coordinate (AB-C) at (1,1,-1);
  \coordinate (-A-BC) at (-1,-1,1);
  \coordinate (-AB-C) at (-1,1,-1);
  \coordinate (A-B-C) at (1,-1,-1);
  \coordinate (-A-B-C) at (-1,-1,-1);


  \draw[thick] (A) -- (AC);
  \draw[thick] (A) -- (AB);
  \draw[thick] (B) -- (AB);
  \draw[thick] (B) -- (BC);
  \draw[thick] (C) -- (AC);
  \draw[thick] (C) -- (BC);
  \draw[thick] (ABC) -- (AB);
  \draw[thick] (ABC) -- (AC);
  \draw[thick] (ABC) -- (BC);

  \draw[dashed] (A-B-C) -- (-A-B-C);
  \draw[dashed] (A-B) -- (-A-B);
  \draw[dashed] (A-C) -- (-A-C);
  \draw[dashed] (A) -- (-A);
  \draw[dashed] (-A-BC) -- (-A-B-C);
  \draw[dashed] (-AC) -- (-A-C);
  \draw[dashed] (-BC) -- (-B-C);
  \draw[dashed] (C) -- (-C);
  \draw[dashed] (-AB-C) -- (-A-B-C);
  \draw[dashed] (B-C) -- (-B-C);
  \draw[dashed] (-AB) -- (-A-B);
  \draw[dashed] (B) -- (-B);

\fill (0,0,0) circle (1pt);
\fill (1,0,0) circle (1pt);
\fill (0,1,0) circle (1pt);
\fill (0,0,1) circle (1pt);
\fill (-1,0,0) circle (1pt);
\fill (0,-1,0) circle (1pt);
\fill (0,0,-1) circle (1pt);
\fill (1,1,0) circle (1pt);
\fill (0,1,1) circle (1pt);
\fill (1,0,1) circle (1pt);
\fill (1,-1,0) circle (1pt);
\fill (0,1,-1) circle (1pt);
\fill (1,0,-1) circle (1pt);
\fill (-1,1,0) circle (1pt);
\fill (0,-1,1) circle (1pt);
\fill (-1,0,1) circle (1pt);
\fill (-1,-1,0) circle (1pt);
\fill (0,-1,-1) circle (1pt);
\fill (-1,0,-1) circle (1pt);
\fill (1,1,1) circle (1pt);
\fill (-1,1,1) circle (1pt);
\fill (1,-1,1) circle (1pt);
\fill (1,1,-1) circle (1pt);
\fill (-1,-1,1) circle (1pt);
\fill (-1,1,-1) circle (1pt);
\fill (1,-1,-1) circle (1pt);
\fill (-1,-1,-1) circle (1pt);

  \filldraw[fill=blue!10, opacity=0.7] (A) -- (AC) -- (A-BC)  -- (A-B) -- cycle;
  \filldraw[fill=blue!10, opacity=0.7] (A) -- (A-C) -- (A-B-C)  -- (A-B) -- cycle;
  \filldraw[fill=blue!10, opacity=0.7] (A) -- (A-C) -- (AB-C)  -- (AB) -- cycle;
  \filldraw[fill=blue!10, opacity=0.7] (C) -- (AC) -- (A-BC)  -- (-BC) -- cycle;
  \filldraw[fill=blue!10, opacity=0.7] (C) -- (-AC) -- (-A-BC)  -- (-BC) -- cycle;
  \filldraw[fill=blue!10, opacity=0.7] (C) -- (-AC) -- (-ABC)  -- (BC) -- cycle;
  \filldraw[fill=blue!10, opacity=0.7] (B) -- (AB) -- (AB-C)  -- (B-C) -- cycle;
  \filldraw[fill=blue!10, opacity=0.7] (B) -- (-AB) -- (-AB-C)  -- (B-C) -- cycle;
  \filldraw[fill=blue!10, opacity=0.7] (B) -- (-AB) -- (-ABC)  -- (BC) -- cycle;

  \filldraw[fill=blue!20, opacity=0.8] (A) -- (AC) -- (ABC)  -- (AB) -- cycle;
  \filldraw[fill=blue!20, opacity=0.8] (B) -- (BC) -- (ABC)  -- (AB) -- cycle;
  \filldraw[fill=blue!20, opacity=0.8] (C) -- (AC) -- (ABC)  -- (BC) -- cycle;

\end{scope}

\end{tikzpicture}
\caption{On the left, the moment polytope (a cube)
for a standard toric $\CC\PP^1 \times \CC\PP^1 \times \CC\PP^1$.
On the right, its kaleidoscope, where
parallel (outer) facets of the larger cube are pairwise identified
to produce a 3-torus.}
\label{fig:kaleidoscope_cp1_cp1_cp1}
\end{figure}
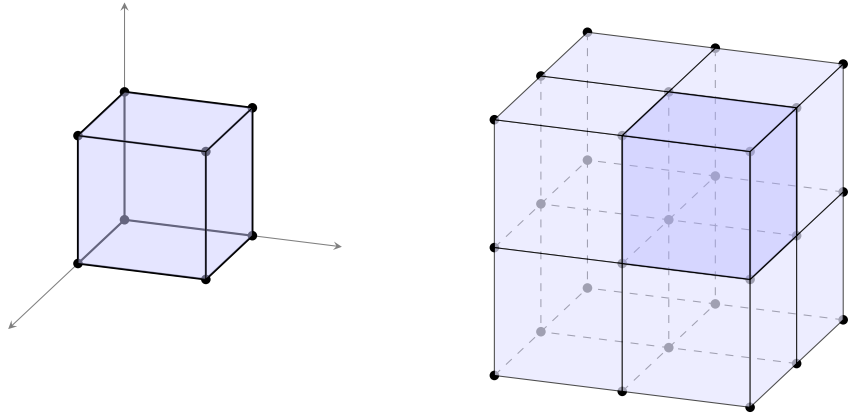


\begin{example}
\label{exs_fano_4}
Smooth toric Fano fourfolds were originally classified by Batyrev \cite{Batyrev1999}, with one missing example provided by Sato \cite{Sato}.
Within a total of 124 corresponding polytopes, 
we find only four orientable kaleidoscopes, namely those of the polytopes
corresponding to the manifolds
\begin{itemize}
    \item $\CC\PP^3 \times \CC\PP^1$,
    \item $\left( \CC\PP^1 \right)^4$,
    \item $\left( \CC\PP^3 \# \overline{\CC\PP^3} \right) \times \CC\PP^1$ and
    \item $\PP_{_{\CC\PP^3}} (\cO \oplus \cO (2))$
    (a $\CC\PP^1$-bundle over $\CC\PP^3$).
    
\end{itemize}
with respective (orientable) toric real loci
\begin{itemize}
    \item $\RR\PP^3 \times S^1$,
    \item $\left( S^1 \right)^4$,
    \item $\left( \RR\PP^3 \# \RR\PP^3 \right) \times S^1$, and
    \item again $\RR\PP^3 \times S^1$\footnote{The compatibility
    of complex conjugation with the toric fibration gives that
    this toric real locus is a \(S^1\)-bundle over \(\RP{3}\).
    Since both the base and the total space of this bundle are orientable, it follows that it is orientable \emph{as a bundle}. Orientable \(S^1\)-bundles over \(\RP{3}\) are classified by their Euler class in \(H^2(\RP{3};\Z) = \Z/2\Z\). By a Gysin sequence computation, we see that the total spaces of these two bundles are distinguished by their second integral cohomology group, which is \(\Z/2\Z\) for the trivial bundle and \(0\) for the non-trivial one. With the help of \emph{Sage} and \emph{Regina}, as explained in \cref{sec:sage_regina},
    we computed this cohomology group for this example, identifying it as the trivial \(S^1\)-bundle over \(\RP{3}\).}.
\end{itemize}
These are the polytopes listed with ID numbers
145, 142, 140 and 139 in \O bro's database~\cite{ObroSite}, with the associated toric varieties listed as numbers 5, 56, 25 and 3 in Batyrev's list~\cite{Batyrev1999}.
\end{example}


\begin{remark}[Small Cover Viewpoint]\label{rmk:small_cover}
We check here that the criterion given in \Cref{thm:orientability}
is equivalent to the one derived via the theory of \emph{small covers}.

Let $(M, \omega, \mu, \tau)$ be a toric symplectic manifold
equipped with a toric real structure,
let $M^\tau$ be the toric real locus, and let $\Delta$ be the moment polytope.
The restriction $\mu : M^\tau \to \Delta$ may be viewed as a
\emph{small cover}~\cite[\S 1]{DavisJanuszkiewicz91}.
In this perspective,
$M^\tau$ is determined, up to equivalence (see below), by the
\emph{characteristic function}~\cite[Prop.\ 1.8]{DavisJanuszkiewicz91},
$\lambda : (\ZZ/2\ZZ)^d \to (\ZZ/2\ZZ)^n$
taking the generator of the $j$-th factor $\ZZ/2\ZZ$ to the primitive
normal vector to the facet $F_j$ modulo 2, $[u_j] \in (\ZZ/2\ZZ)^n$:
\[
   \lambda (a_1, \ldots , a_d) = a_1 [u_1] + \dots + a_d [u_d].
\]
The \emph{equivalence} between two small covers $\mu_k : M_k^\tau \to \Delta$,
$k=1,2$, is an automorphism $\theta : (\ZZ/2\ZZ)^n \to (\ZZ/2\ZZ)^n$
together with a $\theta$-equivariant homeomorphism
$h : M_1^\tau \to M_2^\tau$ that covers the identity on $\Delta$.

Davis and Januszkiewicz~\cite[Coroll.\ 6.7]{DavisJanuszkiewicz91}
provided the formula $w = \Pi_{j=1}^{d} (1 + x_j)$ for the
total Stiefel-Whitney class of a small cover, where the $x_j \in H^1(M^\tau, \ZZ/2\ZZ)$, 
$j=1,\ldots,d$ are the classes dual to the preimages of each facet $F_j$ in $M^\tau$.
From this, an expression for the first Stiefel-Whitney
class of $M^\tau$ follows, namely $w_1 = x_1 + \ldots + x_d$.
Since the orientability is equivalent to the vanishing of $w_1$, we obtain
the orientability criterion for toric real loci
\[
   x_1 + \ldots + x_d=0 .
\]

On the other hand, they generalized the Danilov-Jurkiewicz
theorem~\cite[Thm.\ 4.14]{DavisJanuszkiewicz91}
to give the cohomology ring of $M^\tau$ with coefficients
in $\ZZ / 2\ZZ$ as
\[
   H^*(M^\tau, \ZZ/2\ZZ) \cong (\ZZ / 2\ZZ)[x_1,\ldots,x_d] / ( \cI + \cJ ),
\]
where $\cI$ is the ideal generated by $\Pi_{j\in I} x_j$
for each index subset $I \subseteq \{ 1,\ldots, d\}$ for which
the corresponding facets $F_j$, $j \in I$ have no common intersection point,
and $\cJ$ is the ideal generated by the linear combinations of the form
\begin{equation}\label{eq:lin_relns}
   \sum_{j=1}^d \alpha_j x_j \text{ with }
   \alpha \in \lambda^* (\ZZ / 2\ZZ)^n .
\end{equation}
Representing $\lambda$ by a matrix $\Lambda$
with columns $[u_1], \ldots , [u_d]$,
its dual $\lambda^*$ is represented by the transpose matrix $\Lambda^T$,
and the coefficient vectors $\alpha$ are the vectors
in the row space of $\Lambda$.

Since the only linear relations are those in $\cJ$,
the vanishing of $w_1$ is equivalent to $w_1 \in \cJ$,
i.e., to $w_1$ being expressible as a linear combination
of the relations \eqref{eq:lin_relns},
which boils down to the condition that
\begin{equation}\label{eq:condition}
\text{the vector } (1,1,\ldots,1) \text{ from } (\ZZ / 2\ZZ)^d
\text{ lie in the row space of }\Lambda .
\end{equation}

When we pick a lattice basis so that $\Delta$ has a standard vertex
(unimodularity allows this),
by reordering the facets, we may assume that the first
$n$ normal vectors are the standard basis vectors.
Then the $n \times d$ matrix $\Lambda$ contains
the $n \times n$ identity matrix in its first $n$ columns.
Hence, the condition \eqref{eq:condition} is equivalent in matrix terms to
\[
   \underbrace{(1\; 1\; \ldots \; 1)}_{n \text{ entries}} \, \Lambda \, = \,
   \underbrace{(1\; 1\; \ldots \; 1\; 1)}_{d \text{ entries}} .
\]
This is equivalent to the sum of the entries of each vector $[u_j]$,
$j=1,\ldots , d$ giving 1 (mod 2), which in turn is 
equivalent to each normal vector $u_j \in \ZZ^n$, $j=1,\ldots , d$
having an odd number of odd entries.

The criteria of Nakayama and Nishimura~\cite[Thm 1.7]{NakayamaNishimura},
and of Soprunova and Sottile~\cite[Thm 3.1]{SoprunovaSottile}
are also equivalent and arise from the same small-cover viewpoint.
\end{remark}


\section{The Case $n=2$}
\label{sec:case_n=2}

Let $(M, \omega, \mu)$ be a toric symplectic $T$-manifold of dimension $4$
with moment polygon $\Delta$.
Up to weak isomorphism, we may assume that $\Delta$ has one standard vertex at the origin.
The corresponding $\Delta$-kaleidoscope is a topological surface (actually a PL-surface by \cref{rmk:pl_str}).
Since, in this dimension, any topological manifold admits a unique
smooth structure up to diffeomorphism, we obtain from
\cref{coroll:kaleidoscope_vs_real_locus} the following special case:

\begin{corollary}
\label{coroll:reallocusaspolygonquotient}
Let $\tau$ be a toric real structure on $(M, \omega, \mu)$.
Under the assumptions of the previous paragraph,
the toric real locus $M^\tau$ is diffeomorphic to the
smooth, compact, connected surface corresponding to the
$\Delta$-kaleidoscope.
\end{corollary}

A full overview of the $2$-dimensional toric real loci
is thus possible, thanks to Miyake and Oda's classification of
compact complex toric surfaces~\cite[Theorem 11]{OdaMiyake}
(or see~\cite[Theorem~8.2]{Oda78}, \cite[Theorem 1.28]{Oda88}) in the context of toric varieties.
Later accounts of this classification
include~\cite[Section 2.5 and Notes to Chapter 2]{Fulton}
and~\cite[Theorem VII.4.1]{Audin}.

From that classification follows that, up to the action of $\AGL(2,\Z)$,
the moment polygon of a toric symplectic $4$-manifold with $d$ edges is:

\qquad \quad \begin{minipage}{14cm}
\begin{itemize}
    \item[-- when $d=3$] an isosceles triangle as in \cref{fig:triangle_and_trapezoid};
    \item[-- when $d=4$] a trapezoid, called a
\emph{Hirzebruch trapezoid}, as in \cref{fig:triangle_and_trapezoid};
    \item[-- when $d\geq 5$] a polygon obtained from one of the above trapezoids
    by a sequence of $d-4$ blow-ups at vertices (see \Cref{def:blow_up_at_vertex}).
\end{itemize}
\end{minipage}


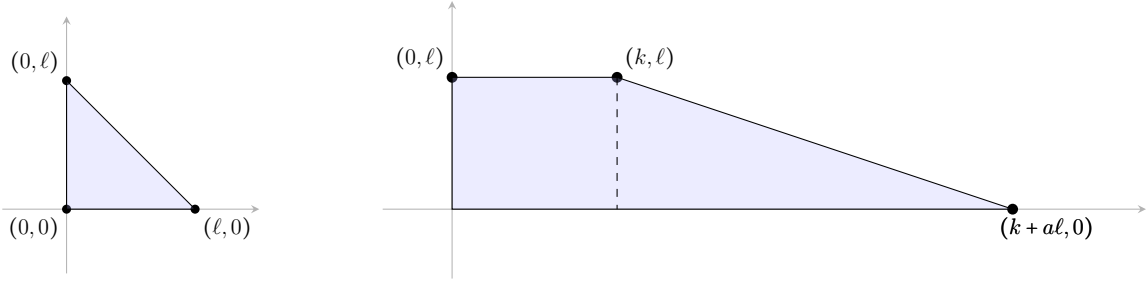
\begin{figure}[ht]
\centering
\begin{tikzpicture}[scale=1.7]


\begin{scope}[xshift=0cm]

\draw[-stealth,opacity=.3] (-0.5,0) -- (1.5,0);
\draw[-stealth,opacity=.3] (0,-0.5) -- (0,1.5);
\fill (0,0) circle (1pt) node[scale=0.7,below left] {$(0,0)$};
\fill (1,0) circle (1pt) node[scale=0.7,below right] {$\paren*{\ell,0}$};
\fill (0,1) circle (1pt) node[scale=0.7,above left] {$\paren*{0,\ell}$};
\filldraw[fill=blue!30,fill opacity=0.25] (0,0) -- (1,0) -- (0,1) -- cycle;

\end{scope}


\begin{scope}[xshift=3cm,scale=0.6]

\draw[-stealth,opacity=.3] (-0.9,0) -- (9,0);
\draw[-stealth,opacity=.3] (0,-0.9) -- (0,2.7);
\node[scale=0.7,below right] at (7,0) {$\paren*{k + a \ell ,0}$};
\fill (2.14 + 3*1.71,0) circle (2pt);
\node[scale=0.7,below right] at (7,0) {$\paren*{k + a \ell ,0}$};
\fill (2.14, 1.71) circle (2pt) node[scale=0.7,above right] {$\paren*{k,\ell}$};
\fill (0,1.71) circle (2pt) node[scale=0.7,above left] {$\paren*{0,\ell}$};
\filldraw[fill=blue!30,fill opacity=0.25] (0,0) -- (2.14 + 3*1.71,0) -- (2.14, 1.71) -- (0,1.71) -- cycle;
\draw[dashed] (2.14,1.71)--(2.14,0);
\end{scope}

\end{tikzpicture}
\caption{Moment polytopes with three and four edges up to $\AGL(2,\Z)$.
Here, $k,\ell >0$ and $a$ is a nonnegative integer.}
\label{fig:triangle_and_trapezoid}
\end{figure}

For instance, an $\epsilon$-blow-up at the top vertex of the isosceles triangle
on the left of \cref{fig:triangle_and_trapezoid} produces a trapezoid as on
the right of that figure with $a=1$, $k=\varepsilon$ and $\ell=a-\varepsilon$.
An $\epsilon$-blow-up of the trapezoid in \cref{fig:triangle_and_trapezoid}
is illustrated in \cref{fig:blow_up_at_vertex_other}.


\begin{figure}[ht]
\centering
\begin{tikzpicture}[scale=0.9]
\draw[-stealth,opacity=.3] (-0.9,0) -- (9,0);
\draw[-stealth,opacity=.3] (0,-0.9) -- (0,2.7);
\fill (0,0) circle (2pt) node[scale=0.7,below left] {$(0,0)$};
\fill[fill opacity=0.5] (2.14 + 3*1.71,0) circle (2pt) node[scale=0.7,below right] at (7,0) {$\paren*{k + a \ell,0}$};
\fill (2.14 + 3*1.71 - .57,0) circle (2pt);
\fill (2.14 + 2*1.71,.57) circle (2pt);
\fill (2.14, 1.71) circle (2pt) node[scale=0.7,above right] {$\paren*{k,\ell}$};
\fill (0,1.71) circle (2pt) node[scale=0.7,above left] {$\paren*{0,\ell}$};
\filldraw[fill=blue!30,fill opacity=0.25] (0,0) -- (2.14 + 3*1.71 - .57,0) -- (2.14 + 2*1.71,.57) -- (2.14, 1.71) -- (0,1.71) -- cycle;
\filldraw[dashed,fill=blue!30,fill opacity=0.11] (2.14 + 2*1.71,.57) -- (2.14 + 3*1.71,0) -- (2.14 + 3*1.71 - .57,0) -- cycle;
\draw[dashed] (2.14,1.71)--(2.14,0);
\end{tikzpicture}
\caption{An $\epsilon$-blow-up of the trapezoid from
\cref{fig:triangle_and_trapezoid} at the right-most vertex.}
\label{fig:blow_up_at_vertex_other}
\end{figure}
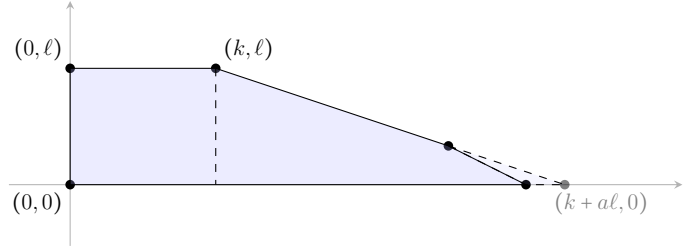


\cref{coroll:reallocusaspolygonquotient}
gives a practical description of a toric real locus as a smooth surface.
For instance, the unimodular triangle $\Delta$ in \cref{fig:triangle_and_trapezoid}
is the moment polygon of the toric symplectic 4-manifold
$(\CP{2}, 2\ell \omega_{_\text{FS}}, \mu)$.
For any toric real structure, we recover the toric real locus $\RR\PP^2$
from the  $\Delta$-kaleidoscope as in \cref{fig:square}.

\begin{figure}[ht]
\centering
\begin{tikzpicture}[scale=1.8]
\draw[-stealth,opacity=.5] (-1.3,0) -- (1.5,0);
\draw[-stealth,opacity=.5] (0,-1.3) -- (0,1.5);
\fill (0,0) circle (1pt) node[scale=0.7,below left] {$(0,0)$};
\fill (1,0) circle (1pt) node[scale=0.7,below right] {$\paren*{\ell,0}$};
\fill (0,1) circle (1pt) node[scale=0.7,above left] {$\paren*{0,\ell}$};
\filldraw[fill=blue!50,fill opacity=0.25] (0,0) -- (1,0) -- (0,1) -- cycle;
\filldraw[fill=blue!20,fill opacity=0.25] (0,0) -- (-1,0) -- (0,1) -- cycle;
\filldraw[fill=blue!20,fill opacity=0.25] (0,0) -- (1,0) -- (0,-1) -- cycle;
\filldraw[fill=blue!20,fill opacity=0.25] (0,0) -- (-1,0) -- (0,-1) -- cycle;
\node at (0.2,0.33) {$\Delta$};
\node at (0.33,-0.3) {$\Delta_{_{(+,-)}}$};
\node at (-0.3,0.3) {$\Delta_{_{(-,+)}}$};
\node at (-0.3,-0.3) {$\Delta_{_{(-,-)}}$};
\draw[thick, -stealth] (0.5,0.5) -- (0.8,0.8) node[scale=0.7,below right] {$u=\begin{pmatrix}1 \\1 \end{pmatrix}$};
\end{tikzpicture}
\hspace{20mm}
\begin{tikzpicture}[scale = 1.8, baseline = -66pt]
\begin{scope}[very thin]
\draw[opacity=.5] (-1,0) -- (1,0);
\draw[opacity=.5] (0,-1) -- (0,1);
\end{scope}
\filldraw[fill=blue!20,fill opacity=0.25] (1,0) -- (0,1) -- (-1,0) -- (0,-1) -- cycle;
\begin{scope}[very thick,decoration={markings,
mark=at position 0.5 with {\arrow{>}}}]
\draw[postaction={decorate}] (0,1)--(1,0);
\draw[postaction={decorate}] (0,-1)--(-1,0);
\end{scope}
\begin{scope}[very thick,decoration={markings,
mark=at position 0.45 with {\arrow{<}},
mark=at position 0.55 with {\arrow{<}}}]
\draw[postaction={decorate}] (0,1)--(-1,0);
\draw[postaction={decorate}] (0,-1)--(1,0);
\end{scope}
\end{tikzpicture}
\caption{The $\Delta$-kaleidoscope for $\CC\PP^2$ with its boundary
identifications (given by different arrows) yielding $\RR\PP^2$ as toric real locus.}
\label{fig:square}
\end{figure}
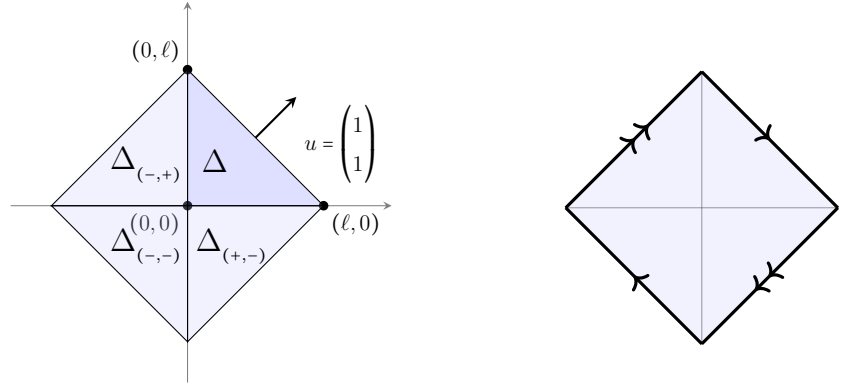

\begin{example}\label{ex:hirzebruchlocus}
The family of trapezoids from \cref{fig:triangle_and_trapezoid},
corresponds to toric symplectic Hirzebruch surfaces,
shortly denoted $\cH_a$ (see, for instance, \cite[Exercise IV.4]{Audin}).
Here, we choose to omit the (symplectically-significant) side lengths
$k$ and $\ell$, since these do not affect the normal vectors,
hence do not affect the diffeomorphism type of a toric real locus.
The relevant outward-pointing primitive normal vectors to the edges are
$u=(0,1)$ and $v=(1,a)$.
According to \cref{coroll:reallocusaspolygonquotient},
for any toric real structure on $\cH_a$,
the diffeomorphism type of the toric real locus
depends only on the parity of $a$:
whereas the horizontal edge in the first quadrant with normal
$u=(0,1)$ always gets identified with the horizontal edge in
$\Delta_{(-1)^u}=\Delta_{(+,-)}$ in the fourth quadrant,
the slanted edge in the first quadrant with normal
$v=(1,a)$ gets identified with the slanted edge in
$\Delta_{(-1)^{v}}=\Delta_{(-,(-)^a)}$,
in the second or third quadrant, depending on whether $a$
is even or odd, respectively.
When $a$ is even, we see that
the toric real locus of $\cH_a$ is diffeomorphic to a 2-torus,
as illustrated in~\cref{fig:evenhirzebruch};
when $a$ is odd, the toric real locus of
$\cH_a$ is diffeomorphic to a Klein bottle, as illustrated
in \cref{fig:oddhirzebruch}.
\qedhere

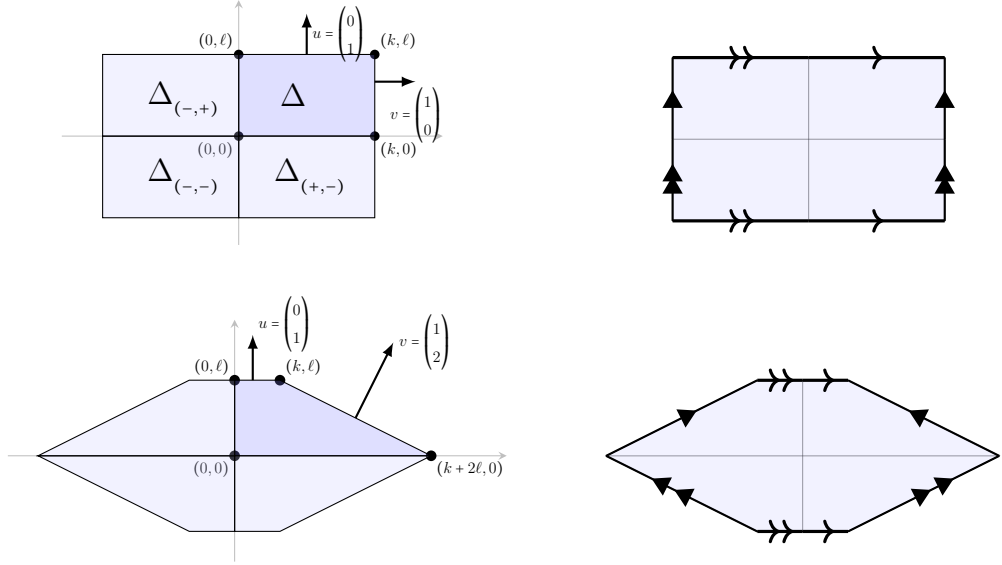
\begin{figure}[ht]
\centering

\begin{tikzpicture}[scale = 1.8]
\draw[-stealth,opacity=.5,gray] (-1.3,0) -- (1.5,0);
\draw[-stealth,opacity=.5,gray] (0,-0.8) -- (0,1);
\fill (0,0) circle (1pt) node[scale=0.5,below left] {$(0,0)$};
\fill (1,0) circle (1pt) node[scale=0.5,below right] {$(k,0)$};
\fill (1,0.6) circle (1pt) node[scale=0.5,above right] {$(k,\ell)$};
\fill (0,0.6) circle (1pt) node[scale=0.5,above left] {$(0, \ell)$};
\filldraw[fill=blue!50,fill opacity=0.25] (0,0) -- (1,0) -- (1,0.6) -- (0,0.6) -- cycle;
\filldraw[fill=blue!20,fill opacity=0.25] (0,0) -- (1,0) -- (1,-0.6) -- (0,-0.6) -- cycle;
\filldraw[fill=blue!20,fill opacity=0.25] (0,0) -- (-1,0) -- (-1,-0.6) -- (0,-0.6) -- cycle;
\filldraw[fill=blue!20,fill opacity=0.25] (0,0) -- (-1,0) -- (-1,0.6) -- (0,0.6) -- cycle;
\node at (0.4,0.3) {$\Delta$};
\node at (0.53,-0.3) {$\Delta_{_{(+,-)}}$};
\node at (-0.4,0.27) {$\Delta_{_{(-,+)}}$};
\node at (-0.4,-0.3) {$\Delta_{_{(-,-)}}$};
\draw[-latex, thick] (0.5,0.6) -- node[right, scale=0.5] {$u=\begin{pmatrix}0 \\1 \end{pmatrix}$} (0.5,0.9);
\draw[-latex, thick] (1,0.4) -- (1.3,0.4) node[below, scale=0.5] {$v=\begin{pmatrix}1 \\0 \end{pmatrix}$};
\end{tikzpicture}
\hspace{27mm}
%
\begin{tikzpicture}[scale = 1.8, baseline = -40pt]
\begin{scope}[very thin]
\draw[opacity=.5] (-1,0) -- (1,0);
\draw[opacity=.5] (0,-0.6) -- (0,0.6);
\end{scope}
\filldraw[fill=blue!20,fill opacity=0.25] (-1,0.6) -- (1,0.6) -- (1,-0.6) -- (-1,-0.6) -- cycle;
\begin{scope}[very thick,decoration={markings,
mark=at position 0.55 with {\arrow{>}}}]
\draw[postaction={decorate}] (0,0.6)--(1,0.6);
\draw[postaction={decorate}] (0,-0.6)--(1,-0.6);
\end{scope}
\begin{scope}[very thick,decoration={markings,
mark=at position 0.5 with {\arrow{>}},
mark=at position 0.6 with {\arrow{>}}}]
\draw[postaction={decorate}] (-1,0.6)--(0,0.6);
\draw[postaction={decorate}] (-1,-0.6)--(0,-0.6);
\end{scope}
\begin{scope}[thick,decoration={markings,
mark=at position 0.6 with {\arrow{triangle 60}}}]
\draw[postaction={decorate}] (1,0) -- (1,0.6);
\draw[postaction={decorate}] (-1,0) -- (-1,0.6);
\end{scope}
\begin{scope}[thick,decoration={markings,
mark=at position 0.55 with {\arrow{triangle 60}},
mark=at position 0.7 with {\arrow{triangle 60}}}]
\draw[postaction={decorate}] (1,-0.6) -- (1,0);
\draw[postaction={decorate}] (-1,-0.6) -- (-1,0);
\end{scope}
\end{tikzpicture}
\\[3ex]
%
\begin{tikzpicture}[scale = 2]
\draw[-stealth,opacity=.5,gray] (-1.5,0) -- (1.8,0);
\draw[-stealth,opacity=.5,gray] (0,-0.7) -- (0,0.9);
\fill (0,0) circle (1pt) node[scale=0.5,below left] {$(0,0)$};
\fill (1.3,0) circle (1pt) node[scale=0.5,below right] {$(k+2\ell,0)$};
\fill (0.3,0.5) circle (1pt) node[scale=0.5,above right] {$(k,\ell)$};
\fill (0,0.5) circle (1pt) node[scale=0.5,above left] {$(0, \ell)$};
\filldraw[fill=blue!50,fill opacity=0.25] (0,0) -- (1.3,0) -- (0.3,0.5) -- (0,0.5) -- cycle;
\filldraw[fill=blue!20,fill opacity=0.25] (0,0) -- (1.3,0) -- (0.3,-0.5) -- (0,-0.5) -- cycle;
\filldraw[fill=blue!20,fill opacity=0.25] (0,0) -- (-1.3,0) -- (-0.3,-0.5) -- (0,-0.5) -- cycle;
\filldraw[fill=blue!20,fill opacity=0.25] (0,0) -- (-1.3,0) -- (-0.3,0.5) -- (0,0.5) -- cycle;
\draw[-latex, thick] (0.12,0.5) -- node[above right, scale=0.5] {$u=\begin{pmatrix}0 \\1 \end{pmatrix}$} (0.12,0.8);
\draw[-latex, thick] (0.8,0.25) -- (1.05,0.75) node[right, scale=0.5] {$v=\begin{pmatrix}1 \\2 \end{pmatrix}$};
\end{tikzpicture}
\hspace{10mm}
%
\begin{tikzpicture}[scale = 2, baseline = -40pt]
\begin{scope}[very thin]
\draw[opacity=.5] (-1.3,0) -- (1.3,0);
\draw[opacity=.5] (0,-0.5) -- (0,0.5);
\end{scope}
\filldraw[fill=blue!20,fill opacity=0.25] (-1.3,0) -- (-0.3,0.5) -- (0.3,0.5) -- (1.3,0) -- (0.3,-0.5) -- (-0.3,-0.5) -- cycle;
\begin{scope}[very thick,decoration={markings,
mark=at position 0.7 with {\arrow{>}}}]
\draw[postaction={decorate}] (0,0.5)--(0.3,0.5);
\draw[postaction={decorate}] (0,-0.5)--(0.3,-0.5);
\end{scope}
\begin{scope}[very thick,decoration={markings,
mark=at position 0.5 with {\arrow{>}},
mark=at position 0.8 with {\arrow{>}}}]
\draw[postaction={decorate}] (-0.3,0.5)--(0,0.5);
\draw[postaction={decorate}] (-0.3,-0.5)--(0,-0.5);
\end{scope}
\begin{scope}[thick,decoration={markings,
mark=at position 0.6 with {\arrow{triangle 60}}}]
\draw[postaction={decorate}] (1.3,0) -- (0.3,0.5);
\draw[postaction={decorate}] (-1.3,0) -- (-0.3,0.5);
\end{scope}
\begin{scope}[thick,decoration={markings,
mark=at position 0.55 with {\arrow{triangle 60}},
mark=at position 0.7 with {\arrow{triangle 60}}}]
\draw[postaction={decorate}] (-0.3,-0.5) -- (-1.3,0);
\draw[postaction={decorate}] (0.3,-0.5) -- (1.3,0);
\end{scope}
\end{tikzpicture}
\caption{$\Delta$-kaleidoscopes for Hirzebruch surfaces
$\cH_a$ with $a = 0$ and $a = 2$,
and their boundary identifications yielding 2-tori as toric real loci.}
\label{fig:evenhirzebruch}
\end{figure}

\begin{figure}[ht]
\centering
\begin{tikzpicture}[scale = 2]
\draw[-stealth,opacity=.5,gray] (-1.5,0) -- (1.8,0);
\draw[-stealth,opacity=.5,gray] (0,-0.7) -- (0,0.9);
\fill (0,0) circle (1pt) node[scale=0.5,below left] {$(0,0)$};
\fill (1.3,0) circle (1pt) node[scale=0.5,below right] {$(k+\ell,0)$};
\fill (0.8,0.5) circle (1pt) node[scale=0.5,above right] {$(k,\ell)$};
\fill (0,0.5) circle (1pt) node[scale=0.5,above left] {$(0, \ell)$};
\filldraw[fill=blue!50,fill opacity=0.25] (0,0) -- (1.3,0) -- (0.8,0.5) -- (0,0.5) -- cycle;
\filldraw[fill=blue!20,fill opacity=0.25] (0,0) -- (1.3,0) -- (0.8,-0.5) -- (0,-0.5) -- cycle;
\filldraw[fill=blue!20,fill opacity=0.25] (0,0) -- (-1.3,0) -- (-0.8,-0.5) -- (0,-0.5) -- cycle;
\filldraw[fill=blue!20,fill opacity=0.25] (0,0) -- (-1.3,0) -- (-0.8,0.5) -- (0,0.5) -- cycle;
\draw[-latex, thick] (0.32,0.5) -- node[above right, scale=0.5] {$u=\begin{pmatrix}0 \\1 \end{pmatrix}$} (0.32,0.8);
\draw[-latex, thick] (1.05,0.25) -- (1.35,0.55) node[right, scale=0.5] {$v=\begin{pmatrix}1 \\1 \end{pmatrix}$};
\end{tikzpicture}
\hspace{10mm}
%
\begin{tikzpicture}[scale = 2, baseline = -40pt]
\begin{scope}[very thin]
\draw[opacity=.5] (-1.3,0) -- (1.3,0);
\draw[opacity=.5] (0,-0.5) -- (0,0.5);
\end{scope}
\filldraw[fill=blue!20,fill opacity=0.25] (-1.3,0) -- (-0.8,0.5) -- (0.8,0.5) -- (1.3,0) -- (0.8,-0.5) -- (-0.8,-0.5) -- cycle;
\begin{scope}[very thick,decoration={markings,
mark=at position 0.6 with {\arrow{>}}}]
\draw[postaction={decorate}] (0,0.5)--(0.8,0.5);
\draw[postaction={decorate}] (0,-0.5)--(0.8,-0.5);
\end{scope}
\begin{scope}[very thick,decoration={markings,
mark=at position 0.6 with {\arrow{>}},
mark=at position 0.7 with {\arrow{>}}}]
\draw[postaction={decorate}] (-0.8,0.5)--(0,0.5);
\draw[postaction={decorate}] (-0.8,-0.5)--(0,-0.5);
\end{scope}
\begin{scope}[thick,decoration={markings,
mark=at position 0.7 with {\arrow{triangle 60}}}]
\draw[postaction={decorate}] (1.3,0) -- (0.8,0.5);
\draw[postaction={decorate}] (-1.3,0) -- (-0.8,-0.5);
\end{scope}
\begin{scope}[thick,decoration={markings,
mark=at position 0.5 with {\arrow{triangle 60}},
mark=at position 0.7 with {\arrow{triangle 60}}}]
\draw[postaction={decorate}] (-0.8,0.5) -- (-1.3,0);
\draw[postaction={decorate}] (0.8,-0.5) -- (1.3,0);
\end{scope}
\end{tikzpicture}
\caption{The $\Delta$-kaleidoscope for a Hirzebruch surface $\cH_a$ with $a = 1$,
and its boundary identifications yielding a Klein bottle as toric real locus.}
\label{fig:oddhirzebruch}
\end{figure}
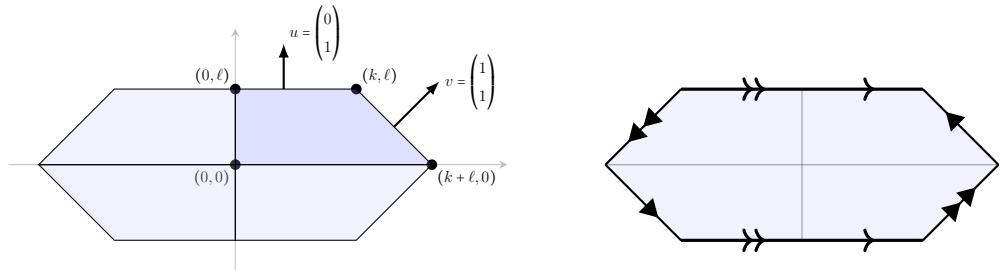
\end{example}


So we know the toric real loci in the basic cases of a moment polygon
with three edges, namely $\RP{2}$, and four edges, namely a torus or a Klein bottle.
Miyake and Oda's theorem states that the other toric real loci, i.e., for number $d \geq 5$ of edges,
are obtained from the basic cases for four edges by $d-4$ blow-ups at fixed points.
Since such a blow-up amounts to a connected sum
with $\RR \PP^2$ (cf. the proof of \Cref{prop:kaleidoscope_blowup}),
the set of all toric real loci (up to diffeomorphism)
of $4$-dimensional toric symplectic manifolds is clear.
In particular, with the exception of the even Hirzebruch surfaces,
such toric real loci are never orientable, and all nonorientable surfaces
occur as such toric real loci.
Altogether, we have shown the following upgrade
of \cref{coroll:reallocusaspolygonquotient}.
A similar result was obtained by Delaunay~\cite[Proposition 4.1.2]{Delaunay2} in the algebraic geometry setting.

\begin{corollary}[List of 2-Dimensional Toric Real Loci]
\label{coroll:4dim}
Let $(M,\omega,\mu,\tau)$ be a $4$-dimensional toric symplectic manifold
equipped with a toric real structure.
Then its toric real locus, $M^\tau$, is a compact connected surface which,
up to diffeomorphism, is independent of the particular choice of $\tau$.
		
If $M^\tau$ is orientable, then $(M, \omega, \mu)$ is weakly isomorphic
to a Hirzebruch surface $\cH_a$ for some even $a$, and $M^\tau$ is a $2$-torus.
In particular, $M^\tau$ is never diffeomorphic to a sphere,
or to a compact connected orientable surface of genus higher than $1$.
		
By contrast, any compact connected nonorientable surface can be realized
as such a $M^\tau$:
\begin{itemize}
\item if $M^\tau \cong \RP{2}$,
then $(M,\omega,\mu)$ is weakly isomorphic to a standard $\CP{2}$\footnote{By a
\textit{standard} $\CP{2}$, we mean that
$\Delta = \{ (x,y) \in \RR^2 \mid x,y \geq 0, x+y \leq k \}$ for some $k>0$.};
\item if $M^\tau \cong \RP{2} \,\#\, \RP{2} = \cK$ is a Klein bottle,
then $(M,\omega,\mu)$ is weakly isomorphic to a Hirzebruch surface $\cH_a$,
for some odd $a$;
\item if $M^\tau \cong \#_{m}\,\RP{2}$ is the $m$-fold
connected sum of $\RP{2}$'s for some $m \ge 3$,
then $(M,\omega,\mu)$ is weakly isomorphic to a $(m-2)$-fold blow-up of a Hirzebruch surface $\cH_a$, for some positive integer $a$.
\end{itemize}
\end{corollary}

The last bullet point follows from the identities
$\TT^2 \,\#\, \RP{2} \; \cong \; \cK \,\#\, \RP{2}
   \; \cong \; \RP{2} \,\#\, \RP{2} \,\#\, \RP{2}$.


Moreover, thanks to \cite[Corollary 1.29]{Oda88}, it becomes possible to
understand the $\AGL(2,\Z)$-equivalence classes of
unimodular polygons in $\RR^2$ with a specific number $d$
of edges, although this becomes more involved when $d$ is high.
Next, we cite the lists for pentagons and hexagons,
as these will be useful in \Cref{sec:case_n=3}.

When the unimodular polygon $\Delta$ has five edges,
then, up to $\AGL(2,\Z)$,
it is a blow-up of one of the previous trapezoids in
\cref{fig:triangle_and_trapezoid} at the upper right vertex;
thus, it is as in \cref{fig:pentagons}; cf.~\cite[Fig.1.17(iv)]{Oda88}.
The outward-pointing primitive normal vectors are then
$(-1,0),(0,-1),(1,a), (1,{a+1}) \allowbreak \text{ and } \allowbreak (0,1)$.


	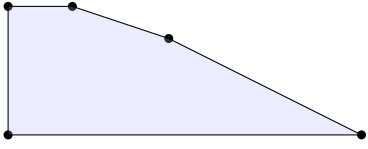
\begin{figure}[ht]
		\centering
			\begin{tikzpicture}[scale=1.7]
			\fill (0,0) circle (1pt) node[scale=0.7,below left] {};
			\fill (2.75,0) circle (1pt) node[scale=0.7,below right] {};
			\fill (0,1) circle (1pt) node[scale=0.7,above left] {};
			\fill (0.5,1) circle (1pt) node[scale=0.7] {};
			\fill (1.25,0.75) circle (1pt) node[scale=0.7] {};
			\filldraw[fill=blue!30,fill opacity=0.25] (0,0) -- (2.75,0) -- (1.25, 0.75) -- (0.5,1) -- (0,1) -- cycle;
		\end{tikzpicture}
		\caption{Moment polygon with five edges up to $\AGL(2,\Z)$.}
		\label{fig:pentagons}
	\end{figure}

When the unimodular polygon $\Delta$ has six edges,
then, up to $\AGL(2,\Z)$, it is:
\begin{itemize}
				\item a blow-up of one of the previous pentagons at the lower left vertex. Equivalently, this is also a blow-up of a standard Hirzebruch trapezoid at the upper left and lower right vertices, with normal vectors \((-1,0),(0,-1),(1,{a-1}),(1,a),\allowbreak (0,1),\allowbreak (-1,1)\);
				\item a blow-up of one of the previous pentagons at the middle right vertex, with normal vectors \((-1,0),(0,-1),(1,a),(2,2a+1),(1,a+1),(0,1)\);
				\item a blow-up of one of the previous pentagons at the rightmost vertex, with normal vectors \((-1,0),(0,-1),(1,a-1),(1,a),(1,a+1),(0,1)\).
\end{itemize}
Thus $\Delta$ is as one of the three cases in \cref{fig:hexagons};
cf.~\cite[Fig.1.17(v)]{Oda88}.


	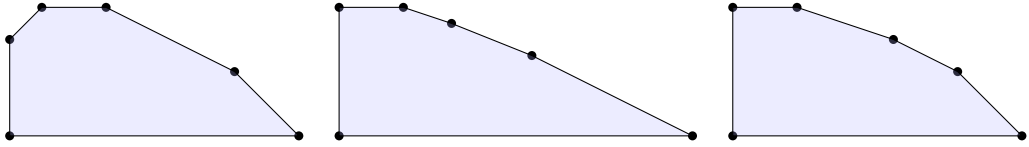
\begin{figure}[ht]
		\centering
			\begin{tikzpicture}[scale=1.7]
			\fill (0,0) circle (1pt) node[scale=0.7,below left] {};
			\fill (2.25,0) circle (1pt) node[scale=0.7,below right] {};
			\fill (1.75,0.5) circle (1pt) node[scale=0.7,above left] {};
			\fill (0.75,1) circle (1pt) node[scale=0.7] {};
			\fill (0.25,1) circle (1pt) node[scale=0.7] {};
			\fill (0,0.75) circle (1pt) node[scale=0.7] {};
			\filldraw[fill=blue!30,fill opacity=0.25] (0,0) -- (2.25,0) -- (1.75, 0.5) -- (0.75,1) -- (0.25,1) -- (0,0.75) -- cycle;
		\end{tikzpicture}
		\begin{tikzpicture}[scale=1.7]
			\fill (0,0) circle (1pt) node[scale=0.7,below left] {};
			\fill (2.75,0) circle (1pt) node[scale=0.7,below right] {};
			\fill (0,1) circle (1pt) node[scale=0.7,above left] {};
			\fill (0.5,1) circle (1pt) node[scale=0.7] {};
			\fill (1.5,0.625) circle (1pt) node[scale=0.7] {};
			\fill (0.875,0.875) circle (1pt) node[scale=0.7] {};
			\filldraw[fill=blue!30,fill opacity=0.25] (0,0) -- (2.75,0) -- (1.5,0.625) -- (0.875, 0.875) -- (0.5,1) -- (0,1) -- cycle;
		\end{tikzpicture}
		\begin{tikzpicture}[scale=1.7]
			\fill (0,0) circle (1pt) node[scale=0.7,below left] {};
			\fill (2.25,0) circle (1pt) node[scale=0.7,below right] {};
			\fill (0,1) circle (1pt) node[scale=0.7,above left] {};
			\fill (0.5,1) circle (1pt) node[scale=0.7] {};
			\fill (1.25,0.75) circle (1pt) node[scale=0.7] {};
			\fill (1.75,0.5) circle (1pt) node[scale=0.7] {};
			\filldraw[fill=blue!30,fill opacity=0.25] (0,0) -- (2.25,0) -- (1.75,0.5) -- (1.25, 0.75) -- (0.5,1) -- (0,1) -- cycle;
		\end{tikzpicture}
		\caption{Moment polygons with six edges up to $\AGL(2,\Z)$.}
		\label{fig:hexagons}
	\end{figure}
    

\section{The Case $n=3$}
	\label{sec:case_n=3}
	
	We have seen in \cref{sec:case_n=2} how, when \(n = 2\),
    the kaleidoscope model allows for a complete determination
    of all compact connected surfaces
    that arise as toric real loci of toric symplectic \(4\)-manifolds.
    This relied crucially on the existence of an explicit finite family
    of \textit{minimal models} for such manifolds;
    by \textit{minimal} we mean a manifold that is not the toric blow-up of another.

    In higher dimensions, the situation is more involved.
	Besides blow-ups at (fixed) points,
    blow-ups along toric symplectic submanifolds of codimension \(2k\)
    for \(k \ge 2\) should also be considered.\footnote{Symplectic blow-up
    along a symplectic submanifold, like that at a point, goes back to
    Gromov~\cite[3.4.4(D),p.342]{GromovBook}, McDuff \cite{McDuff}, Guillemin and Sternberg \cite{GuilleminSternberg89}, and, as a particular case of cutting
    suitable for the toric setting,
    Lerman \cite[Remark 1.5]{LermanSymplecticCuts}.}
    Such a blow-up corresponds to a blow-up of the moment polytope along a face of codimension \(k\) (see e.g.\ \cite[Definition 2.29]{McDuffTolman2}).
    \cref{def:blow_up_along_edge} covers the case $k=1$.
	
	Moreover, even with this increased range of allowed blow-up types,
    it follows from a result of Pelayo and Santos \cite[Theorems 2.26 and 2.27]{PelayoSantos23} that, for any \(n \ge 3\), there exist unimodular \(n\)-dimensional polytopes with an arbitrarily large number of facets that are not obtained from any other unimodular polytope by a blow-up along a face.
    In particular, this implies that there cannot be a finite list of minimal models for toric symplectic \(2n\)-manifolds.
	
	Therefore, the methods of \cref{sec:case_n=2} cannot provide a classification of higher-dimensional toric real loci.
    This does not rule out the possibility of such a classification through other methods.
    Moreover, by \cref{thm:orientability}, the real loci
    of the toric symplectic manifolds obtained through the construction of
    Pelayo and Santos
    are always nonorientable.
    We are thus led to the question:
	\begin{question}\label{question:minimaltoricrealloci}
		Are there finitely many minimal models for toric real loci when \(n \ge 3\)? What if we restrict to  orientable toric real loci?
	\end{question}
    
	Nevertheless, we consider below the toric real loci of toric symplectic $6$-manifolds,
    where the practical applicability of the
    kaleidoscope framework is demonstrated with examples from a
    partial classification result of Miyake, Oda, and Nagaya \cite{Oda78}.
	
	\subsection{Blow-ups of 3-Dimensional Toric Real Loci}
	
	A toric symplectic \(6\)-manifold can be equivariantly blown-up either at a fixed point or along a toric symplectic submanifold of dimension \(2\).
    This corresponds to a blow-up of the 3-dimensional moment polytope either at a vertex or \emph{along an edge}.
	The case of blow-up at a vertex was covered in \cref{prop:kaleidoscope_blowup} for all $n$.
    We now turn to the case of a blow-up along an edge.

    \begin{definition}
\label{def:blow_up_along_edge}
    Let $E$ be the edge of a unimodular polytope
    $\Delta$ in $\RR^n$ ($n \ge 3$) connecting vertices $\xi_1$ and $\xi_2$.
    Let $v_{k,1}, \ldots, v_{k,n-1} \in \Z^n$ be the primitive vectors
    along the other edges (i.e., besides $E$)
    emanating from $\xi_k$, $k=1,2$.
Let $\epsilon$ be a positive number small enough
for each of $\xi_k+\varepsilon v_{k,j}$, $k=1,2$, $j=1,\ldots,n-1$,
to lie in the relative interior of an edge incident to the $\xi_k$'s.
    The \textbf{$\boldsymbol{\epsilon}$-blow-up of $\Delta$
along $E$} is the polytope
$\Bl_E \! \Delta$,
whose vertices are those of $\Delta$ removing $\xi_1$ and $\xi_2$
and adding the $2(n-1)$ new vertices $\xi_k+\varepsilon v_{k,j}$, $k=1,2$,
$j=1,\ldots,n-1$.
\end{definition}

The polytope $\Bl_E \! \Delta$ is again unimodular:
the primitive vectors along the edges
emanating from the new vertex $\xi_k+\varepsilon v_{k,j}$
are $v_{k,j}$, $v_{k,j'}-v_{k,j}$ with $j' \neq j$,
and the primitive vector of $\pm(\xi_1-\xi_2)$.
The normal to the new facet (shared by the new vertices) is the sum of
the $n-1$ normals to the old facets meeting along the edge $E$.
This is illustrated in \cref{fig:polytope_blowup_along_edge} for $n=3$.
The polytope $\Bl_E \! \Delta$ corresponds to the toric symplectic manifold
that
(up to isomorphism) is the $\epsilon$-blow-up of the one
corresponding to $\Delta$ along the toric symplectic submanifold
$\mu^{-1}(E)$.


        \begin{figure}[ht]
		\begin{tikzpicture}[tdplot_main_coords, scale=1.8]

            
			\begin{scope}[xshift=0cm]
				
				\coordinate (O) at (0,0,0);
				\coordinate (A) at (1.5,0,0); 
				\coordinate (B) at (0,1,0); 
				\coordinate (C) at (0,0,1); 
				\coordinate (AB) at (1.5,1,0);
				\coordinate (BC) at (0,1,1);
				
				\fill[red] (0,0,0) circle (1pt);
				\fill (1.5,0,0) circle (1pt);
				\fill (0,1,0) circle (1pt);
				\fill[red] (0,0,1) circle (1pt);
				\fill (1.5,1,0) circle (1pt);
				\fill (0,1,1) circle (1pt);
				
				\draw[-stealth,opacity=.5] (O) -- (2.5,0,0);
				\draw[-stealth,opacity=.5] (O) -- (0,1.7,0);
				\draw[-stealth,opacity=.5] (O) -- (0,0,1.7);
				
				\draw[thick,dashed] (O) -- (A);
				\draw[thick,dashed] (O) -- (B);
				\draw[red,very thick,dashed] (O) -- (C);
				\draw[thick] (B) -- (BC);
				\draw[thick] (C) -- (BC);
				\draw[thick] (A) -- (AB);
				\draw[thick] (B) -- (AB);
				\draw[thick] (AB) -- (BC);
				\draw[thick] (A) -- (C);
				
				\filldraw[fill=blue!20, opacity=0.5] (B) -- (BC) -- (AB) -- cycle;
				\filldraw[fill=blue!20, opacity=0.5] (C) -- (BC) -- (AB) -- (A) -- cycle;

                \node at (0,1.5,1.2) {$\Delta$};

			\end{scope}
	
           			
			\begin{scope}[xshift=4cm]
				
				\coordinate (O) at (0,0,0);
				\coordinate (A) at (1.5,0,0); 
				\coordinate (B) at (0,1,0); 
				\coordinate (C) at (0,0,1); 
				\coordinate (AB) at (1.5,1,0);
				\coordinate (BC) at (0,1,1);
						
	\coordinate (AC1) at (0.5,0,0.666666);
	\coordinate (BCC1) at (0,0.333333,1);
	\coordinate (A1) at (0.5,0,0);
	\coordinate (B1) at (0,0.333333,0);
					
	\fill[red] (A1) circle (1pt);
	\fill[red] (B1) circle (1pt);
			
	\fill[red] (BCC1) circle (1pt);
	\fill[red] (AC1) circle (1pt);				
				
	\fill[fill=red!40, opacity=0.8] (AC1) -- (BCC1) -- (B1) -- (A1) -- cycle;
				
				\draw[-stealth,opacity=.5] (O) -- (2.5,0,0);
				\draw[-stealth,opacity=.5] (O) -- (0,1.7,0);
				\draw[-stealth,opacity=.5] (O) -- (0,0,1.7);

				\draw[dashed] (B) -- (B1);
				\draw[dashed] (A) -- (A1);
				\draw[dashed,red,thick] (BCC1) -- (B1);
				\draw[dashed,red,thick] (A1) -- (B1);
				\draw[dashed,red,thick] (AC1) -- (A1);

	\filldraw[fill=blue!20, opacity=0.5] (B) -- (BC) -- (AB) -- cycle;
	\filldraw[fill=blue!20, opacity=0.5] (AC1) -- (BCC1) -- (BC) -- (AB) -- (A) -- cycle;			

				\draw[thick] (B) -- (BC);
				\draw[thick] (BCC1) -- (BC);
				\draw[thick] (A) -- (AB);
				\draw[thick] (B) -- (AB);
				\draw[thick] (AB) -- (BC);
				\draw[thick] (A) -- (AC1);
				\draw[red,thick] (BCC1) -- (AC1);

	\fill (1.5,0,0) circle (1pt);
	\fill (0,1,0) circle (1pt);
	\fill (1.5,1,0) circle (1pt);
	\fill (0,1,1) circle (1pt);

                \node at (0,1.5,1.2) {$\widetilde \Delta$};

			\end{scope}
			
		\end{tikzpicture}
		\caption{An $\epsilon$-blow-up of the polytope $\Delta$ on the
        left along the edge on the $z$-axis produces a polytope
        $\widetilde \Delta$ as on the right side.
        The new facet, shaded in red, is a Hirzebruch trapezoid.}
		\label{fig:polytope_blowup_along_edge}
	\end{figure}
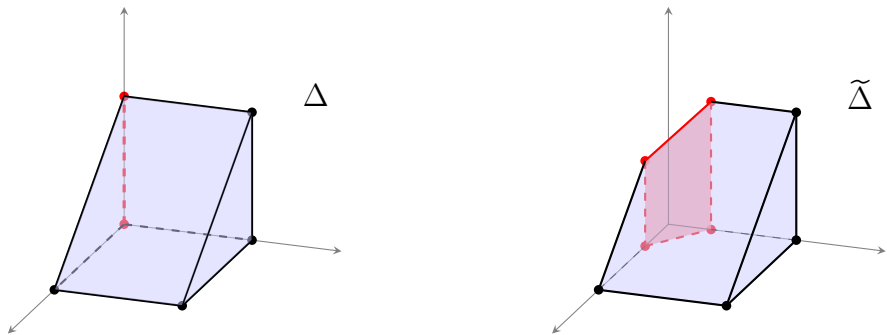

    Recall that, analogously to the classical complex blow-up,
    the \emph{(projective) real blow-up of a smooth
    manifold $X$ along a closed submanifold \(Y\)} of codimension \(k\), denoted \(\Bl_Y \! X\),
    is obtained by removing from \(X\) a tubular neighborhood of \(Y\), which can be seen as a normal open \(k\)-disk bundle over \(Y\), and identifying the boundary points of each disk fiber in antipodal pairs according to the quotient map \(S^{k-1} \to \RP{k-1}\); see, for example, \cite{AkbulutKing81, AkbulutKing85,AroneKankaanrinta,Zinger24}.

    We next provide a proof of the following proposition
    using kaleidoscopes.
	
\begin{proposition}[Toric Real Locus of Blow-Up is Blow-Up of Toric Real Locus, Instance 2]
\label{prop:blow_up_along_edge}
		Let \((M, \omega, \mu)\) be a \(6\)-dimensional toric symplectic manifold, let $\tau$ be a toric real structure on $M$,
        let $N \subset M$ be a 2-dimensional toric symplectic submanifold,
        and let \((\Bl_N \! M, \widetilde{\omega}, \widetilde{\mu})\) be an \(\epsilon\)-blow-up of \((M,\omega,\mu)\) along \( N \).
		Then the toric real locus of any toric real structure on \(\Bl_N \! M\) is diffeomorphic to a real blow-up of the toric real locus \(M^\tau\) along the one-dimensional submanifold \(N^\tau \coloneqq N \cap M^\tau\), i.e.,
		\[
        (\Bl_N \! M)^{\tau} \cong \Bl_{N^\tau} \! M^\tau.
        \]
\end{proposition}

	\begin{proof}
		Let \(\Delta\) be the moment polytope of \(M\)
        and $E$ its edge corresponding to $N$, i.e., $N= \mu^{-1} (E)$.
        Without loss of generality, we assume $\Delta$ to have
        a standard vertex at the origin and $E$ to lie
        along the positive \(z\)-axis with the origin as one of its vertices.
		
		Let $\widetilde{\Delta} := \Bl_E \! \Delta$ be the moment polytope of \(\Bl_N \! M\),
        let \(U \subset \R^3\) be the open vertical square prism
        defined by $|x|+|y| < \epsilon$, and let \([U]\) be its image in the \(\Delta\)-kaleidoscope.
        The \(\widetilde{\Delta}\)-cluster is obtained from the \(\Delta\)-cluster by removing from the latter its intersection with \(U\). Therefore, the \(\widetilde{\Delta}\)-kaleidoscope is obtained from the \(\Delta\)-kaleidoscope by removing \([U]\) and identifying the resulting vertical quadrilateral facets in opposite pairs. This is illustrated in \cref{fig:reallocusafterblowupatedge}.

        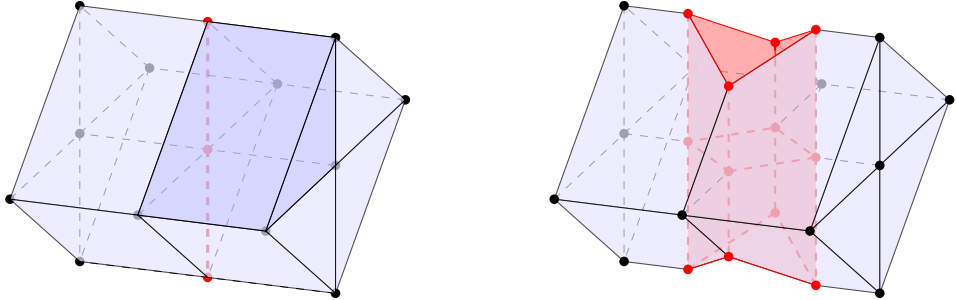
\begin{figure}[ht]
		\begin{tikzpicture}[tdplot_main_coords, scale=1.8]
			

        \begin{scope}[xshift=0cm]
				
				\coordinate (O) at (0,0,0);
				\coordinate (A) at (1.5,0,0); 
				\coordinate (B) at (0,1,0); 
				\coordinate (C) at (0,0,1); 
				\coordinate (AB) at (1.5,1,0);
				\coordinate (BC) at (0,1,1);
				\coordinate (-A) at (-1.5,0,0); 
				\coordinate (-B) at (0,-1,0); 
				\coordinate (-C) at (0,0,-1); 
				\coordinate (-A-B) at (-1.5,-1,0);
				\coordinate (-B-C) at (0,-1,-1);
				\coordinate (A-B) at (1.5,-1,0);
				\coordinate (B-C) at (0,1,-1);
				\coordinate (-AB) at (-1.5,1,0);
				\coordinate (-BC) at (0,-1,1);
								
				\draw (B) -- (BC);
				\draw (C) -- (BC);
				\draw (A) -- (AB);
				\draw (B) -- (AB);
				\draw (AB) -- (BC);
				\draw (A) -- (C);
				
				\draw[dashed] (A-B) -- (-A-B);
				\draw[dashed] (B-C) -- (-B-C);
				\draw[dashed] (A) -- (-A);
				\draw[dashed] (-BC) -- (-B-C);
				\draw[dashed,red,very thick] (C) -- (-C);
				\draw[dashed] (-AB) -- (-A-B);
				\draw[dashed] (B) -- (-B);
                \draw[dashed] (-A) -- (C);
				\draw[dashed] (-A-B) -- (-BC);
				\draw[dashed] (-A) -- (-C);
				\draw[dashed] (-A-B) -- (-B-C);
				
				\fill[red] (0,0,0) circle (1pt);
				\fill (1.5,0,0) circle (1pt);
				\fill (0,1,0) circle (1pt);
				\fill[red] (0,0,1) circle (1pt);
				\fill (-1.5,0,0) circle (1pt);
				\fill (0,-1,0) circle (1pt);
				\fill[red] (0,0,-1) circle (1pt);
				\fill (1.5,1,0) circle (1pt);
				\fill (0,1,1) circle (1pt);
				\fill (1.5,-1,0) circle (1pt);
				\fill (0,1,-1) circle (1pt);
				\fill (-1.5,1,0) circle (1pt);
				\fill (0,-1,1) circle (1pt);
				\fill (-1.5,-1,0) circle (1pt);
				\fill (0,-1,-1) circle (1pt);
				
				\filldraw[fill=blue!10, opacity=0.7] (B) -- (BC) -- (-AB) -- cycle;
				\filldraw[fill=blue!10, opacity=0.7] (B) -- (B-C) -- (AB) -- cycle;
				\filldraw[fill=blue!10, opacity=0.7] (B) -- (B-C) -- (-AB) -- cycle;
				\filldraw[fill=blue!10, opacity=0.7] (C) -- (-BC) -- (A-B) -- (A) -- cycle;
				\filldraw[fill=blue!10, opacity=0.7] (-C) -- (B-C) -- (AB) -- (A) -- cycle;
				\filldraw[fill=blue!10, opacity=0.7] (-C) -- (-B-C) -- (A-B) -- (A) -- cycle;
				
				\filldraw[fill=blue!20, opacity=0.8] (B) -- (BC) -- (AB) -- cycle;
				\filldraw[fill=blue!20, opacity=0.8] (C) -- (BC) -- (AB) -- (A) -- cycle;
			\end{scope}

            
			\begin{scope}[xshift=4cm]
				
				\coordinate (O) at (0,0,0);
				\coordinate (A) at (1.5,0,0); 
				\coordinate (B) at (0,1,0); 
				\coordinate (C) at (0,0,1); 
				\coordinate (AB) at (1.5,1,0);
				\coordinate (BC) at (0,1,1);
				\coordinate (-A) at (-1.5,0,0); 
				\coordinate (-B) at (0,-1,0); 
				\coordinate (-C) at (0,0,-1); 
				\coordinate (-A-B) at (-1.5,-1,0);
				\coordinate (-B-C) at (0,-1,-1);
				\coordinate (A-B) at (1.5,-1,0);
				\coordinate (B-C) at (0,1,-1);
				\coordinate (-AB) at (-1.5,1,0);
				\coordinate (-BC) at (0,-1,1);
				
				\coordinate (AC1) at (0.5,0,0.666666);
				\coordinate (BCC1) at (0,0.5,1);
				\coordinate (-AC1) at (-0.5,0,0.666666);
				\coordinate (-BCC1) at (0,-0.5,1);
				\coordinate (A-C1) at (0.5,0,-0.666666);
				\coordinate (BC-C1) at (0,0.5,-1);
				\coordinate (-A-C1) at (-0.5,0,-0.666666);
				\coordinate (-BC-C1) at (0,-0.5,-1);
				\coordinate (A1) at (0.5,0,0);
				\coordinate (B1) at (0,0.5,0);
				\coordinate (-A1) at (-0.5,0,0);
				\coordinate (-B1) at (0,-0.5,0);
				
				\fill (-1.5,0,0) circle (1pt);
				\fill (0,-1,0) circle (1pt);
				\fill (-1.5,-1,0) circle (1pt);
				
				\draw[dashed] (-AB) -- (-A-B);
				\draw[dashed] (A-B) -- (-A-B);
				\draw[dashed] (-BC) -- (-B-C);
				\draw[dashed] (-A) -- (-AC1);
				\draw[dashed] (-A-B) -- (-BC);
				\draw[dashed] (-A) -- (-A-C1);
				\draw[dashed] (-A-B) -- (-B-C);
				\fill[fill=red!40, opacity=0.8] (-AC1) -- (-BCC1) -- (-BC-C1) -- (-A-C1) -- cycle;
				\fill[fill=red!40, opacity=0.8] (-AC1) -- (BCC1) -- (BC-C1) -- (-A-C1) -- cycle;
				\fill[fill=red!40, opacity=0.8] (AC1) -- (-BCC1) -- (-BC-C1) -- (A-C1) -- cycle;
				\fill[fill=red!40, opacity=0.8] (AC1) -- (BCC1) -- (BC-C1) -- (A-C1) -- cycle;
				
				\fill[red] (-0.5,0,-0.666666) circle (1pt);
				\fill[red] (0.5,0,0) circle (1pt);
				\fill[red] (0,0.5,0) circle (1pt);
				\fill[red] (-0.5,0,0) circle (1pt);
				\fill[red] (0,-0.5,0) circle (1pt);
				
				\draw[dashed,red,thick] (BCC1) -- (BC-C1);
				\draw[dashed,red,thick] (-BCC1) -- (-BC-C1);
				\draw[dashed,red,thick] (AC1) -- (A-C1);
				\draw[dashed,red,thick] (-AC1) -- (-A-C1);
				\draw[dashed,red,thick] (BC-C1) -- (-A-C1);
				\draw[dashed,red,thick] (-BC-C1) -- (-A-C1);
				\draw[red] (-AC1) -- (-0.5,0,0.55);
				
				\draw[dashed] (B) -- (B1);
				\draw[dashed] (-B) -- (-B1);
				\draw[dashed] (A) -- (A1);
				\draw[dashed] (-A) -- (-A1);
				\draw[dashed,red,thick] (A1) -- (B1) -- (-A1) -- (-B1) -- cycle;
				
				\filldraw[fill=blue!10, opacity=0.7] (B) -- (BC) -- (-AB) -- cycle;
				\filldraw[fill=blue!10, opacity=0.7] (B) -- (B-C) -- (AB) -- cycle;
				\filldraw[fill=blue!10, opacity=0.7] (B) -- (B-C) -- (-AB) -- cycle;
				\filldraw[fill=blue!10, opacity=0.7] (AC1) -- (-BCC1) -- (-BC) -- (A-B) -- (A) -- cycle;
				\filldraw[fill=blue!10, opacity=0.7] (A-C1) -- (BC-C1) -- (B-C) -- (AB) -- (A) -- cycle;
				\filldraw[fill=blue!10, opacity=0.7] (A-C1) -- (-BC-C1) -- (-B-C) -- (A-B) -- (A) -- cycle;
				
				\filldraw[fill=blue!10, opacity=0.7] (B) -- (BC) -- (AB) -- cycle;
				\filldraw[fill=blue!10, opacity=0.7] (AC1) -- (BCC1) -- (BC) -- (AB) -- (A) -- cycle;
				
				\fill (1.5,0,0) circle (1pt);
				\fill (0,1,0) circle (1pt);
				\fill (1.5,1,0) circle (1pt);
				\fill (0,1,1) circle (1pt);
				\fill (0,1,-1) circle (1pt);
				\fill (0,-1,1) circle (1pt);
				\fill (0,-1,-1) circle (1pt);
				\fill (1.5,-1,0) circle (1pt);
				\fill (-1.5,1,0) circle (1pt);
				
				\fill[red] (0,0.5,1) circle (1pt);
				\fill[red] (0,-0.5,1) circle (1pt);
				\fill[red] (0,0.5,-1) circle (1pt);
				\fill[red] (0,-0.5,-1) circle (1pt);
				\fill[red] (0.5,0,0.666666) circle (1pt);
				\fill[red] (-0.5,0,0.666666) circle (1pt);
				\fill[red] (0.5,0,-0.666666) circle (1pt);
				
				\draw[red] (-AC1) -- (BCC1) -- (AC1);
				\draw[red] (-AC1) -- (-BCC1) -- (AC1);
                \draw[red] (BC-C1) -- (A-C1) -- (-BC-C1);
				
				
			\end{scope}
			
		\end{tikzpicture}
		\caption{The effect on the kaleidoscope of a blow-up along an edge.
        The left image depicts the $\Delta$-cluster of the
        polytope $\Delta$ from \cref{fig:polytope_blowup_along_edge}.
        The right image depicts the $\widetilde \Delta$-cluster after the
        blow-up, with the new faces shaded red.}
		\label{fig:reallocusafterblowupatedge}
	\end{figure}
		
		The tubular neighborhood $\mu^{-1}(U) \cap M^\tau$
        of the circle \(N^\tau\) in $M^\tau$ corresponds, by
        the homeomorphism $h$ of \cref{coroll:kaleidoscope_vs_real_locus},
        to the tubular neighborhood $[U]$ of the image of the
        $z$-axis in the
        $\Delta$-kaleidoscope.
        Since a tubular neighborhood identification (i.e., an isomorphism with an open disk bundle) can be chosen
        to be symmetric under the reflections given by the coordinate planes, gluing the new vertical facets in opposite pairs corresponds to identifying antipodal points in the boundary of the disk fibers of the normal bundle of $N^\tau$ in
        $M^{\tau}$.
	\end{proof}

    	
	\subsection{The Examples of Miyake--Oda}
    \label{subsec:MiyakeOdaNagaya}
	
	Although there is no hope of classifying the toric symplectic \(6\)-manifolds that are minimal with respect to blow-ups, it is possible to obtain partial classification results if one restricts to a class of examples with limited complexity.
	
	Recall that each toric symplectic manifold has an underlying smooth projective toric variety associated with the normal fan of the moment polytope.
    Conversely, a smooth projective toric variety determines, up to isomorphism, a family of toric symplectic manifolds corresponding to each choice of unimodular polytope\footnote{These
    polytopes are not required to be lattice polytopes.} with normal fan coinciding with that of the given toric variety. These polytopes are all \emph{analogous} in the sense of \cite{McDuffTolman1}, so, in particular, they are combinatorially equivalent and give rise to homeomorphic kaleidoscopes. As such, the toric real locus, up to homeomorphism, does not depend on this choice of polytope.
    
    Indeed, Miyake and Oda obtained a classification of all 3-dimensional smooth compact toric varieties with Picard number up to \(5\) that are minimal with respect to toric blow-ups, with a missing non-projective example and a simplified proof supplied by Nagaya (\cite[Theorem 9.6]{Oda78}; see also \cite[Theorem 1.34]{Oda88}). Among these varieties, we restrict our attention to those that are projective, as these are the ones relevant to the toric symplectic setting. For convenience, these examples are listed in \cref{appendix}.
    We thus gain access to plenty of examples of minimal toric symplectic \(6\)-manifolds with $b_2 \le 5$, i.e., whose moment polytopes have up to 8 facets.
    
    \begin{remark}\label{rmk:Miyake_Oda_not_symplectic}
        This result of Miyake and Oda does not yield a classification
        of all toric symplectic \(6\)-manifolds with \(b_2 \le 5\)
        that are minimal with respect to (toric symplectic) blow-ups.
        Indeed, there exist unimodular 3-dimensional polytopes whose 
        normal fans can be blown down to a smooth projective fan, yet the polytopes themselves cannot be blown down.
        Such an example was given by McDuff and Tolman in
        \cite[Remark 2.42]{McDuffTolman2}. In this case, the symplectic blow-down cannot be performed, despite the possibility of a complex-algebraic blow-down. In general, deciding whether a symplectic blow-down can be performed in dimension \(6\) or higher is a subtle question, see e.g.\ \cite{LiRuanZhang}.
        These complications do not arise when blowing up at a point,
        but only when blowing up along a submanifold of positive dimension.
    \end{remark}

    We now turn our attention to the examples of Miyake–Oda, with the goal of understanding their toric real loci. One feature that is immediately clear, by combining the list of examples in \cref{thm:Oda3d} with \cref{thm:orientability}, is their orientability (or, more commonly, the lack thereof), which can be directly read from the lists of facet normals. Accordingly, we split our analysis of these examples into two cases.

    However, the kaleidoscope construction contains further information
    beyond orientability, yielding a very explicit topological model of the toric real locus.
    In many cases, we can use this to easily recognize the underlying \(3\)-manifold, especially when combined with observations such as \cref{ex:locuscpn,rmk:product,prop:kaleidoscope_blowup}. The following remark plays a role in our analysis below.

    \begin{remark}\label{rmk:dependsonlymod2}
        The examples in \cref{thm:Oda3d} come in families that may depend on integer parameters \(a,b,c,d,e\). However, among each family, the arrangement of the rays of the fan into three-dimensional cones is independent of these parameters. In fact, the only dependence on these integer parameters is in specifying the rays themselves. This means that all moment polytopes for toric varieties in each individual family are \emph{combinatorially equivalent}; i.e., there is an inclusion-preserving bijection between their sets of faces. Of course, this bijection, in general, does not preserve the associated primitive normal vectors, as these do depend on the parameters \(a,b,c,d,e\) (i.e., these polytopes may not be \emph{analogous} in the sense of \cite{McDuffTolman1}).

        From the point of view of the methods of this paper, two combinatorially equivalent unimodular polytopes \(\Delta\) and \(\Delta'\) give rise to homeomorphic clusters with equivalent polyhedral decompositions. Still, since corresponding facets in \(\Delta\) and \(\Delta'\) might have different normal directions, the \(\Delta\)-kaleidoscope and the \(\Delta'\)-kaleidoscope might not be homeomorphic, as they are obtained using different gluing data. Yet, the identifications defining the \(\Delta\)-kaleidoscope depend only on the \emph{parity} of the components of the primitive normal vectors to the facets.
        
        In light of this, we conclude that, up to homeomorphism (and thus, in dimension \(3\), also up to diffeomorphism), the toric real loci of the examples in each family in \cref{thm:Oda3d} only depend on the \emph{parity} of the integer parameters \(a,b,c,d,e\).
        This is similar to the \(2\)-dimensional case in \cref{ex:hirzebruchlocus}, where the toric real locus of the Hirzebruch surface \(\cH_a\) depends only on the parity of \(a\).
    \end{remark}

    On the other hand, it can often be hard to immediately recognize the kaleidoscope as a familiar \(3\)-manifold, since in dimensions \(3\) and higher, there is no direct analog to the classification of surfaces. To deal with this, we made use of the \emph{Regina} software package \cite{Regina}. This program is able to process (generalized) triangulations of \(3\)-manifolds and compute many of their properties, namely various algebraic invariants. It also includes algorithms to simplify triangulations, perform connected sum decompositions, and, in many cases, identify within its database the \(3\)-manifold corresponding to a given triangulation.
    We elaborate on how we used \emph{Regina} in \cref{sec:sage_regina}.
    For now, we illustrate the results by looking at the orientable examples.

    \begin{remark}[Circle Bundles Over a Closed Surface]\label{rmk:sfs}
        Many of the toric symplectic 6-manifolds we consider are toric \(\CP{1}\)-bundles over toric symplectic 4-manifolds.
        For such a case, by an argument mentioned in the footnote of \cref{exs_fano_4}, the toric real locus is (the total space of) an \(S^1\)-bundle over a closed surface.
        These are examples of \emph{Seifert fibered spaces} (see e.g.\ \cite{Scott83}). From the theory of such spaces, one can extract in particular a classification of circle bundles over closed surfaces up to (unoriented) homeomorphism of their total spaces \(X\)
        (cf.\ \cite[Theorems 5, 6, 7, and 8]{Seifert33}, or see \cite{SeifertThrelfall} for a translation in English, \cite[Section 5.3 (Theorem 6), 5.4, and 6.2 (Theorem 2)]{Orlik}, and \cite{GeigesLange18}).

        Indeed, if \(X\) is orientable, then it can be described as
        \[
        (\mathrm{Oo}, g, b) \quad \text{or} \quad (\mathrm{On}, g, b).
        \]
        Here, \(\mathrm{O}\) indicates that the total space \(X\) is orientable, while \(\mathrm{o}\) or \(\mathrm{n}\) indicate that the base surface is, respectively, orientable or non-orientable. Moreover, \(g\) is the genus of the base surface (which must be at least \(1\) in the non-orientable case), and \(b\) is an obstruction constant taking values in \(\Z_{\ge 0}\).

        If \(X\) is non-orientable, then it can be described as\footnote{Many sources, including \cite{Regina,Orlik}, use different notation for the classes of bundles: \(\mathrm{Oo}\), \(\mathrm{On}\), \(\mathrm{No}\), \(\mathrm{Nn \, I}\), \(\mathrm{Nn \, II}\), \(\mathrm{Nn \, III}\) correspond, respectively, to \(\mathrm{o_1}\), \(\mathrm{n_2}\), \(\mathrm{o_2}\), \(\mathrm{n_1}\), \(\mathrm{n_3}\), \(\mathrm{n_4}\).} 
        \[
        (\mathrm{No}, g, b), \quad (\mathrm{Nn \, I}, g, b), \quad (\mathrm{Nn \, II}, g, b), \quad \text{or} \quad (\mathrm{Nn \, III}, g, b).
        \]
        
        Again, \(\mathrm{N}\) indicates that \(X\) is non-orientable, while \(\mathrm{o}\) or \(\mathrm{n}\) indicate whether the base surface is orientable, except that now the combination \(\mathrm{Nn}\) comprises three distinct classes of bundles. As before, \(g \ge 1\) denotes the genus of the base surface (we exclude the \(2\)-sphere, which cannot be the base of a non-orientable \(S^1\)-bundle), and \(b\) is an obstruction constant, which this time takes values in \(\Z/2\Z\). Finally, the class \(\mathrm{Nn \, II}\) only exists when \(g \ge 2\), and the class \(\mathrm{Nn \, III}\) only exists when \(g \ge 3\).

        Apart from the restrictions mentioned, every combination of values is realized by a bundle, and they are all non-homeomorphic, except that:
        \begin{itemize}
            \item the \(S^1\)-bundle over the torus of type \((\mathrm{No}, 1, b)\) is homeomorphic to the \(S^1\)-bundle over the Klein bottle of type \((\mathrm{Nn \, I}, 2, b)\);
            \item the \(S^1\)-bundle over the \(2\)-sphere of type \((\mathrm{Oo}, 0, 4)\) is homeomorphic to the \(S^1\)-bundle over the projective plane of type \((\mathrm{On}, 1, 1)\) (this is the lens space \(L(4,1)\), which is homeomorphic to the total space of the unit tangent bundle of \(\RP{2}\)).
        \end{itemize}

        Some notable examples are: \(X = \RP{3} \# \RP{3}\), which fibers as an \(S^1\)-bundle over \(\RP{2}\) of type \((\mathrm{On},1,0)\); \(X = S^3\), which through the Hopf fibration is an \(S^1\)-bundle over \(S^2\) of type \((\mathrm{Oo},0,1)\); and \(X = \RP{3}\), which fibers as an \(S^1\)-bundle over \(S^2\) of type \((\mathrm{Oo},0,2)\). The trivial bundle \(\Sigma \times S^1\) is of type \(\mathrm{Oo}\) when \(\Sigma\) is orientable, and of type \(\mathrm{Nn \, I}\) when \(\Sigma\) is non-orientable, with \(b = 0\) in both cases.
    \end{remark}

    \begin{example}
    \label{exs:orientable_miyake_oda_nagaya}
		Among the examples of Miyake–Oda (\cref{thm:Oda3d}), we find the following ones with an orientable toric real locus:
        
		\begin{itemize}
        
			\item Complex projective space \(\CP{3}\), whose toric real locus is the real projective space \(\RP{3}\).
            
			\item The \(\CP{1}\)-bundle \(\PP(\cO_{\CP{2}} \oplus \cO_{\CP{2}}(a)) \to \CP{2}\), for \(a\) odd. By \cref{rmk:dependsonlymod2}, it is enough to consider \(a = 1\). However, \(\PP(\cO_{\CP{2}} \oplus \cO_{\CP{2}}(1))\) is a blow-up of \(\CP{3}\) at a point (compare \cref{thm:Oda3d} and \cref{fig:kaleidoscope_cp3_blownup_at_point}). Therefore, we conclude that in this case the toric real locus is diffeomorphic to \(\RP{3} \# \RP{3}\).
            
			\item The \(\CP{1}\)-bundle \(X_{a,b,c}\) over the Hirzebruch surface \(\cH_a\), for \(a,b\) of the same parity and \(c\) even. Using \emph{Regina}, we determined that:
            \begin{itemize}
                \item if \(a,b,c\) are all even, then the toric real locus is the trivial \(S^1\)-bundle over the \(2\)-torus, i.e.\ the \(3\)-torus \(\T^3\). This can also be seen more directly using \cref{rmk:dependsonlymod2}, as when \(a=b=c=0\) this toric symplectic manifold is \((\CP{1})^3\).
                \item if \(a,b\) are odd and \(c\) is even, then the toric real locus is the (non-trivial) \(S^1\)-bundle over the Klein bottle of type \((\mathrm{On}, 2, 0)\).
            \end{itemize}
            
			\item The \(\CP{1}\)-bundle \(X_{a,b,c,d}\) over the \(\cH_a \# \overline{\CP{2}}\)
            (the base corresponds to a pentagon as in \cref{fig:pentagons}),
            when \(a,b\) have the same parity, \(c\) has the opposite parity, and \(d\) is even. Using \emph{Regina}, we determined that, in both cases, the toric real locus is the (non-trivial) \(S^1\)\nobreakdash-bundle over \(\#_3 \RP{2}\) of type \((\mathrm{On}, 3, 0)\).
            
			\item A \(\CP{1}\)-bundle \(X_{a,b,c,d,e}\) over one of the 
            \(\cH_a \# 2\overline{\CP{2}}\) (the base corresponds to one of the three types of hexagons in \cref{fig:hexagons}), for the following combinations of parameters:
            \begin{itemize}
                \item if the base is a hexagon of the first type, we must have \((a,b,c,d,e) \equiv (0,1,0,0,1)\) or \((1,0,1,0,1)\) mod \(2\);
                \item if the base is a hexagon of the second type, we must have \((a,b,c,d,e) \equiv (0,0,0,1,0)\) or \((1,1,0,0,0)\) mod \(2\);
                \item if the base is a hexagon of the third type, we must have \((a,b,c,d,e) \equiv (0,1,0,1,0)\) or \((1,0,1,0,0)\) mod \(2\).
            \end{itemize}
            In all these cases, using \emph{Regina}, we determined that the toric real locus is the (non-trivial) \(S^1\)-bundle over \(\#_4 \RP{2}\) of type \((\mathrm{On},4,0)\).
		\end{itemize}
		All of these orientable toric real loci live in toric \(\CP{1}\)-bundles
        or in \(\CP{3}\). In particular, none of the
        exceptional examples in \cref{thm:Oda3d} admit orientable toric real loci.
	\end{example}

    The examples with non-orientable toric real loci were harder to analyze since \emph{Regina}'s algorithm for performing connected sum decompositions can fail when applied to closed non-orientable \(3\)-manifolds that contain embedded two-sided projective planes, and the non-orientable summands are not always successfully recognized even when this decomposition succeeds. Nevertheless, we were still able to identify the toric real loci in some cases. Alternatively, we also considered the orientable double cover of each non-orientable toric real locus, which we were able to identify in all cases.

    \begin{example}\label{exs:non-orientable_miyake_oda_nagaya}
        Among the examples of Miyake--Oda (\cref{thm:Oda3d}), we find a non-orientable toric real locus in all cases not covered in \cref{exs:orientable_miyake_oda_nagaya}.
        \begin{itemize}
            \item The \(\CP{2}\)-bundle \(\PP(\cO_{\CP{1}} \oplus \cO_{\CP{1}}(b) \oplus \cO_{\CP{1}}(c)) \to \CP{1}\), for any choice of parities for \(b,c\). In all four cases, the toric real locus is the trivial bundle \(\RP{2} \times S^1\). Indeed, this is the unique \(\RP{2}\)-bundle over \(S^1\). The orientable double cover is \(S^2 \times S^1\).
            \item The \(\CP{1}\)-bundle \(\PP(\cO_{\CP{2}} \oplus \cO_{\CP{2}}(a)) \to \CP{2}\), for \(a\) even. By \cref{rmk:dependsonlymod2}, it is enough to consider the case \(a = 0\), which is the trivial bundle \(\CP{1} \times \CP{2}\). As such, the toric real locus is again \(\RP{2} \times S^1\), with orientable double cover \(S^2 \times S^1\).
            \item The \(\CP{1}\)-bundle \(X_{a,b,c}\) over the Hirzebruch surface \(\cH_a\), for \((a,b,c) \not\equiv (0,0,0)\) or \((1,1,0)\) mod \(2\).
            \begin{itemize}
                \item If \((a,b,c)=(1,0,0)\), then \(X_{a,b,c}\) is the trivial bundle \(\CP{1} \times \cH_1\). Therefore in this case the toric real locus is the trivial bundle \(K \times S^1\) over the Klein bottle.
                \item If \((a,b,c) \equiv (0,0,1)\), \((0,1,0)\) or \((0,1,1)\) mod \(2\), the toric real locus is also the trivial bundle \(K \times S^1\), which is homeomorphic to the \(S^1\)-bundle over the torus of type \((\mathrm{No},1,0)\). The orientable double cover is the \(3\)-torus \(\T^3 = \T^2 \times S^1\).
                \item In the remaining cases, when \((a,b,c) \equiv (1,0,1)\) or \((1,1,1)\) mod \(2\), the toric real locus is the non-trivial \(S^1\)-bundle over the Klein bottle of type \((\mathrm{Nn \, II}, 2, 0)\). The orientable double cover is the \(S^1\)-bundle over the Klein bottle of type \((\mathrm{On}, 2, 0)\).
            \end{itemize}
            \item The \(\CP{1}\)-bundle \(X_{a,b,c,d}\) over the \(\cH_a \# \overline{\CP{2}}\)
            (the base corresponds to a pentagon as in \cref{fig:pentagons}),
            when \((a,b,c,d) \not\equiv (0,0,1,0)\) or \((1,1,0,0)\) mod \(2\).
            \begin{itemize}
                \item If \((a,b,c,d) = (0,0,0,0)\) or \((1,0,0,0)\), \(X_{a,b,c,d}\) is the trivial product bundle \({\CP{1} \times (\cH_a \# \overline{\CP{2}})}\). Therefore, the toric real locus is the trivial bundle \(\#_3 \RP{2} \times S^1\), with orientable double cover \((\T^2 \# \T^2) \times S^1\).
                \item In the remaining cases, \emph{Regina} could not completely identify the toric real locus. Still, we were able to conclude that it is a non-trivial \(S^1\)-bundle over \(\#_3 \RP{2}\) of type \((\mathrm{Nn \, II},3,0)\) or \((\mathrm{Nn \, III},3,0)\), with orientable double cover the \(S^1\)-bundle over \(\#_4 \RP{2}\) of type \((\mathrm{On},4,0)\).
            \end{itemize}
            \item A \(\CP{1}\)-bundle \(X_{a,b,c,d,e}\) over one of the 
            \(\cH_a \# 2\overline{\CP{2}}\) (the base corresponds to one of the three types of hexagons in \cref{fig:hexagons}), for the combinations of parameters not covered in \cref{exs:orientable_miyake_oda_nagaya}. The situation is similar to the previous bullet point:
            \begin{itemize}
                \item If \((b,c,d,e) \cong (0,0,0,0)\) mod \(2\), then \(X_{a,b,c,d,e} = {\CP{1} \times (\cH_a \# 2\overline{\CP{2}})}\) is a trivial product bundle. This means that the real locus is the trivial bundle \(\#_4 \RP{2} \times S^1\), with orientable double cover \(\#_3 \T^2 \times S^1\).
                \item In all remaining cases, \emph{Regina} could not completely identify the toric real locus. Still, we were able to conclude that it is a non-trivial \(S^1\)-bundle over \(\#_4 \RP{2}\) of type \((\mathrm{Nn \, II},4,0)\) or \((\mathrm{Nn \, III},4,0)\), with orientable double cover the \(S^1\)-bundle over \(\#_6 \RP{2}\) of type \((\mathrm{On},6,0)\).
            \end{itemize}
            \item The exceptional example \(X_b\) labeled as [7-2] in \cite{Oda78}.
            \begin{itemize}
                \item When \(b\) is even, the toric real locus of \(X_b\) is diffeomorphic to the connected sum \((\RP{2} \times S^1) \# (\RP{2} \times S^1)\). The orientable double cover is \(\#_3(S^2 \times S^1)\).
                \item \emph{Regina}'s connected sum decomposition algorithm did not succeed on the toric real locus of \(X_b\) with \(b\) odd. Nevertheless, it has the same homology groups and fundamental group as in the even case, and the orientable double cover is again \(\#_3(S^2 \times S^1)\).
            \end{itemize}
            \item The exceptional example \(X_b\) labeled as [8-2] in \cite{Oda78}.
            \begin{itemize}
                \item When \(b\) is odd, the toric real locus of \(X_b\) is diffeomorphic to \((\RP{2} \times S^1) \# 3 \RP{3}\). The orientable double cover is \((S^2 \times S^1) \# 6 \RP{3}\).
                \item \emph{Regina}'s connected sum decomposition algorithm did not succeed on the toric real locus of \(X_b\) with \(b\) even. This closed non-orientable \(3\)-manifold has fundamental group
                \[\pi_1(X_b^\tau) = (\Z \times \Z/2\Z) * (\Z \times \Z/2\Z) * \Z/2\Z\]
                and the orientable double cover is \(3(S^2 \times S^1) \# 2 \RP{3}\), so it is not diffeomorphic to the toric real locus of the odd case.
            \end{itemize}
            \item The exceptional example \(X\) labeled as [8-10] in \cite{Oda78}. \emph{Regina}'s connected sum decomposition algorithm did not succeed either on this toric real locus. This closed non-orientable \(3\)-manifold has fundamental group
            \[\pi_1(X^\tau) = (\Z \times \Z/2\Z) * (\Z \times \Z/2\Z) * \Z/2\Z\]
            and the orientable double cover is \(3(S^2 \times S^1) \# 2 \RP{3}\), just as the even case in the previous bullet point.
            \item The exceptional example \(X_{b,c}\) labeled as [8-11] in \cite{Oda78}. In all four cases, the toric real locus is prime (i.e. it does not decompose as a non-trivial connected sum), but it contains a two-sided embedded projective plane. \emph{Regina} was not able to further identify the underlying closed non-orientable \(3\)-manifold. They all have the same homology groups, a fundamental group presented by five generators and five relations, and the orientable double cover \(Y \# 2(S^2 \times S^1)\), where \(Y\) is the \(S^1\)-bundle over the Klein bottle of type \((\mathrm{On},2,0)\).
        \end{itemize}
        \end{example}
        All the \(S^1\)-bundles appearing as toric real loci in
        these examples have \(b = 0\) (in the sense of \cref{rmk:sfs}).
        Moreover, we know that the real locus cannot be a sphere (which would have \(b = 1\)); cf.~\ref{prop:no_spheres}.
        By taking trivial products, the list of toric symplectic
        4-manifolds in \ref{coroll:4dim} gives rise to
        toric real loci of the form $\Sigma \times S^1$, where
        $\Sigma$ is any compact connected nonorientable surface
        or a torus, hence still with $b=0$.
        Now, $b=2$ occurs in the case of the toric real locus
        $\RR\PP^3$, which is a circle bundle over the sphere with
        Euler class $2$.
        Besides this, we are not aware of any other example
        with nonvanishing $b$.
        These observations lead to the following question:

        \begin{question}
        \label{question:circlebundles}
            What is the full list of circle bundles over closed
            surfaces that can occur as toric real loci
            for toric symplectic 6-manifolds?
            In particular, besides the $\RR \PP^3$ case,
            does any other circle bundle with nontrivial
            $b$ invariant occur?
        \end{question}

        Another interesting direction would be to investigate the geometries
        in the sense of Thurston that can occur for 
        \(3\)-dimensional toric real loci.
        Delaunay \cite{Delaunay3} proved that a \(3\)-dimensional toric real locus can never be hyperbolic. Above, we exhibited examples with the geometries \(S^3\) (spherical), \(S^2 \times \R\), \(E^3\) (euclidean), and \(H^2 \times \R\); cf.\ \cite[Table 4.1]{Scott83}.

        \begin{question}
        \label{question:3_geometries}
            If a $3$-dimensional toric real locus \(M^\tau\) admits one of Thurston's
            homogeneous geometries globally, must that geometry be one of
            the four
            \(S^3\), \(S^2 \times \R\), \(E^3\), and \(H^2 \times \R\)? Otherwise, if \(M^\tau\) is not geometric, must every piece in a geometric decomposition of \(M^\tau\) admit one of these geometries?
        \end{question}
        \vspace{-2.5mm}
        Since toric real loci are examples of small covers, an affirmative answer to the first part of this question in the orientable case follows by combining Delaunay's theorem and the work of Erokhovets \cite{Erokhovets22}. 


    \subsection{Our Use of \emph{Sage} and \emph{Regina}}
    \label{sec:sage_regina}
    In order to prepare an example to be analyzed in \emph{Regina}, we made use of \emph{Sage}, particularly the modules on fans, polyhedra, and triangulations. The process consists of the following steps:
    \begin{enumerate}
        \item Given the data of the fan of a smooth projective toric variety (as in \cref{thm:Oda3d}), construct a polytope \(\Delta\) with that normal fan. This was done by reusing the source code of the \emph{Sage} function \(\href{https://github.com/sagemath/sage/blob/develop/src/sage/geometry/fan.py#L2442}{\texttt{RationalPolyhedralFan.is\_polytopal}}\) (due essentially to Jean-Philippe Labbé).
        \item Triangulate the polytope \(\Delta\) into tetrahedra, since \emph{Regina} works only with manifold triangulations. This is natively supported in \emph{Sage}.
        \item Recreate this triangulation of \(\Delta\) as a \emph{Regina} triangulation.
        \item Create \(2^3 = 8\) copies of the triangulated \(\Delta\) and glue them together as in \cref{rmk:def_kaleidoscope} to obtain a \emph{Regina} triangulation of the \(\Delta\)-kaleidoscope.
    \end{enumerate}
    This generalizes in a straightforward way to dimensions \(n > 3\), although in these cases, there are fewer methods available in \emph{Regina} to analyze the results. Indeed, our code also addresses the case \(n = 4\), as this was necessary in \cref{exs_fano_4}.

    Given a \emph{Regina} triangulation of a three-dimensional \(\Delta\)-kaleidoscope, we analyze it as follows:
    \begin{enumerate}
        \item Simplify the triangulation to reduce the number of tetrahedra.
        \item Attempt to perform a connected sum decomposition. This always succeeds in the orientable case, but may fail for a non-orientable \(3\)-manifold in the presence of two-sided embedded projective planes.
        \item In many cases, \emph{Regina} now successfully recognizes the summands in the connected sum decomposition. Otherwise, we still have access to data such as the fundamental group and the homology groups.
        \item In the non-orientable case, we apply the same steps to the orientable double cover.
    \end{enumerate}
    Note that the simplification in step (1) is not fully deterministic. As such, in some cases (namely for the toric real locus of a manifold labeled as [8-11] in \cite{Oda78}), the connected sum decomposition algorithm will sometimes succeed and sometimes fail, depending on the simplified triangulation obtained in step (1). In this case, we may conclude both that the manifold in question is prime (from the successful runs of the algorithm), and that it contains embedded two-sided projective planes (as this is the only case in which the algorithm can fail).
	

\appendix

\section{The Smooth Projective Toric Threefolds of Miyake--Oda}\label{appendix}

    By Oda's classification of smooth compact (complex) toric surfaces~\cite{Oda78}
    (cf.\ \cref{sec:case_n=2}), these manifolds are always projective.
    Even though such a classification is impossible in higher dimensions
    (cf.\ \cref{sec:case_n=3}), one can still obtain a finite list of minimal models by suitably restricting their complexity. For instance, Miyake, Oda, and Nagaya \cite[Theorem~9.6]{Oda78} described all minimal smooth compact toric threefolds of Picard number up to \(5\). Among these, they obtained both projective and non-projective examples\footnote{\cite[p. 80]{Oda78} points out for each example whether or not it is projective, with the exception of the example labeled [8-11]. This remaining example is also projective, as can be checked by finding an appropriate polytope as in \cref{subsec:MiyakeOdaNagaya}, but see also \cite[Section 4]{FujinoSato25}.}.

Since a smooth compact toric variety is projective if and only if
its fan is the normal fan\footnote{We use the convention that the normal fan of a polytope is obtained using the \emph{outward}-pointing primitive normal vectors to the facets, which is opposite to the most common usage in algebraic geometry. Of course, though, multiplying every ray of the fan by \(-1\) does not change the isomorphism type of the toric variety.} of a unimodular polytope,
    we are interested in the projective case.
    A smooth projective toric variety gives rise to a family of toric symplectic manifolds, with the symplectic structure being determined by a choice of polytope with the given normal fan. In the toric symplectic perspective, the Picard number corresponds to the second Betti number \(b_2\), which coincides with the number of facets of the moment polytope
    minus the polytope dimension. Therefore, the examples of \cite{Oda78} have \(3\)-dimensional polytopes with up to \(8\) facets.

We now list those examples of smooth projective toric threefolds,
giving their defining data in a way particularly suitable to the applications described in \cref{subsec:MiyakeOdaNagaya}. The reader interested in the toric symplectic setting should keep in mind the discussion in \cref{rmk:Miyake_Oda_not_symplectic}.
    
	\begin{theorem}[{\cite[Theorem 9.6]{Oda78}, \cite[Theorem 1.34]{Oda88}}]
    \label{thm:Oda3d}
		Let \(X\) be a \(3\)-dimensional smooth projective toric variety with Picard number \(\rho \leq 5\) that is not a toric blow-up of any other smooth compact toric threefold. Throughout, let \(a,b,c,d,e\) denote arbitrary integers, with \(a \geq 0\).
		\begin{enumerate}
			\item If \(\rho = 1\), then \(X\) is isomorphic to the complex projective space \(\CP{3}\). This is the toric variety associated to the fan with rays \[n_0 = (-1,0,0), \, n_1 = (0,-1,0), \, n_2 = (0,0,-1), \, n_3 = (1,1,1)\] and three-dimensional cones \[\pair{n_0,n_1,n_2}, \, \pair{n_0,n_1,n_3}, \, \pair{n_0,n_2,n_3}, \, \pair{n_1,n_2,n_3}.\]
			\item If \(\rho = 2\), then \(X\) is isomorphic to one of the following two alternatives:
			\begin{enumerate}
				\item a \(\CP{2}\)-bundle over \(\CP{1}\) of the form \(\PP(\cO_{\CP{1}} \oplus \cO_{\CP{1}}(b) \oplus \cO_{\CP{1}}(c))\). This is the toric variety associated to the fan with rays \[n_0 = (-1,0,0), \, n_1 = (0,-1,0), \, n_2 = (0,0,-1), \, n_3 = (0,1,1), \, n_4 = (1,-b,-c)\] and three-dimensional cones \[\pair{n_0,n_1,n_2}, \, \pair{n_0,n_1,n_3}, \, \pair{n_0,n_2,n_3}, \, \pair{n_1,n_2,n_4}, \, \pair{n_1,n_3,n_4}, \, \pair{n_2,n_3,n_4}.\] 
				\item a \(\CP{1}\)-bundle over \(\CP{2}\) of the form \(\PP(\cO_{\CP{2}} \oplus \cO_{\CP{2}}(a))\). This is the toric variety associated to the fan with rays \[n_0 = (-1,0,0), \, n_1 = (0,-1,0), \, n_2 = (0,0,-1), \, n_3 = (0,0,1), \, n_4 = (1,1,a)\] and three-dimensional cones \[\pair{n_0,n_1,n_2}, \, \pair{n_0,n_1,n_3}, \, \pair{n_0,n_2,n_4}, \, \pair{n_0,n_3,n_4}, \, \pair{n_1,n_2,n_4}, \, \pair{n_1,n_3,n_4}.\] 
			\end{enumerate}
			\item If \(\rho = 3\), then \(X\) is isomorphic to a \(\CP{1}\)-bundle over a Hirzebruch surface \(\cH_a\). This is a toric variety associated to a fan with rays
			\begin{gather*}
				n_0 = (0,0,-1), \, n_1 = (0,0,1), \, n_2 = (-1,0,0),\\ n_3 = (0,-1,0), \, n_4 = (1,a,b), \, n_5 = (0,1,c)
			\end{gather*} and three-dimensional cones 
			\begin{gather*}
				\pair{n_0,n_2,n_3}, \, \pair{n_0,n_3,n_4}, \, \pair{n_0,n_4,n_5}, \, \pair{n_0,n_5,n_2},\\ \pair{n_1,n_2,n_3}, \, \pair{n_1,n_3,n_4}, \, \pair{n_1,n_4,n_5}, \, \pair{n_1,n_5,n_2}.
			\end{gather*}
			\item If \(\rho = 4\), then \(X\) is isomorphic to one of the following two alternatives:
			\begin{enumerate}
				\item a \(\CP{1}\)-bundle over a smooth compact toric surface of Picard number \(3\) (cf.\ end of \cref{sec:case_n=2} and \cref{fig:pentagons}). This is a toric variety associated to a fan with rays
				\begin{gather*}
					n_0 = (0,0,-1), \, n_1 = (0,0,1), \, n_2 = (-1,0,0), \, n_3 = (0,-1,0), \\ 
					n_4 = (1,a,b), \, n_5 = (1,a+1,c), \, n_6 = (0,1,d)
				\end{gather*} and three-dimensional cones 
				\begin{gather*}
					\pair{n_0,n_2,n_3}, \, \pair{n_0,n_3,n_4}, \, \pair{n_0,n_4,n_5}, \, \pair{n_0,n_5,n_6}, \, \pair{n_0,n_6,n_2},\\ 
					\pair{n_1,n_2,n_3}, \, \pair{n_1,n_3,n_4}, \, \pair{n_1,n_4,n_5}, \, \pair{n_1,n_5,n_6}, \, \pair{n_1,n_6,n_2}.
				\end{gather*}
				\item the toric variety associated to the fan with rays
				\begin{gather*}
					n_0 = (-1,0,0), \, n_1 = (0,-1,0), \, n_2 = (0,0,-1), \, n_3 = (0,1,b), \\
					n_4 = (0,0,1), \, n_5 = (0,-1,1), \, n_6 = (1,-2,1)
				\end{gather*} and three-dimensional cones 
				\begin{gather*}
					\pair{n_0,n_1,n_2}, \, \pair{n_0,n_2,n_3}, \, \pair{n_0,n_3,n_4}, \, \pair{n_0,n_4,n_5}, \, \pair{n_0,n_5,n_6},\\
					\pair{n_0,n_6,n_1}, \, \pair{n_1,n_2,n_6}, \, \pair{n_2,n_3,n_6}, \, \pair{n_3,n_4,n_6}, \, \pair{n_4,n_5,n_6}.
				\end{gather*}
				This is the variety labeled as [7-2] in \cite{Oda78} and as \(3^2 4^3 6^2\) in \cite{Oda88}.
			\end{enumerate}
			\item If \(\rho = 5\), then \(X\) is isomorphic to one of the following alternatives:
			\begin{enumerate}
				\item a \(\CP{1}\)-bundle over a smooth compact toric surface of Picard number \(4\). Up to isomorphism, we have seen that these surfaces come in three families (cf.\ end of \cref{sec:case_n=2}). Let \[m_0 = (-1,0), \, m_1 = (0,-1), \, m_2, m_3, m_4, m_5\]
				be the normal vectors to the edges of the associated moment hexagon (cf.\ \cref{fig:hexagons}). Then, \(X\) is a toric variety associated to a fan with rays
				\begin{gather*}
					n_0 = (0,0,-1), \, n_1 = (0,0,1), \, n_2 = (-1,0,0), \, n_3 = (0,-1,0),\\ 
					n_4 = (m_2,b), \, n_5 = (m_3,c), \, n_6 = (m_4,d), \, n_7 = (m_5,e)
				\end{gather*} and three-dimensional cones 
				\begin{gather*}
					\pair{n_0,n_2,n_3}, \, \pair{n_0,n_3,n_4}, \, \pair{n_0,n_4,n_5}, \, \pair{n_0,n_5,n_6}, \, \pair{n_0,n_6,n_7}, \, \pair{n_0,n_7,n_2}\\ 
					\pair{n_1,n_2,n_3}, \, \pair{n_1,n_3,n_4}, \, \pair{n_1,n_4,n_5}, \, \pair{n_1,n_5,n_6}, \, \pair{n_1,n_6,n_7}, \, \pair{n_1,n_7,n_2}.
				\end{gather*}
				
				\item the toric variety associated to the fan with rays
				\begin{gather*}
					n_0 = (-1,0,0), \, n_1 = (0,2,1), \, n_2 = (0,1,0), \, n_3 = (0,0,-1), \\
					n_4 = (0,-1,-b), \, n_5 = (0,0,1), \, n_6 = (0,1,1), \, n_7 = (1,3,2)
				\end{gather*} and three-dimensional cones
				\begin{gather*}
					\pair{n_0,n_1,n_2}, \, \pair{n_0,n_2,n_3}, \, \pair{n_0,n_3,n_4}, \, \pair{n_0,n_4,n_5}, \, \pair{n_0,n_5,n_6}, \, \pair{n_0,n_1,n_7},\\
					\pair{n_1,n_2,n_7}, \, \pair{n_2,n_3,n_7}, \, \pair{n_3,n_4,n_7}, \, \pair{n_4,n_5,n_7}, \, \pair{n_5,n_6,n_7}, \, \pair{n_6,n_0,n_7}.
				\end{gather*}
				This is the variety labeled as [8-2] in \cite{Oda78} and as \(3^2 4^4 7^2\) in \cite{Oda88}.
				\item the toric variety associated to the fan with rays
				\begin{gather*}
					n_0 = (-1,0,0), \, n_1 = (0,-1,0), \, n_2 = (0,0,-1), \, n_3 = (-1,0,1), \\
					n_4 = (0,0,1), \, n_5 = (0,1,2), \, n_6 = (1,1,1), \, n_7 = (-1,1,2)
				\end{gather*} and three-dimensional cones
				\begin{gather*}
					\pair{n_0,n_1,n_2}, \, \pair{n_0,n_1,n_3}, \, \pair{n_0,n_2,n_6}, \, \pair{n_0,n_3,n_7}, \, \pair{n_0,n_6,n_7}, \, \pair{n_1,n_2,n_6},\\
					\pair{n_1,n_3,n_7}, \, \pair{n_1,n_4,n_6}, \, \pair{n_1,n_4,n_7}, \, \pair{n_4,n_5,n_6}, \, \pair{n_4,n_5,n_7}, \, \pair{n_5,n_6,n_7}.
				\end{gather*}
				This is the variety labeled as [8-10] in \cite{Oda78} and as \(3^3 4^1 5^1 6^3\) in \cite{Oda88}.
                
				\item the toric variety associated to the fan with rays
				\begin{gather*}
					n_0 = (-1,0,0), \, n_1 = (0,1,0), \, n_2 = (0,0,-1), \, n_3 = (-1,-1,b), \\
					n_4 = (0,0,1), \, n_5 = (0,1,1), \, n_6 = (1,2,1), \, n_7 = (0,-1,c)
				\end{gather*} and three-dimensional cones
				\begin{gather*}
					\pair{n_0,n_1,n_2}, \, \pair{n_0,n_2,n_3}, \, \pair{n_0,n_3,n_4}, \, \pair{n_0,n_4,n_5}, \, \pair{n_0,n_5,n_6}, \, \pair{n_0,n_6,n_1},\\
					\pair{n_1,n_2,n_6}, \, \pair{n_2,n_7,n_6}, \, \pair{n_7,n_4,n_6}, \, \pair{n_4,n_5,n_6}, \, \pair{n_2,n_3,n_7}, \, \pair{n_3,n_4,n_7}.
				\end{gather*}
				This is the variety labeled as [8-11] in \cite{Oda78} and as \(3^2 4^2 5^2 6^2 (ii)\) in \cite{Oda88}.
			\end{enumerate}
		\end{enumerate}
	\end{theorem}

    Even though the majority of these examples are either
    \(\CP{3}\), toric \(\CP{1}\)-bundles over a toric surface,
    or toric \(\CP{2}\)-bundles over \(\CP{1}\), there are
    some ``exceptional'' examples in {(4-b)}, {(5-b)}, {(5-c)}, and {(5-d)}.

    \begin{remark}
        We do not claim that the list of \cref{thm:Oda3d} is irredundant. For example, \(\CP{2} \times \CP{1}\) appears in both \emph{(2-a)} and \emph{(2-b)}. Moreover, there are certain special values of the parameters \(a,b,c,d,e\) for which these examples can actually be equivariantly blown down (in the complex/algebraic category). This can be explicitly worked out in each case as an application of \cite[Remark after Corollary 1.32]{Oda88}.
        \begin{itemize}
            \item \emph{(2-a)} can be blown down when \((b,c) = (-1,0),(0,-1),(1,1)\).
            \item \emph{(2-b)} can be blown down when \(a = 1\).
            \item \emph{(3)} can be blown down in the following cases: \[c = \pm 1; \quad a = 0 \text{ and } b = \pm 1; \quad a = 1 \text{ and } b=c.\]
            \item \emph{(4-a)} can be blown down in the following cases: \[b=c; \quad c = b+d; \quad a = 0 \text{ and } c=d.\]
            \item \emph{(4-b)} can be blown down when \(b = -1,0\).
            \item \emph{(5-a)} can be blown down for certain combinations of \(a,b,c,d,e\).
            \item \emph{(5-b)} can be blown down when \(b=0,1\).
            \item \emph{(5-d)} can be blown down when \(b=c\).
        \end{itemize}
    \end{remark}

\bibliographystyle{plain}
\bibliography{references}

\end{document}